\documentclass[a4paper]{amsart}
\usepackage{xcolor} 
\usepackage{graphicx}
\usepackage[toc]{appendix} 
\usepackage{amsmath, amssymb}
\usepackage{amsfonts} 
\usepackage{amsthm} 
\usepackage{systeme}
\usepackage[top=3cm, bottom=3cm,inner=1.5cm, outer=1.5cm]{geometry}	
\usepackage{color}   
\usepackage{hyperref}
\hypersetup{
    colorlinks=true, 
    linktoc=all,    
    linkcolor=black,  
    citecolor = black,
}

\newtheorem{theorem}{Theorem}[section] 
\newtheorem{lemma}[theorem]{Lemma}

\newtheorem{prop}[theorem]{Proposition}

\numberwithin{equation}{section}

\newcommand{\abs}[1]{\lvert #1 \rvert}
\newcommand{\sumstack}[1]{{\substack{#1}}}
\newcommand{\nsmod}[1]{\text{ }\mathrm{mod}\text{ $#1$}} 
\newcommand{\fq}[0]{\mathbb{F}_q}
\DeclareMathOperator{\Tr}{Tr}
\DeclareMathOperator*{\GLtwo}{GL_2}
\DeclareMathOperator*{\SLtwo}{SL_2}
\DeclareMathOperator*{\vol}{Vol}
\theoremstyle{definition}

\title{On the counting function of cubic function fields}
\author{Victor Ahlquist}
\address{Department of Mathematical Sciences, Chalmers University of Technology and the University \newline
	\rule[0ex]{0ex}{0ex}\hspace{8pt} of Gothenburg, SE-412 96 Gothenburg, Sweden}
\email{vicahlqu@chalmers.se} 
\date{}
\begin{document}

\begin{abstract}
    We study the counting function of cubic function fields. Specifically, we derive an asymptotic formula for this counting function including a secondary term and an error term of order $\mathcal{O}\big(X^{2/3+\epsilon}\big)$, which matches the best-known result, due to Bhargava, Taniguchi and Thorne, over $\mathbb{Q}$. Furthermore, we obtain estimates for the refined counting function, where one specifies the splitting behaviour of finitely many primes. Also in this case, our error term matches what is known for number fields. However, in the function field setting, the secondary term becomes more difficult to write down explicitly.

    Our proof uses geometry of numbers methods, which are especially effective for function fields. In particular, we obtain an exact formula for the number of orbits of cubic forms with fixed absolute discriminant. Moreover, by studying the one-level density of a family of Artin $L$-functions associated to these cubic fields, we prove an unconditional lower bound on the error term in the estimate for the refined counting function. This generalises a conditional result over $\mathbb{Q}$, due to Cho, Fiorilli, Lee and Södergren.
\end{abstract}
\maketitle

\section{Introduction}
Let $Y $ be a positive real number and let $$\mathcal{F}^{\pm}_{\mathbb{Q}}(Y) = \{L: [L:\mathbb{Q}] = 3, \,\, 0\leq \pm \mathrm{Disc}(L) \leq Y\},$$ i.e. $\mathcal{F}^{\pm}_{\mathbb{Q}}(Y)$ consists of cubic number fields with discriminant in the range specified above. We remark that we only include one field from each isomorphism class. The counting function of this set was studied by Davenport and Heilbronn \cite{DH}, who proved the existence of constants $C_1^\pm$ such that
\begin{equation}\label{littleocubic}
    \#\mathcal{F}^{\pm}_{\mathbb{Q}}(Y) = C_1^{\pm}Y + o(Y).
\end{equation}
The error term above was improved to $\mathcal{O}\big(Y^{7/8+\epsilon}\big)$ by Belabas, Bhargava and Pomerance \cite{BBP}. Furthermore, an analogue of \eqref{littleocubic} was proven by Bhargava, Shankar and Wang \cite{BSW}, with $\mathbb{Q}$ replaced by an, essentially, arbitrary global base field.

In contrast to the case of counting quadratic fields, where one quite easily proves the existence of a main term of size $Y$ together with an error term of quality $\mathcal{O}\big(Y^{1/2+\epsilon}\big)$, the counting function $\#\mathcal{F}^{\pm}_{\mathbb{Q}}(Y)$ has a secondary term, of significant size, in addition to the main term. Indeed, the asymptotic formula 
\begin{equation}\label{cubiccountwerr}
    \#\mathcal{F}^{\pm}_{\mathbb{Q}}(Y) = C_1^{\pm}Y + C_2^{\pm}Y^{5/6} + \mathcal{O}\big(Y^{\omega+\epsilon}\big),
\end{equation}
with $\omega < 5/6$, was conjectured by Roberts \cite{Roberts}, and then proven independently by Bhargava, Shankar and Tsimermann \cite{BST} and by Taniguchi and Thorne \cite{TT}. The currently best-known result is $\omega = 2/3$, proven by Bhargava, Taniguchi and Thorne \cite{BTT}. 

Fix primes $p_1,...,p_n$ and splitting types $S_1,..., S_n$, with each $S_i$ representing either the completely split, partially split, inert, partially ramified or totally ramified case. Then, one may consider the refined counting function $\mathcal{F}^{\pm}_{\mathbb{Q}}(p_1,...,p_n; S_1,..., S_n;Y)$ counting the number of cubic number fields with the same condition on the discriminant as before, but also requiring that each $p_i$ splits according to the splitting type $S_i.$ The currently best-known estimate for this counting function was also proven in \cite{BTT}; specifically, one has
\begin{equation}\label{numbfieldsplitres}
    \#\mathcal{F}^{\pm}_{\mathbb{Q}}(p_1,...,p_n; S_1,...,S_n;Y) = C_{1,p_1,...,p_n,S_1,...,S_n}^{\pm}Y + C_{2,p_1,...,p_n,S_1,...,S_n}^{\pm}Y^{5/6} + \mathcal{O}\big(Y^{\theta+\epsilon}(p_1\cdots p_n)^{\omega}\big),
\end{equation}
with $\theta = \omega  = 2/3$.

The proofs in \cite{TT} and \cite{BTT} make use of the analytic theory of certain Shintani zeta functions. In particular, the appearance of the secondary term in the counting function can be explained by these zeta functions having poles, not only at $1$, but also at $5/6$. In \cite{BST}, the counting function is instead studied using geometry of numbers methods, and the secondary term appears as a consequence of secondary terms in certain theorems for lattice point counting. Another proof of a closely related result, based on studying Heegner points, is due to Hough, see the discussion after \cite[Theorem 2.1]{Hough2}. See also \cite{Chang} for a different perspective on the second-order term.

One may also consider the problem of finding a lower bound for the error term in \eqref{numbfieldsplitres}. Such a result was obtained, conditional on the Riemann Hypothesis for Dedekind zeta functions associated with cubic number fields, by Cho, Fiorilli, Lee and Södergren \cite{CFLS}. Specifically, they proved that under these assumptions, one cannot have $\omega + \theta < 1/2$ in \eqref{numbfieldsplitres}. The proof of this result is based on an investigation of the low-lying zeros of an associated family of Artin $L$-functions using the so-called one-level density.

Instead of studying cubic number fields, one may study cubic function fields. To be precise, let $q$ be a fixed prime power, and consider degree three extensions of $\fq(T)$, where $\fq$ is the finite field with $q$ elements. We will always assume that $2,3 \nmid q$. Then, from the Riemann-Hurwitz theorem, we know that the (absolute) discriminant of a cubic function field is an even power of $q$. For such a number $Y$, we define $$\mathcal{F}(Y) = \{L: [L:\fq(T)] = 3, \,\,  \mathrm{Disc}(L) = Y\},$$ where we have suppressed the dependence on $q$. Note that the condition governing the size of the discriminant is now an equality, as is usual when studying function fields.

The existence of a main term in the counting function $\#\mathcal{F}(Y)$ was shown in 1988 by Datskovsky and Wright \cite{DW}. This counting function was later studied by Zhao \cite{Zhao}, who isolated a secondary term using algebro-geometric methods, with a claimed\footnote{See Section \ref{relationch}.} bound for the error term of order $o\big(Y^{5/6}\big)$. Our main goal is to improve the bound on the error term to the same quality as obtained over $\mathbb{Q}$, while also allowing for splitting conditions at finitely many primes. Furthermore, we do so using geometry of numbers methods similar to those in \cite{BST}, through essentially elementary arguments. We should, however, remark that our arguments are in a sense equivalent to the approach using Shintani zeta functions, but the non-archimedean geometry of $\fq(T)$ allows us to bypass the explicit use of the theory of Shintani zeta functions.

Our first main result is the following theorem, establishing an asymptotic formula for the number of cubic extensions of $\fq(T)$.
\begin{theorem}\label{allfieldthm}
    The number of cubic function field extensions of $\fq(T)$ with discriminant $Y = q^{M}$, where $M \geq 0$ is even, is equal to
    \begin{equation}\label{countallflds}
        \frac{(q^2-1)(q^3-1)}{q^4(q-1)}Y - \frac{q^2-1}{q}C_2(M)Y^{5/6} + \mathcal{O}\left(Y^{2/3+\epsilon}\right),
    \end{equation}
    where\footnote{Note that the coefficient of the secondary term depends on the congruence class of $M$ modulo $3$. 
    }
    \begin{equation*}
        C_2(M) = \begin{cases}
    q^{-2}(q+1)     ,\,\, &M \equiv 0 \nsmod{3},\\
    q^{-4/3}      ,\,\, &M \equiv 1 \nsmod{3},\\
    q^{-5/3}(q+1) ,\,\, &M \equiv 2 \nsmod{3}.
    \end{cases}
    \end{equation*}
\end{theorem}
We remark that a cubic extension of $\fq(T)$ is either Galois, or non-Galois with its Galois closure having Galois group $S_3$. The contribution from Galois fields is $\ll Y^{1/2+\epsilon}$, see \cite{WrightC3}, and hence \eqref{countallflds} is also an asymptotic formula for the counting function of non-Galois cubic fields.

Our second main result refines the counting function by also letting us specify the splitting type at finitely many primes. When we control the splitting type at a single prime, the theorem takes the form below;  see Theorem \ref{allfieldSplitCondThm} for the general statement.
\begin{theorem}\label{corollaryforonelev}
   Let $\mathcal{F}_{P,S}(Y)$ denote the set of cubic function field extensions of $\fq(T)$ of discriminant $Y$, where the prime polynomial $P\in \fq[T]$ splits according to the splitting type $S$. Then, we have that
   \begin{equation}\label{corspliteq}
        \#\mathcal{F}_{P,S}(Y) = C_{1,P,S}Y + C_{2,P,S}Y^{5/6} + \mathcal{O}\left(Y^{2/3+\epsilon} \abs{P}^{2/3}\right),
    \end{equation}
    for certain constants $C_{i,P,S}$. Specifically, 
    \begin{equation}\label{firstsplitcoeffeq}
        C_{1,P,S} = \frac{(q^2-1)(q^3-1)}{q^4(q-1)}\left(1+\abs{P}^{-1}+\abs{P}^{-2}\right)^{-1}c_S,
    \end{equation}
    where $c_{S} = 1/6, 1/2, 1/3, 1/\abs{P}, 1/\abs{P}^2$ depending on if $S= (111), (21), (3), (1^21)$ or $(1^3)$ respectively. The constant $C_{2,P,S}\ll 1$ is given in Theorem \ref{allfieldSplitCondThm}.
\end{theorem}

In addition to bounding the error terms from above, we also employ the methods from \cite{CFLS} to obtain an omega-result, bounding the error from below. As in \cite{CFLS}, we use the Riemann Hypothesis for Dedekind zeta functions, but as the Riemann Hypothesis is a theorem over function fields, our result is completely unconditional. More precisely, we prove the following theorem.
\begin{theorem}\label{omegaresultthm}
    Suppose that
    \begin{equation}\label{omegarel}
        \#\mathcal{F}_{P,S}(Y) = C_{1,P,S}Y + C_{2,P,S}Y^{5/6} + \mathcal{O}\left(Y^{\theta+\epsilon} \abs{P}^\omega\right)
    \end{equation}
    holds for some $0\leq \theta < 5/6$ and $\omega \geq 0$, for all splitting types $S$, with $C_{i,P,S}$ as in Theorem \ref{allfieldSplitCondThm}. Then, $\omega+\theta \geq 1/2$.
\end{theorem}
We remark that over $\mathbb{Q}$, there is some numerical evidence suggesting that the relation \eqref{omegarel} holds with $\theta = 1/2$ and $\omega = 0$, see \cite[Appendix A]{CFLS}. This would make Theorem \ref{omegaresultthm} sharp, while the error term in Theorem \ref{corollaryforonelev} would be quite far from being sharp.

In addition to bounding the error term in \eqref{corspliteq} from above and below, our methods also provide another perspective on the secondary term in these counting functions. We will see that the appearance of these terms can be explained by the Fourier transforms of certain functions defined on $\big(\mathbb{F}_{q^n}\big)^4$. Specifically, the main term comes from evaluating these Fourier transforms at zero, while the secondary term comes from values of these Fourier transforms at nonzero points that are, in a sense, highly singular, see e.g. \eqref{sectermfour}.

\subsection{Outline}
We begin, in Section \ref{prelch}, by recalling algebraic notions concerning function fields as well as results about completions of function fields. Then, in Section \ref{ringch}, we specialise to the setting of cubic function fields. Here, we present several results allowing us to count cubic rings, instead of cubic fields. Isomorphism classes of these cubic rings are in bijection with orbits, under a certain group action, of binary cubic forms, through the Levi--Delone--Faddeev correspondence. This allows us to recast the problem of counting fields as a lattice point counting problem.

In Section \ref{latticech} we obtain our first main results. We use geometry of numbers methods to count the number of orbits of integral binary cubic forms having discriminant equal to some $q$-power $q^\ell := X$. In fact, we are able to obtain an exact formula for the number of orbits of such forms in Theorem \ref{explicitformcount}. The proof is based on the method of "thickening the cusp" from \cite{BST} together with an explicit description of a fundamental domain of the function field analogue of $\GLtwo(\mathbb{Z}) \setminus \GLtwo(\mathbb{R})$. Furthermore, we extend the slicing method from \cite[Section 6]{BST}, which is what allows us to count reducible forms. Specifically, when necessary, we slice over two of the coefficients of these reducible binary cubic forms.

Over $\mathbb{R}$, when counting integral forms, one separates the forms with positive discriminant from those with negative discriminant, see e.g. \cite[Theorem 5]{BST}. This is equivalent to separating the forms based on whether the corresponding ring, when viewed as a cubic ring over $\mathbb{R}$, is isomorphic to $\mathbb{R}^3$, or to $\mathbb{R}\oplus \mathbb{C}$. In $\fq(T)$, one can choose a distinguished prime $P_\infty$, and a corresponding completion $K_\infty$ of $\fq(T)$. Hence, over $\fq(T)$, we instead separate the forms depending on which local cubic extension of $K_\infty$, say $\sigma$, the corresponding cubic rings over $K_\infty$ are isomorphic to. Let $V(\fq[T])$ denote the space of integral binary cubic forms. We then prove the following exact formula for $ N(V(\fq[T])^\sigma; X)$, i.e. the number of integral orbits under the standard group action with discriminant equal to $X$, corresponding to $\sigma$ over $K_\infty$.
\begin{theorem}\label{explicitformcount}
        For $X = q^\ell $ larger than some absolute constant, depending only on $q$, we have that
    \begin{equation}\label{explformcounteq}
        N(V(\fq[T])^\sigma; X) =   \frac{q}{(q-1)\#\mathrm{Aut}(\sigma)}X-\frac{C_2(\ell)}{\#\mathrm{Aut}(\sigma)(q-1)}X^{5/6}+\frac{q(q+1)I_0^\sigma}{(q-1)\vol(B)\#\mathrm{Aut}(\sigma)}X,
    \end{equation}
    where $C_2(\ell)$ is given explicitly in Proposition \ref{secevalprop}, $I_0^\sigma$ is defined in \eqref{redtermint} and $B$ is defined above \eqref{Babove}. Furthermore, we obtain an asymptotic formula for the number of irreducible forms, $N(V(\fq[T])^{\mathrm{irr },\sigma}; X)$, by replacing the last term\footnote{This essentially means that the third term in \eqref{explformcounteq} corresponds to the reducible binary cubic forms.} above with $\mathcal{O}(X^{3/4+\epsilon})$.
\end{theorem}

We finally start counting cubic fields in Section \ref{fieldchapter}, where we employ the discriminant-reducing sieve from \cite[Section 9]{BST}. More precisely, we first use the inclusion-exclusion principle to reduce the problem of finding maximal forms to that of finding forms which are nonmaximal at some squarefree polynomial $F$. The number of such forms can then be computed using the discriminant-reducing sieve, which further reduces the problem to counting elements of $V(\fq[T])$ with weight $\omega_F(x)$, the number of distinct roots of the form $x$ modulo $F$.

Next, as in \cite{BST}, we use the inverse Fourier transform to study this counting problem. From \cite{TTOrb}, we know that when $F=P$, the Fourier transform of $\omega_P$ evaluated at some binary cubic form $y$ modulo $P$ is given explicitly by
\begin{equation}\label{sectermfour}
    \widehat{\omega}_P(y)  = \begin{cases}
        1 + \abs{P}^{-1}, &\text{ if $y = 0$ modulo $P$},\\
        \abs{P}^{-1}, &\text{ if $y \neq 0$ has a triple root modulo $P$,}\\ 
        0, &\text{ otherwise}.
    \end{cases}
\end{equation}
Using this result, combined with Möbius inversion, we can essentially reduce our counting problem to studying binary cubic forms whose last coordinate is divisible by $P$. After estimating the number of such forms, we arrive at Theorem \ref{allfieldthm}. We remark that it seems very difficult to improve the error term in this theorem. For example, in \eqref{nondegencontrib} a single term from the sum over $y$ already results in a contribution of size roughly $X^{2/3}$ to the counting function of cubic fields.

In section \ref{onelevch}, we introduce the one-level density, which is a tool for studying the distribution of the zeros close to the real axis, of a family of Artin $L$-functions associated with cubic function fields. Then, following \cite{CFLS}, we use this one-level density to bound the error term, in the counting function for cubic function fields with splitting conditions, from below. Specifically, we prove that in case \eqref{omegarel} holds with $\theta+\omega < 1/2$, then the secondary term in the counting function $\#\mathcal{F}_{P,S}(Y)$ contributes to the one-level density in a way that makes it grow faster than a power of $Y$. However, by using the Riemann Hypothesis, one sees that the one-level density is bounded by a constant multiple of $\log_q(Y)$. This contradiction thus yields Theorem \ref{omegaresultthm}.

Finally, in Section \ref{splitcondchapt} we find an asymptotic formula for the number of cubic function fields with finitely many splitting conditions. This is done by extending the sieve from Section \ref{fieldchapter} to handle splitting conditions. In particular, we use the sieve for counting nowhere ramified extensions from \cite[Section 9]{BST}. Moreover, we utilise \cite[Theorem 11]{TTOrb}, originally proven in \cite{Mori}, where exact formulas for the Fourier transforms of the indicator functions of the various splitting types modulo $P$ are provided. Using these exact formulas, one can proceed with the estimation similarly to Section \ref{fieldchapter}; however, the process of bounding the error term is more involved.

Compared to the results over $\mathbb{Q}$, the secondary term becomes slightly more complicated, because of the dependence on the congruence class of the exponent $M$ modulo three. In particular, we cannot write down a simple product for the secondary term. Instead, the secondary term is equal to a certain sum, which doesn't split multiplicatively. Our computations for fields with splitting conditions result in Theorem \ref{allfieldSplitCondThm}, with an error term of the same size as the analogous result over $\mathbb{Q}$, from \cite{BTT}.

\subsection{Relation to the work of Zhao and to the work of Kural}\label{relationch}
In an earlier version of this paper, we claimed that Zhao had proven the existence of the secondary term in $\#\mathcal{F}(Y)$. However, we were recently made aware that Zhao's work contains a gap. In fact, recently and independently of our work, Kural \cite{Kural} closed this gap through a significant extension of Zhao's methods, obtaining an asymptotic formula for $\#\mathcal{F}(Y)$ involving a main and secondary term, together with an error term of order $\mathcal{O}(Y^{3/4+\epsilon})$. See \cite[Section 1.4]{Kural} for a description of the gap in Zhao's thesis. 
\subsection{Acknowledgements}
We want to thank Anders Södergren for suggesting this problem and for his support and careful proofreading during the writing process. Furthermore, we thank Michael Kural for correspondence regarding our works.
\subsection{Conventions}
Given some set $X$, and functions $f: X\to \mathbb{C}$ and $g: X\to \mathbb{R}_{\geq 0}$, we write $f = \mathcal{O}(g)$, if there is some constant $C$ such that $\abs{f(x)}\leq Cg(x)$ for all $x\in X$. Alternatively, we write this as $f\ll g$. Similarly, we write $f = h + \mathcal{O}(g)$ if $f-h\ll g$. Furthermore, if $f\ll g$ and $g\ll f$, then we write $f\asymp g$. The constant $C$ above is called the implied constant, and its dependence on a collection of variables $\mathcal{D}$ is indicated by writing $\mathcal{O}_\mathcal{D}$ or $\ll_\mathcal{D}$. We will let all of the implied constants depend on the variables $q$ and $\epsilon$, without indicating this. We may also somewhat erroneously refer to constants only depending on $q$ as absolute constants. 

In Section \ref{onelevch}, we will make use of the Fourier transform $\widehat{\psi}$ of a Schwartz function $\psi$. We use the convention that
\begin{equation*}
    \widehat{\psi}(u) = \int_{\mathbb{R}}\psi(x)e^{-2\pi i x u }dx.
\end{equation*}
In addition to the Fourier transform above, we will also make use of the finite Fourier transform. Specifically, if $G$ is a finite abelian group, we define the finite Fourier transform of a function $f: G\to \mathbb{C}$ by
\begin{equation*}
    \widehat{f}(\chi) = \abs{G}^{-1}\sum_{a\in G}f(a)\chi(a),
\end{equation*}
where $\chi$ is a character of $G$. We may recover $f$ through Fourier inversion, using that
\begin{equation*}
    f(a) = \sum_{\chi \in \widehat{G}}\widehat{f}(\chi) \overline{\chi}(a),
\end{equation*}
where $\widehat{G}$ denotes the character group of $G$.

\section{Preliminaries}\label{prelch}
We briefly recall some notions that will be useful in the study of cubic function fields. Readers familiar with the algebraic theory of function fields may skip this section. We refer to \cite[Ch. 5, 7-8]{Rosen} and \cite[Ch. 1-2]{Frankenhuijsen} for the details.

We fix a prime power $q$ not divisible by $2$ or $3$ and consider a finite, separable extension $L$ of $K := \fq(T)$. We allow the possibility $L=\fq(T)$. A prime in $L$ is a discrete valuation $\nu$ on $L$ such that $L$ is the fraction field of the ring $\mathcal{O}_\nu := \{x\in L: \nu(x)\geq 0\}$. Sometimes, we call the prime ideal $\mathfrak{P} = \{x\in \mathcal{O}_\nu: \nu(x) \geq 1\}$ a prime of $L$. If $S$ is a nonempty set of primes in $L$, then we define the $S$-integers, $\mathcal{O}_S := \{x\in L: \nu(x) \geq 0, \nu\notin S\}$. The degree of a prime $\mathfrak{P}$ associated with $\nu$ is defined to be the dimension of  $\mathcal{O}_\nu/\mathfrak{P}$ over the field of constants, which in our case is $\fq$.

To study the primes of a cubic field $L$, we will use the fact that the primes of $K$ are very well-understood. Indeed, a prime of $\fq(T)$ is either the valuation corresponding to a prime polynomial $P$, or the so-called prime at infinity, corresponding to the valuation $\nu_\infty$ defined on $\fq[T]$ as $\nu_\infty(f) = -\deg(f)$ and then extended additively to all of $K^*$. We often denote the prime at infinity $P_\infty$ and the associated absolute value with $\abs{\cdot}_\infty$, or simply $\abs{\cdot}$. For polynomials $f$, we have $\abs{f} = q^{\deg(f)}$. Note that a uniformiser of $P_\infty$ is $1/T$, as it has valuation equal to one. We sometimes write $\pi = \pi_\infty = 1/T$. 

The completion of $K$ with respect to $\nu_\infty$ is denoted $K_\infty$ and contains all Laurent-series of the form
\begin{equation*}
    \sum_{n=-\infty}^\infty a_n(T^{-1})^n,
\end{equation*}
with $a_n\in \fq$, and all but finitely many $a_n$ with negative index $n$ equal to zero. We write $\mathcal{O}_\infty$ for the completion of $\mathcal{O}_{\nu_\infty}$. We have that $K_\infty^* \simeq \mathcal{O}_\infty^* \times \pi^\mathbb{Z}$, so that any $x\in K_\infty^*$ has the form $u(T^{-1})^n$ for some integer $n$ and $u\in \mathcal{O}_\infty^*$. We call $u$ the unit part of $x$.

Every prime $\mathfrak{P}$ in a cubic extension $L$ of $K$ lies over a prime $P$ in $K$, in the sense that the restriction of the valuation $\nu_{\mathfrak{P}}$ to $K$ is a multiple of $\nu_P$. We define the inertial degree $f(\mathfrak{P}/P)$ of $\mathfrak{P}$ over $P$ as the dimension of $\mathcal{O}_\mathfrak{P}/\mathfrak{P}$ over $\mathcal{O}_P/P$. If $\nu_P$ is the valuation of a prime in $K$, then over $L$, $[L:K]\nu_P$ splits into a sum $\nu_S =e_1\nu_{\mathfrak{P}_1}+...+e_r\nu_{\mathfrak{P}_r}$ of valuations $\nu_{\mathfrak{P}}$ on $L$, with $\nu_S$ agreeing with $[L:K]\nu_P$ on $K$. We call $e_i =: e(\mathfrak{P}_i/P)$ the ramification degree of $\mathfrak{P}_i$ over $P$, and they satisfy the formula
\begin{equation*}
    \sum_{i=1}^r f(\mathfrak{P}_i/P)e(\mathfrak{P}_i/P) = [L:K] = 3.
\end{equation*}
If $e(\mathfrak{P}_i/P) = [L:K]$ we say that $P$ is totally ramified and if all $e(\mathfrak{P}_i/P) = 1$ we say that $P$ is unramified.

For primes in $K$ obtained by localising $R$ at a prime polynomial $P\in R$, we can describe this splitting behaviour in another way. If $\mathfrak{P}_1,...,\mathfrak{P}_r$ are the primes lying over (the prime obtained from) $P$ with ramification degrees $e_1,...,e_r$, then if we let $\mathcal{O}_L$ be the integral closure of $R$ in $L$, we have that
\begin{equation*}
    P\mathcal{O}_L = \mathcal{P}_1^{e_1}...\mathcal{P}_r^{e_r},
\end{equation*}
where $\mathcal{P}_i = \mathfrak{P}_i \cap \mathcal{O}_L$.

We now briefly discuss the discriminant of a cubic function field $L$. Let $R=\fq[T] = \mathcal{O}_S$, with $S= \{\nu_\infty\}$. We can then define the $R$-semilocal discriminant by considering the integral closure $\mathcal{O}_L$ of $R$ in $K$. Letting $\alpha_1, \alpha_2,\alpha_3$ be an $R$-basis of $\mathcal{O}_L$, we define the (absolute) semilocal discriminant as 
\begin{equation*}
    \abs{\det(\Tr_{L/K}(\alpha_i\alpha_j)_{i,j})}_\infty.
\end{equation*}
If $P_\infty$ is unramified in $L$, this agrees with the (absolute) global discriminant. Otherwise, in case $P_\infty$ is ramified, the global discriminant is the $R$-semilocal discriminant multiplied by $q^a$, with $a=1$, unless $P_\infty$ is totally ramified, in which case $a=2$. See \cite[Lemma 7.10, Cor. 7.2]{Rosen} for a proof.

 The extension $L$ is said to be geometric if its field of constants is equal to $\fq$. As $L$ is of prime degree, the only non-geometric extension is $L= \mathbb{F}_{q^3}(T)$, and we may therefore assume that all extensions are geometric, for the purpose of counting them. Evidently, the norm of the discriminant is a nonnegative integral power of $q$, and the Riemann--Hurwitz formula implies that all such norms are in fact even integer powers of $q$. 

Our tools for counting the number of cubic field extensions will require us to choose a distinguished prime. Over $\mathbb{Q}$, this prime is naturally taken to be the Archimedean prime, but in the function field case, we instead use $P_\infty$ for this purpose. We remark that $P_\infty$ may be replaced by any prime of degree one in $K$ without changing our arguments.

With $P_\infty$ as our distinguished prime, the local field $K_\infty$ is given the analogous role of $\mathbb{R}$ in the number field setting. However, as the completion of the global field $K$ at a non-Archimedean prime, $K_\infty$ enjoys several useful properties. The most important one for our purposes is Hensel's lemma, see \cite[II.4.6]{Neukirch}.
\begin{prop}[Hensel's lemma]\label{HenselLemma}
    If a monic polynomial $f(x)\in \mathcal{O}_\infty[x]$ admits a factorisation $f(x) \equiv \overline{g(x)}\,\overline{h(x)} \nsmod{P_\infty}$, with $\overline{g(x)}$ and $\overline{h(x)} \in \big(\mathcal{O}_\infty/P_\infty\big)[x]$ relatively prime, then this factorisation lifts to $\mathcal{O}_\infty[x]$. More precisely, there are $g,h\in \mathcal{O}_\infty[x]$ with $\deg(g) = \deg(\overline{g})$ and $\deg(h) = \deg(\overline{h})$ whose reductions modulo $P_\infty$ agree with $\overline{g}$ and $\overline{h}$ respectively, such that
    \begin{equation*}
        f(x) = g(x)h(x)
    \end{equation*}
    in $\mathcal{O}_\infty[x]$.
\end{prop}
In particular, this allows us to lift roots modulo $P_\infty$ of multiplicity one to unique roots in $\mathcal{O}_\infty$. 

\section{Cubic Rings and Maximality}\label{ringch}
Rather than counting cubic fields directly, we will make use of a correspondence theorem allowing us to count certain cubic rings instead, as in \cite{BST} and \cite{BSW}. We begin with a few definitions.

First, we let $R=\fq[T]$, which is a PID. A cubic ring over $R$ is a commutative ring $A$, with unity, such that $A$ is a free module of rank $3$ over $R$. If $A$ is a cubic ring over $R$, then we will simply call $A$ a cubic ring. We say that $A$ is maximal in case it is not strictly contained in any other cubic ring.

By taking fraction fields, we have the following first result connecting cubic field extensions of $K$ with maximal cubic orders, i.e. maximal cubic rings without zero-divisors. All objects are assumed to lie in some fixed algebraic closure of $K$. 
\begin{lemma}
    Maximal cubic orders are in bijection with cubic field extensions of $K = \fq(T)$.
\end{lemma}
One can define the discriminant of a cubic ring $A$ as the determinant of the trace form on $A$. We denote the norm of this discriminant with $\mathrm{Disc}(R)$. When $A$ is a maximal cubic order, one sees immediately that $\mathrm{Disc}(A)$ is equal to the absolute semilocal discriminant of the corresponding cubic field extension, with respect to $S = \{P_\infty\}$. 

The lemma above implies that in order to count cubic fields (up to isomorphism), we may instead count maximal cubic orders (up to isomorphism). Next, we relate cubic rings to so-called binary cubic forms. 

A binary cubic form $f$ with coefficients in some PID $B$ is an expression $f(x,y) = ax^3+bx^2y+cxy^2+dy^3$ with $a,b,c,d\in B$. We write $V(B)$ for the space of such forms, and we often identify $V(B)$ with $B^4$. The discriminant of $f\in V(B)$ is defined in the usual way so that $\mathrm{Disc}(f) = b^2c^2-4ac^3-4b^3d-27a^2d^2+18abcd$. On $V(B)$ we can define a $\GLtwo(B)$ action by letting $(gf)(x,y) = \det(g)^{-1}f((x,y)g)$. One checks that $\mathrm{Disc}(gf) = \det(g)^2\mathrm{Disc}(f)$. For this action, we have the following result, see \cite[Section 2]{BST} and \cite[Theorem 5]{BSW}.
\begin{theorem}[The Levi--Delone--Faddeev correpondence]\label{DeloneFaddeev}
    Let $B$ be a PID. Then there is a discriminant-preserving bijection between the $\GLtwo(B)$-orbits of binary cubic forms with coefficients in $B$ and the set of isomorphism classes of cubic rings over $B$. Explicitly, the isomorphism is given by mapping $(a,b,c,d) $ to the ring $ \langle 1,\omega,\theta\rangle,$
    with multiplication laws 
    \begin{equation*}
    \begin{cases}
        &\omega\theta = -ad \\
        &\omega^2 = -ac-b\omega+a\theta,\\
        &\theta^2 = -bd-d\omega+c\theta.
    \end{cases}    
    \end{equation*}
Furthermore, the automorphism group of the ring $\langle 1,\omega,\theta \rangle$ over $B$ is isomorphic to the stabiliser of the form $(a,b,c,d)\in V(B)$ in $\GLtwo(B)$. The ring $A$ corresponding to a form $f$ is an order precisely when $f$ is irreducible.
\end{theorem}
We will primarily apply this theorem when $B = R$. However, another important case is when $B = K_\infty$, with $K_\infty$ being the local field obtained by completing $K$ at the prime $P_\infty$.

By using the Levi--Delone--Faddeev correspondence, we may translate the study of cubic rings to the study of cubic forms. However, we are only interested in maximal cubic rings, and we therefore need a criterion to check rings for maximality.

We say that a cubic ring $A$ over $R$ is nonmaximal if it is not maximal. Let $P\in R$ be a prime polynomial and $T = R\setminus RP$. We say that the cubic ring $A$ over $R$ is nonmaximal at $P$ if the cubic ring over $T^{-1}R$ obtained by localising at $P$ is nonmaximal. One checks that nonmaximality is a local condition, so that $A$ is nonmaximal if and only if it is nonmaximal at all primes $P$. We then have the following result, which is the function field version of \cite[Lemma 13]{BST} and is proved in the same way.
\begin{lemma}\label{nonmaxringlemma}
Let $P\in R$ be a prime. Then a cubic ring (over $R$) is nonmaximal at $P$ if and only if there is a basis  $1,\omega,\theta$ of $A$, with $\omega\theta\in R$, such that one of
\begin{equation*}
\begin{split}
    &R + R\frac{\omega}{P}+R\theta,\\
\end{split}
\end{equation*}
and
\begin{equation*}
\begin{split}
    &R + R\frac{\omega}{P}+R\frac{\theta}{P}
\end{split}
\end{equation*}
forms a ring.
\end{lemma}
We remark that a basis $1,\omega,\theta$ such that $\omega\theta \in R$ is called a normal basis. By using the explicit description in Theorem \ref{DeloneFaddeev}, we see that being nonmaximal at $P$ can be checked by considering the coefficients of an associated form modulo $P^2$.

We end this section with a brief discussion concerning subrings and overrings. Let $P\in R$ be a prime polynomial. We say that a cubic ring $A'$ is a $P$-overring of the cubic ring $A$ if $A\subseteq A'$, $\abs{A'/A} = \abs{R/PR}$ and $T_Q^{-1}A' = T_Q^{-1}A$ for all primes $P\neq Q\in R$, where $T_Q = R\setminus RQ$. We say that $A'$ is a $P$-subring of $A$ if $A$ is a $P$-overring of $A'$. We then have the following analogues of \cite[Propositions 15-16]{BST} with the same proof.
\begin{lemma}\label{subringformlemma}
   The number of $P$-subrings of a cubic ring over $R$ is equal to $\omega_P(f)$, the number of zeros of the corresponding binary cubic form $f$, reduced modulo $P$, in $\mathbb{P}^1(R/P R)$.
\end{lemma}
\begin{lemma}\label{overringlemma}
    The number of $P$-overrings of a cubic ring $A$, contained in $A\otimes K$ is equal to the number of double zeros $\alpha \in \mathbb{P}^1(R/PR)$ of the corresponding binary cubic form $f \nsmod{P}$, such that $f(\alpha') \equiv 0 \mod P^2$ for all $\alpha'\equiv \alpha$ modulo $P$.
\end{lemma}
These lemmas will be employed when we sieve cubic rings for maximality in Section \ref{fieldchapter}, using the discriminant reducing method from \cite[Chapter 9]{BST}.

\section{Counting cubic forms}\label{latticech}
We now show how to count orbits of binary cubic forms over $R$. Our method is analogous to \cite[Section 5]{BST}. We let $v \in V(K_\infty)$ have nonzero discriminant. By Theorem \ref{DeloneFaddeev} with $B=K_\infty$, $v$ lies in an orbit corresponding to some cubic Étale extension of $K_\infty$. As $2,3\nmid q$, there are only finitely many such Étale extensions. Specifically, one sees this by first noting that there is a unique unramified extension of every degree and then bounding the number of totally ramified extensions using Krasner's lemma. We will provide an explicit description of these cubic Étale extensions in Section \ref{autch}.

We write $(\mathrm{Aut}(\sigma))_\sigma$ for the collection of automorphism groups over $K_\infty$, where $\sigma$ runs over the isomorphism classes of Étale extensions of $K_\infty$ of degree three. For each $\sigma$, we define $ V(K_\infty)^\sigma$ to be the subset of $V(K_\infty)$ whose associated cubic ring has isomorphism class $\sigma$. We also fix representatives $v_\sigma \in V(K_\infty)^\sigma$. For $v\in V(R)^\sigma\subseteq V(K_\infty)^\sigma$, the absolute value of $\mathrm{Disc}(v)$ provides the absolute semilocal discriminant, with respect to $S=\{P_\infty\}$, of the corresponding cubic ring. Indeed, this follows from our previous discussion. Furthermore, by considering the ramification of the Étale extension corresponding to $\sigma$, we can obtain the part of the absolute discriminant corresponding to $P_\infty$.

Let $\mathcal{F}$ be a fundamental domain for the action of $\GLtwo(R) \backslash \GLtwo(K_\infty)$. We fix some $v_\sigma \in V(K_\infty)$. The first key observation is that for any $v\in V(R)$ in the same $\GLtwo(K_\infty)$ orbit as $v_\sigma$, we have that
\begin{equation}\label{multiplic}
     \#\{g\in \mathcal{F}: gv_\sigma \in  \GLtwo(R)v\} = \#\mathrm{Aut}(\sigma)/\#\mathrm{Stab}_{\GLtwo(R)}(v),
\end{equation}
cf. \cite[Theorem 9]{BSW}. Here, $\GLtwo(R)v$ should be interpreted as a set, not a multiset. To see this relation, note that if $h\in \GLtwo(R)$ and $gv_\sigma = hv$, then $\Tilde{g} = h^{-1}g$ is in the same $\GLtwo(R)$ orbit of $\GLtwo(K_\infty)$ as $g$ and $\Tilde{g}v_\sigma = v$. So the left-hand side above counts the number of $\GLtwo(R)$-orbits containing a $\Tilde{g}$ with $\Tilde{g}v_\sigma = v$. If we write $v = g_0v_\sigma$ with $g_0\in \GLtwo(K_\infty)$, then we can write $\Tilde{g} = g'g_0$. The number of choices for $g'$ is given by the number of $\GLtwo(R)$-orbits in $\mathrm{Stab}_{\GLtwo(K_\infty)}(v)$, which is exactly $\#\mathrm{Aut}(\sigma)/\#\mathrm{Stab}_{\GLtwo(R)}(v)$.

We let $V(R)^\sigma$ be the set of forms lying in the orbit corresponding to $\sigma$, and we let $N(V(R)^\sigma; X)$ denote the number of $\GLtwo(R)$-orbits of elements in $V(R)^\sigma$ whose discriminant has absolute value equal to $X$, with respect to $\abs{\cdot} = \abs{\cdot}_\infty$. We count each such element with a factor $(\#\mathrm{Stab}_{\GLtwo(R)}(v))^{-1}$ in the left-hand side below. Then, the relation \eqref{multiplic} shows that
\begin{equation*}
   N(V(R)^\sigma, X) = \frac{1}{\#\mathrm{Aut}(\sigma)}\#\{\mathcal{F}_Xv_\sigma \cap V(R)\},
\end{equation*}
where $\mathcal{F}_X$ is a fundamental domain for the left-action of $\GLtwo(R)$ on $$\GLtwo(K_\infty)_X = \big\{g\in \GLtwo(K_\infty): \abs{\det(g)}^2\abs{\mathrm{Disc}(v_\sigma)}=X\big\},$$ and $\mathcal{F}_Xv_\sigma$ is a multiset. Here, we made use of the fact that $\abs{\det(g)} = 1$ for any $g\in \GLtwo(R)$.
\subsection{The fundamental domain}
To understand $\mathcal{F}_Xv_\sigma$, we begin by explicitly describing $\mathcal{F}_X$. Over $\mathbb{Q}$, the fundamental domain is related to the classical fundamental domain for the action of $\SLtwo(\mathbb{R})$ on the upper-half plane. Over $\fq(T)$ we obtain something similar; see \cite[Ch. II.1]{Serre} for a more abstract treatment.

We note that $X$ must be some nonnegative power of $q$. Let us write $q^{2m}$ for $X/\abs{\mathrm{Disc}(v_\sigma)}$. Recall that $v_\sigma$ is a representative of the $\GLtwo(K_\infty)$-orbit $V(K_\infty)^\sigma$. We now pick another representative, which we, with some abuse of notation also denote by $v_\sigma$, whose discriminant differs by a factor $q^2$ in absolute value from the original representative. Then, by possibly changing which of the two representatives $v_\sigma$ we are considering, depending on the congruence class of $\log_q(X)$ modulo $4$, we may assume that $m =: 2m'$ is even. We then have the following result. 
\begin{prop}\label{fundprop}
     Let $\alpha \in \fq$ be a fixed non-square. A $(q-1)$-fold fundamental domain for the action of $\GLtwo(R)$ on $\{g\in \GLtwo(K_\infty): \abs{\det(g)} = q^{2m'}\}$ is given by the set of all matrices of the form $\lambda nak$, with
     \begin{equation*}
     \begin{cases}
         \lambda\in K_\infty: \lambda = \lambda' \pi^{m'}, \,\,\abs{\lambda'}= 1, \lambda' \equiv 1 \,\, \mathrm{mod}\,\, P_\infty.
        \\ n=n(f) = \begin{pmatrix}
             1 & f \\0 & 1
         \end{pmatrix}: f\in K_\infty, \abs{f}\leq q^{-1},
         \\ a = a(t) = \begin{pmatrix}
             t & 0 \\0 & t^{-1}
         \end{pmatrix}:  \,\, t = \pi^{\ell}t_0\in K_\infty,\,\, \ell \leq 0, \abs{t_0} = 1,  t_0 \equiv 1 \,\, \mathrm{mod}\,\, P_\infty,

         \\ k = \begin{pmatrix}
             a & c\alpha \\c & a
         \end{pmatrix}: \abs{a^2-c^2\alpha} = 1,
     \end{cases}
    \end{equation*}
    except for finitely many exceptions, for every fixed $k$ and $\lambda$, and that matrices with $\abs{t} = 1$ are represented $(q-1)(q+1)$ times. 
\end{prop}
We denote the subgroup of matrices of the form $\lambda I $ with $\lambda$ as above by $\Lambda_X$.
\begin{proof}
    We give a proof that is quite similar to the number field case. First, we claim that $\GLtwo(K_\infty)$ acts on the set-difference $K_\infty(\sqrt{\alpha})\setminus K_\infty$. Indeed, for $\tau\in K_\infty(\sqrt{\alpha})\setminus K_\infty$ one defines
    \begin{equation*}
        g(\tau) = \begin{pmatrix}
            a &b \\c &d
        \end{pmatrix}\tau = \frac{a\tau+b}{c\tau+d}.
    \end{equation*}
    Any element in $\tau \in K_\infty(\sqrt{\alpha})$ has the form $x+y\sqrt{\alpha}$, $a,b\in K_\infty$. Let us write $\Re(\tau) = x$ and $\Im(\tau) =y$. One checks that
    \begin{equation*}
        \Im(g(\tau)) = \frac{\det(g)\Im(\tau)}{\mathcal{N}(c\tau+d)},
    \end{equation*}
    where $\mathcal{N}(x+y\sqrt{\alpha}) = x^2-\alpha y^2$ is the Galois norm.
    
    We claim that the action above is transitive.  Indeed, by acting with a matrix of the form $n(f)$ above, without the condition on $\abs{f}$, $\tau$ maps to $\tau+f$, $f\in K_\infty$. This means that we can map $\tau$ to an element $\tau'$ with arbitrary $\Re(\tau')$, while keeping $\Im(\tau) = \Im(\tau')$ fixed. We are therefore done if we can show that given any $y\in K_\infty$, there is a $g\in \GLtwo(K_\infty)$ with, say, $\Im(g\sqrt{\alpha}) = y$.

    Recall that $\Im(g\sqrt{\alpha}) = \det(g)/(d^2-c^2\alpha)$. Let us write $y = u(T^{-1})^{m_0}$, where $u\in \mathcal{O}_\infty^*$ is a unit. We can further decompose $u$ into $u=u'u_0$, where $u'\in \fq$ and $u_0 \in 1+\pi_\infty\mathcal{O}_\infty$. In particular $u_0 = v^2$ for some $v\in \mathcal{O}_\infty^*$, by Hensel's lemma; see Lemma \ref{HenselLemma}. We also write $m_0=2n_0+\epsilon$, with $\epsilon$ being zero or one. Then, we have that
    \begin{equation*}
    \tau' : =\begin{pmatrix}
        u' & 0 \\ 0& 1
    \end{pmatrix}\begin{pmatrix}
        v\pi^{n_0} & 0 \\ 0& v^{-1}\pi^{-n_0}
    \end{pmatrix}\begin{pmatrix}
        \pi^{\epsilon} & 0 \\ 0& 1
    \end{pmatrix}(\sqrt{\alpha})
    \end{equation*}
    satisfies $\Im(\tau') = y$. This concludes the proof of transitivity.

    A short calculation shows that the stabiliser of $\sqrt{\alpha}$ is the group $\Tilde{K}$ of matrices of the form
    \begin{equation*}
        \begin{pmatrix}
            a & c\alpha \\ c & a
        \end{pmatrix}.
    \end{equation*}
    Setting $c = 0$ we see that this group contains the subgroup of matrices $\lambda I$, with $\lambda \in K_\infty^*$. Furthermore, as $\alpha$ is a nonsquare, the determinant $a^2-c^2\alpha$ has an even valuation. The stabiliser is thus the direct product of the group $\Lambda$ of diagonal matrices $\lambda I$, with $\lambda$ having unit part congruent to one modulo $P_\infty$, and the subgroup $K$ of $\Tilde{K}$ consisting of matrices with determinant in $\fq^*$.

    By the orbit-stabiliser theorem, we have that
    \begin{equation*}
        \GLtwo(K_\infty)/\Lambda K \simeq K(\sqrt{\alpha})\setminus K_\infty,
    \end{equation*}
    with the map being given by acting on $\sqrt{\alpha}$. Next, we find a fundamental domain for the left-action of $\GLtwo(R)$ on the right-hand side above.

    First, by translating with an element of the form $n(f)\in \GLtwo(R)$, with $f\in R$, we can map any element $\tau$ to an element of the form $a+b\sqrt{\alpha}$, with $\abs{a} < 1$. We claim that we can, in fact, map any element to an element of the form $x+y\sqrt{\alpha}$ with $\abs{x} < 1$ and $\abs{y}\geq 1$. Indeed, if $x+y\sqrt{\alpha}$ has $\abs{x} < 1$, and $\abs{y} < 1$, then
    \begin{equation*}
        \begin{pmatrix}
            0 & 1 \\ 1 & 0
        \end{pmatrix}(x+y\sqrt{\alpha}) = \frac{1}{x+y\sqrt{\alpha}} = \frac{x-y\sqrt{\alpha}}{x^2-y^2\alpha},
    \end{equation*}
    where
    \begin{equation*}
        \frac{\abs{y}}{\abs{x^2-y^2\alpha}} \geq q^{2}\abs{y}.
    \end{equation*}
    Repeating this procedure, we eventually obtain an element of the desired form. By scaling with $\mathrm{diag}(c,1)\in \GLtwo(R)$, with $c\in \fq$, we can see that every element can be mapped to one of the form $x+y\sqrt{\alpha}$ with not only $\abs{x} < 1$, $\abs{y}\geq 1$, but also with the unit part of $y$ being congruent to one. We denote the set of all such elements by $D$.

    We now check for elements of the form above lying in the same orbit.
    Suppose that $\tau_1$ and $\tau_2$ have the stated form, and assume that $\tau_2 = g(\tau_1)$ for some $g\in \GLtwo(R)$ of the form
    \begin{equation*}
    \begin{pmatrix}
        a & b \\c & d
    \end{pmatrix}.
    \end{equation*}
    Without loss of generality $\abs{\Im(\tau_1)} \leq \abs{\Im(\tau_2)}$. As $\tau_2 = g(\tau_1)$, we have that $\mathcal{N}(c\tau_1+d)\Im(\tau_2)= \det(g)\Im(\tau_1)$. Note that $\abs{\det(g)} = 1$.  Let us first assume that $c = 0$. Then, $a,d\in \fq^*$, and $g(\tau_1) = ad^{-1}\tau_1+d^{-1}b$. Now, $\abs{b}\geq 1$, or $b= 0$ which means that we must have that $b=0$ for both $\tau_1$ and $\tau_2$ to be in $D$. However, this further implies that $\tau_2 = ad^{-1}\tau_1$, and the restriction on the congruence class of $\Im(\tau_i)$ then shows that $\tau_1 = \tau_2$.

    Let us turn to the case when $c\neq 0$. Then, $$\abs{\mathcal{N}(c\tau_1+d)} = \abs{(d+c\Re(\tau_1))^2-c^2\Im(\tau_1)^2\alpha} = (\max\{\abs{d+c\Re(\tau_1)}, \abs{c\Im(\tau_1)}\})^2.$$ As $\abs{\Im(\tau_1)} \leq \abs{\Im(\tau_2)}$, we must have that $\abs{\Im(\tau_1)}, \abs{c},\abs{d} \leq 1 $, i.e. that $\abs{\Im(\tau_1)}=1, c\in \fq^*$, and $d\in \fq$. Here we used that $c\neq 0$. This also implies that $\abs{\Im(\tau_1)} =\abs{\Im(\tau_2)}$ and that $\abs{\mathcal{N}(c\tau_1+d)}$ =1. 

    As the situation is now symmetric with respect to $\tau_1$ and $\tau_2$ we can repeat the procedure. A matrix taking $\tau_2$ to $\tau_1$ is $g' = \det(g)g^{-1}$, i.e.
    \begin{equation*}
        \begin{pmatrix}
            d & -b \\ -c & a
        \end{pmatrix}.
    \end{equation*}
    Proceeding as before, we find that $a\in \fq$. This also implies that $b\in \fq$ so that in fact $g\in \GLtwo(\fq)$. Note that this already shows that $\tau_1$ can only lie in the orbit of finitely many other elements from $D$.

    We now study the relation $\tau_2 = g(\tau_1)$ more closely. Let us assume that $d= 0$. Then,
    \begin{equation*}
        \tau_2 = \frac{a}{c}+\frac{b}{c\tau_1} = \frac{a}{c}+\frac{b\overline{\tau_1}}{c\mathcal{N}(\tau_1)},
    \end{equation*}
    where $\overline{\tau_1}$ is the Galois conjugate of $\tau_1$. We compute $\abs{\mathcal{N}(\tau_1)} = \abs{\Re(\tau_1)^2-\Im(\tau_1)^2\alpha} =\abs{\Im(\tau_1)^2\alpha}= 1$. If $a\neq 0$ this shows that $\abs{\Re(\tau_2)}=1$ contradicting $\tau_2\in D$, so we must in fact have $a = 0$. As $a=d=0$, we must have that $b\neq 0$ and $\tau_2 = (bc^{-1})\tau_1^{-1}$.  This is an involution composed with a multiplication map. By the condition on $\Im(\tau)$ for $\tau \in D$, we see that $\tau_1$ is related to precisely one distinct $\tau_2$ in $D$, unless $\tau_1^{-1}$ is an $\fq$-multiple of $\tau_1$. This happens for only finitely many $\tau_1$, so that we can disregard these cases.

    Finally, we treat the case when $d\neq 0$. In this case, one sees that $a\neq 0$ by switching the roles of $\tau_1$ and $\tau_2$. The relation between the $\tau_i$ is then
    \begin{equation}\label{tau2eq}
        \tau_2 = \frac{a\tau_1+b}{c\tau_1+b} = \frac{1}{\mathcal{N}(c\tau_1+d)}\big(bd +ac \mathcal{N}(\tau_1) + (ad+bc)\Re(\tau_1) +\det(g)\Im(\tau_1)\sqrt{\alpha} \big).
    \end{equation}
    Recall that $\abs{\mathcal{N}(\tau_1)} = \abs{\mathcal{N}(c\tau_1+d)} = \abs{\Im(\tau_1)} = \abs{\det(g)}= 1$ and that $\abs{\Re(\tau_1)} < 1$. Now, $\mathcal{N}(\tau_1) = \Re(\tau_1)^2-\Im(\tau_1)^2\alpha$, which is congruent to $-\alpha$ modulo $P_\infty$. Furthermore, $\mathcal{N}(c\tau_1+d)$ is congruent to $d^2-c^2\alpha$ modulo $P_\infty$. Hence, by isolating the part involving $\sqrt{\alpha}$ above, we see that the right-hand side in \eqref{tau2eq} is in $D$ if and only if
    \begin{equation*}
        bd \equiv \alpha ac \,\,\mathrm{mod}\,\, P_\infty, \,\,\, ad-bc\equiv d^2-c^2\alpha \,\,\mathrm{mod}\,\, P_\infty.
    \end{equation*}
    As $a,b,c,d\in \fq$ these congruences are equalities in $\fq$. If one fixes $c,d\in \fq^*$, then one sees that there is always a unique choice of $a,b$ solving the above equations. Hence, the system above has precisely $(q-1)^2$ solutions with $c,d\neq 0$.

    Possibly some of the solutions above simply maps $\tau_1$ to $\tau_1$. We therefore study the stabilisers of elements in $D$ and suppose that $g(\tau_1) = \tau_1$. If $c=0$, then our previous calculations show that $b=0$ and $a=d\in \fq^*$ so that $g$ is of the form $\lambda I$. In the case when $c\neq0$, the calculations above show that $g\in \GLtwo(\fq)$. This is a finite group, and one checks that each matrix which is not a multiple of the identity can stabilise at most two distinct points. Hence, up to finitely many exceptions, the stabiliser of $\tau_1$ in $D$ is the set of matrices $\lambda I$ in $\GLtwo(\fq)$. As there are $q-1$ such matrices, we find that up to finitely many exceptions, each element $\tau_1\in D$ with $\abs{\Im(\tau_1)} = 1$ is in the same $\GLtwo(R)$-orbit as $q+1 = (q-1) + 2$ distinct elements from $D$ (including $\tau_1$ itself). Here, the term $2$ comes from the two elements in the same orbit in the case $d=0$ studied above. If $\abs{\Im(\tau_1)} > 1$ then $\tau_1$ is the unique element from $D$ in its orbit.

    We have now found a fundamental domain for the action of $\GLtwo(R)$ on $K_\infty(\sqrt{\alpha})\setminus K_\infty$, up to some repetitions. We can use this to find a fundamental domain for the action of $\GLtwo(R)$ on $\GLtwo(K_\infty)/(\Lambda K)$. Indeed, from our proof of transitivity of the group action, we see that we can use $D$ to construct a fundamental domain $\Tilde{S}$ of matrices of the form
    \begin{equation*}
    \begin{pmatrix}
        1 & f\\ 0 & 1
    \end{pmatrix}\begin{pmatrix}
        t & 0\\ 0 & t^{-1}
    \end{pmatrix}\begin{pmatrix}
        \pi^{-\epsilon} & 0\\ 0 & 1
    \end{pmatrix},
    \end{equation*}
    with $\epsilon \in \{0,1\}$, $\abs{t}\geq 1$, $\abs{f} < 1$ and with the unit part of $t$ congruent to $1$. The points with $\epsilon =0$ and $\abs{t} = 1$ corresponds to the elements in $D$ with $(q+1)$ orbit representatives. Note that setting $\epsilon = 0$ or $1$ splits the fundamental domain into two parts depending on the parity of the valuation of the determinant.

    Finally, we use the fundamental domain above to find a fundamental domain for $ \GLtwo(R)\setminus\GLtwo(K_\infty)$. Let $\lambda, \lambda' \in \Lambda$, $s,s'\in S$ and $k,k'\in K$. Then, $\lambda s k = g\lambda' s'k'$ with $g\in \GLtwo(R)$ first of all means that $\lambda = \lambda'$ by considering the determinant. By acting on $\sqrt{\alpha}$, we see that in fact $g=g_0g'$, where $g' \in \fq^*$ and $g_0\in \GLtwo(\fq)$ is such that $g_0s' = s$. This provides either $q-1$ or $(q+1)(q-1)$ choices for $g$, with finitely many exceptions. As we are only interested in matrices whose determinant has even valuation, the proposition follows, with the set $S$ obtained from restricting to $\epsilon =0$ in the set $\Tilde{S}$ above.
\end{proof}

\subsection{Thickening the cusp}

We now wish to study $\#\{\mathcal{F}_Xv_\sigma \cap V(R)\}$. We accomplish this by averaging, as in \cite[Section 5.3]{BST}, over a well-chosen compact set. Specifically, we can replace $v_\sigma$ above with any $v\in V(K_\infty)^\sigma$ and in particular, we can replace $v_\sigma$ with $gv_\sigma$ for any $g\in \GLtwo(K_\infty)$. Using \eqref{multiplic} and Proposition \ref{fundprop}, we see that if $G_0\subseteq \GLtwo(K_\infty)$ is any set of finite nonzero measure with respect to a Haar measure $dg'$, then
\begin{equation}\label{intsetstart}
    N(V(R)^\sigma; X) = \frac{1}{\vol(G_0)\#\mathrm{Aut}(\sigma)}\int_{g'\in G_0} \#\{\mathcal{F}_Xg'v_\sigma \cap V(R)^\sigma\}'dg',
\end{equation}
where the $'$ indicates that the forms are counted with a factor $(q-1)^{-1}$, except for those corresponding to elements in $\mathcal{F}_X$ with $\abs{t}=1$, in the language of Proposition \ref{fundprop}, which should be weighed with a factor $(q-1)^{-1}(q+1)^{-1}$.

We now describe the set $G_0$ over which we will perform our averaging. The choice of this set marks the first occasion where we are significantly helped by the fact that we are working over a function field to simplify matters. Indeed, we will be able to choose $G_0$ to be both open and compact.

We first let $G'\subseteq \GLtwo(K_\infty)$ be the set of all matrices of the form
\begin{equation*}
    \begin{pmatrix}
        a & b \\ c &d
    \end{pmatrix},
\end{equation*}
with $\abs{a-1}, \abs{d-1} < 1$ and $\abs{b},\abs{c} < 1$. Then, $G'$ is compact and only consists of matrices with determinant whose absolute value equals one. Recall the compact set $K$ from Proposition \ref{fundprop}. We consider the product $G'' = KG'$, which is necessarily compact as the product of compact sets. As $G'$ is open, this product is also open as a union of the open sets $kG'$, $k\in K$.

For a fixed $v_\sigma$, consider the map $e:\GLtwo(K_\infty) \to V(K_\infty)$ given by $g\mapsto gv_\sigma$. A computation shows that the Jacobian of this map is $\mathrm{Disc}(v_\sigma) \neq 0$ so that the map is open. We finally let $G_0 = e^{-1}(e(G''))$ so that
\begin{equation*}
    G_0 = \bigcup_{h\in \mathrm{Stab}_{\GLtwo(K_\infty)}(v_\sigma)}G''h,
\end{equation*}
 which is also open and compact. We note that $e^{-1}(e(G_0)) = G_0$.

Now, we rewrite $\eqref{intsetstart}$ into a form that is easier to study. For $g\in \mathcal{F}_X$, let $\eta(g)$ be $(q-1)^{-1}(q+1)^{-1}$ if $g$ corresponds to a matrix with $\abs{t}=1$ and $(q-1)^{-1}$ else. The integral in \eqref{intsetstart} is then equal to
\begin{equation*}
    \sum_{x\in V(R)}\sum_{g\in \mathcal{F}_X}\eta(g)\int_{g'\in G_0} \mathbf{1}_{\{gg'v_\sigma = x\}}dg' = \sum_{x\in V(R)}\sum_\sumstack{h\in \GLtwo(K_\infty)\\ x=hv_\sigma}\int_{g'\in G_0}\eta(hg'^{-1}) \mathbf{1}_{\{g' \in \mathcal{F}_X^{-1}h\}}dg'.
\end{equation*}
By inversion invariance of $dg'$, we find that
\begin{equation*}
    \sum_{x\in V(R)}\sum_\sumstack{h\in \GLtwo(K_\infty)\\ x=hv_\sigma}\int_{g'\in G_0 \cap \mathcal{F}_X^{-1}h}\eta(hg'^{-1})dg' = \sum_{x\in V(R)}\sum_\sumstack{h\in \GLtwo(K_\infty)\\ x=hv_\sigma}\int_{g'\in G_0^{-1} \cap h^{-1}\mathcal{F}_X}\eta(hg')dg',
\end{equation*}
which, by multiplication invariance, is
\begin{equation*}
    \sum_{x\in V(R)} \sum_\sumstack{h\in \GLtwo(K_\infty)\\ x=hv_\sigma}\int_{g\in hG_0^{-1} \cap \mathcal{F}_X}\eta(g)dg = \int_{g\in \mathcal{F}_X}\eta(g)\sum_{x\in V(R)}\sum_\sumstack{h\in \GLtwo(K_\infty)\\ x=hv_\sigma}\mathbf{1}_{\{g\in hG_0^{-1}\}}dg.
\end{equation*}
The double sum equals
\begin{equation*}
    \sum_{x\in V(R)}\#\{h\in gG_0: hv_\sigma =x\} = \#\mathrm{Aut}(\sigma)\#\{x\in gG_0v_\sigma\cap V(R)\}.
\end{equation*}
Letting $B=G_0v_\sigma$, viewed as a set and not a multiset, we finally have that
\begin{equation}\label{Babove}
    N(V(R)^\sigma; X) = \frac{1}{\vol(G_0)}\int_{g\in \mathcal{F}_X}\eta(\abs{t}) \#\{x\in gB\cap V(R)\}dg,
\end{equation}
where points on the left-hand side are weighed by the inverse size of their stabilisers in $\GLtwo(R)$. Here we write $\eta(\abs{t})$ instead of $\eta(g)$ as $\eta$ only depends on $\abs{t}$.

To study the integral on the right-hand side above, we will need more information about the Haar measure on $\GLtwo(K_\infty)$. First, just as over $\mathbb{R}$, one sees that one Haar measure is $dg = \abs{\det g}^{-2}d\alpha d\beta d\gamma d\delta$ with $\alpha,\beta,\gamma,\delta$ being the entries of $g$. One can check through a differential calculation that in the coordinates of Proposition \ref{fundprop}, the Haar measure is given by $\abs{t}^{-3}\abs{\lambda}^{-1}dtd\lambda df dk$, where $dk$ is a Haar measure on the group $K$, cf. \cite[Section 5.3]{BST}. Furthermore, in the coordinates $a$ and $c$ from Proposition \ref{fundprop}, we have that $dk = a\,dc-c\,da$. We denote the measure of $K$ with respect to $dk$ by $\nu(K)$. The measures $dt$, $d\lambda$, and $df$ are Haar measures on the local field $K_\infty$, giving $\mathcal{O}_\infty$ measure $q$ (the purpose of this normalisation will be clear soon). Finally, we recall that $B$ is left $K$-invariant which means that
\begin{equation}\label{intcountcoords}
    N(V(R)^\sigma; X) = \frac{\nu(K)}{\vol(G_0)}\int_{ g\in \Lambda_X S }\eta(\abs{t}) \#\{x\in gB\cap V(R)\}\abs{t}^{-3}\abs{\lambda}^{-1}dtd\lambda df.
\end{equation}

\subsection{Geometry of numbers}
To study \eqref{intcountcoords}, we will find a good estimate of $\#\{x\in gB\cap V(R)\}$ by employing methods from the geometry of numbers. Over $\mathbb{Q}$, this was originally done by Davenport and Heilbronn \cite{DH} using Davenport's lemma. Over function fields, we can obtain very precise results by utilising the non-archimedean geometry.

We first recall that the set $B$ is defined as $G_0v_\sigma$. One checks through a Jacobian calculation that the map $g\mapsto gv_\sigma$ is open, so that $B$ is open. As $G_0$ is also compact, the same holds for $B$. This shows that the indicator function of $B$ is continuous on a compact set, and hence uniformly continuous. This means that there is some integer $\kappa$ such that $\abs{x_i-y_i} < q^\kappa$ for all $i=1,2,3,4$, implies that $x\in B$ if and only if $y\in B$. In particular, $B$ must be the disjoint union of finitely many $C_i$, where $C_i = v_i + \{(x_1,x_2,x_3,x_4): \abs{x_i} < q^\kappa\}$.

The above argument shows that counting lattice points in an open compact set can be reduced to the study of lattice points inside a box. After translation, we may assume that this box is centred at the origin. Counting points inside such a box can be done using the following result, proven in \cite[Theorem 23]{BSW} in greater generality using Poisson summation. We give an elementary proof which is sufficient for our purposes below.
\begin{lemma}\label{latticeptlemma}
    The set $\Tilde{C} = \{x\in K_\infty: \abs{x} < q^k\}$ contains exactly $q^k$ points from $R$ if $k \geq 0$, else it contains precisely one such point.
\end{lemma}
\begin{proof}
    The points from $R$ inside $\Tilde{C}$ are precisely the polynomials in $T$ of degree less than or equal to $k-1$, and zero. If $k\geq 0$, there are precisely $q^k$ such polynomials. 
\end{proof}
Recalling our normalisation of the Haar measure on $K_\infty$, we see that $\#\{\Tilde{C}\cap R\} = \vol(\Tilde{C})$ for $k\geq 0$. Lemma \ref{latticeptlemma} is easily generalised to $K_\infty^4$.

While the proof of Lemma \ref{latticeptlemma} may seem trivial, the lemma itself is strong enough to allow us to determine $N(V(R)^\sigma; X)$ to a very high precision. Specifically, we have the following result.
\begin{prop}\label{implicitformcount}
    Let $\lambda_0$ be an arbitrary element from $\Lambda_X$. For $X $ larger than some absolute constant, we have that
    \begin{equation*}
        N(V(R)^\sigma; X) = \frac{\nu(K)}{\vol(G_0)}\left(\frac{q}{(q-1)(q^2-1)}\abs{\lambda_0}^{4}\vol(B)- \frac{1}{(q-1)(q^2-1)}\abs{\lambda_0}^{10/3}I^\sigma_1(\lambda_0)+\frac{q}{(q-1)^2}\abs{\lambda_0}^4I_0^\sigma\right), 
    \end{equation*}
    where $I_0^\sigma$ is defined in \eqref{redtermint} and $I^\sigma_1(\lambda_0) =: I^\sigma_1(X)$, defined in \eqref{sectermint}, depends only on the congruence class of $\log_q(X)$ modulo $3$ and on $\sigma$. 
\end{prop}
\begin{proof}
    We evaluate \eqref{intcountcoords} by determining $\#\{x\in gB\cap V(R)\}$ depending on
    \begin{equation*}
        g = \lambda\begin{pmatrix}
            1 & f\\ 0 & 1
        \end{pmatrix}\begin{pmatrix}
            t & 0\\ 0 & t^{-1}
        \end{pmatrix}.
    \end{equation*}
    We begin by writing $B$ as the disjoint union 
    \begin{equation*}
        B = \bigcup_{i\leq i_0}\left(v_i + C\right),
    \end{equation*}
    for some $v_i\in K_\infty^4$ and $C=\{(x_1,x_2,x_3,x_4): \abs{x_i} < c=q^{c'}\}$, with $c'\in \mathbb{Z}$. Now, 
    \begin{equation}\label{grpact}
        gC = \begin{pmatrix}
            1 & f\\ 0 & 1
        \end{pmatrix} \big\{(x_1,x_2,x_3,x_4): \abs{x_1} < c\abs{\lambda t^3}, \,\,\abs{x_2} < c\abs{\lambda t},\,\, \abs{x_3} < c\abs{\lambda/ t}, \,\,\abs{x_4} < c\abs{\lambda/ t^{3}}\big\} =: \begin{pmatrix}
            1 & f\\ 0 & 1
        \end{pmatrix}C_{\lambda,t}.
    \end{equation}
    As the form $(a,b,c,d)$ is mapped to $(a+bf+cf^2+df^3,b+2cf+3df^2,c+3df,d)$, the right-hand side of \eqref{grpact} is just $C_{\lambda,t}$, as $\abs{f} < 1$ and $\abs{t}\geq 1$. We conclude that
    \begin{equation}\label{gBdecomp}
        gB = \bigcup_{i\leq i_0}\left(gv_i + C_{\lambda,t}\right).
    \end{equation}

We now use the slicing method from \cite[Section 6]{BST}, and slice over the last coordinate of the forms in $gB$. To obtain a precise expression, we split the slicing into two cases. The first case is when the last coordinate is nonzero. Then, we must have that $c\abs{\lambda /t^3} \gg 1$, which for large enough $X$ implies that $\abs{\lambda/t}$ is large, as $\abs{\lambda}^4 \asymp X$. Here we use that $g$ scales the last coordinate of $v_i$ with a factor $\lambda/t^3$. If the last coordinate in $gB$ is zero, then we extend the slicing method and slice over the third coordinate as well.

We now make the argument above explicit. We have that
\begin{equation}\label{firstSlice}
    \#\{x\in gB\cap V(R)\} = (q-1)\sum_{d\in R\setminus\{0\}}' \#\{x\in (gB)_d\cap R^3\} + (q-1)\sum_{c\in R\setminus\{0\}}' \#\{x\in (gB)_{c,0}\cap R^2\},
\end{equation}
where the $'$ denotes that the sum is restricted to monic polynomials, which is why the factor $(q-1)$ appears. This is justified by the $K$-invariance of $B$ as $K$ contains the matrices $\lambda I$ with $\lambda\in \fq^*$. Here, with $F\subseteq V(R)$ a collection of forms, $F_d$ denotes the set of elements $(a,b,c)$ such that $(a,b,c,d)\in F$, and similarly $F_{c,0}$ denotes the set of elements $(a,b)$ with $(a,b,c,0)\in F$. As $gB$ contains no singular forms, we see that $(gB)_{0,0}$ is empty.

The two sums above should be integrated according to \eqref{intcountcoords}. We begin by computing the integral over the second sum, as it will turn out to be the simplest of the two integrals. As $c\abs{\lambda t^3}\geq c\abs{\lambda t}\geq 1$, we have using Lemma \ref{latticeptlemma} that $\#\{x\in (gB)_{c,0}\cap R^2\} = \vol\big((gB)_{c,0}\big) = \abs{\lambda}^2\abs{t}^4\vol(B_{ct/\lambda,0})$. We then compute that
\begin{equation*}
    (q-1)\sum_{c\in R\setminus\{0\}}'\int_{ g\in \Lambda_X S }\eta(\abs{t}) \#\{x\in (gB)_{c,0}\cap R^2\}\abs{t}^{-3}\abs{\lambda}^{-1}dtd\lambda df = (q-1)\sum_{c\in R\setminus\{0\}}'\abs{\lambda_0}^2\int_{ \abs{t}\geq 1}\eta(\abs{t}) \vol(B_{ct/\lambda_0,0})\abs{t}dt,
\end{equation*}
where $\lambda_0$ is an arbitrary representative from $\Lambda_X$. Here, the integration over $t$ is only over $t$ with unit part congruent to $1$ modulo $P_\infty$, but we suppress this in the notation above. We need only consider $\lambda_0$ instead of $\lambda$ as $\abs{\lambda}$ is constant for all $\lambda \in \Lambda_X$, for a fixed $X$. We also made use of the fact that the integral over $\lambda$ and $f$ of $1$ is simply equal to $1$. 

Next, making the change of variables $u=ct/\lambda_0$ so that $dt = \abs{\lambda_0 c^{-1}}du$, we see that the above is equal to
\begin{equation}\label{sumcint}
    \sum_{c\in R\setminus\{0\}}'\abs{c}^{-2}\abs{\lambda_0}^4\int_{\abs{u}\geq \abs{c/\lambda_0} } \eta(\abs{u\lambda_0/c})\vol(B_{u,0})\abs{u}du = \abs{\lambda_0}^4\int_{\abs{u}\geq  \abs{1/\lambda_0} } \Bigg(\sum_\sumstack{c\in R\setminus\{0\}\\ \abs{c}\leq \abs{\lambda_0 u}}'\abs{c}^{-2}\eta(\abs{u\lambda_0/c})\Bigg)\vol(B_{u,0})\abs{u}du,
\end{equation}
where the factor $(q-1)$ was used to extend the integration over $u$ so that its unit part can be congruent modulo $P_\infty$ to any nonzero element in $\fq$.

We compute the sum over $c$. Write $\abs{u\lambda_0} = q^m$ with $m\geq 0$. Then, the sum is
\begin{equation*}
    \sum_\sumstack{c\in R\setminus\{0\}\\ \deg(c) \leq m}'\abs{c}^{-2}\eta(\abs{u\lambda_0/c})  = \frac{q^{-m}}{(q-1)(q+1)}+\frac{1}{q-1}\sum_\sumstack{c\in R\setminus\{0\}\\ \deg(c) \leq m-1}'\abs{c}^{-2},
\end{equation*}
which, by a geometric sum calculation, equals
\begin{equation*}
    \frac{q^{-m}}{(q-1)(q+1)}+\frac{1}{q-1}\sum_\sumstack{0\leq k\leq m-1} q^{-k} = \frac{q^{-m}}{(q-1)(q+1)}+q\cdot\frac{1-q^{-m}}{(q-1)^2} = \frac{1}{q-1}\left(\frac{q}{q-1} - \abs{u\lambda_0}^{-1}\frac{q^2+1}{q^2-1}\right).
\end{equation*}
The right-hand side in \eqref{sumcint} thus equals
\begin{equation}\label{redintcontr}
    \frac{q\abs{\lambda_0}^4}{(q-1)^2}I_0^\sigma- \frac{(q^2+1)\abs{\lambda_0}^3}{(q^2-1)(q-1)}\vol(B_0),
\end{equation}
where 
\begin{equation}\label{redtermint}
    I_0^\sigma = \int_u \vol(B_{u,0})\abs{u}du.
\end{equation}
Here, we made use of the fact that for $\abs{u} < \abs{1/\lambda_0}$, and large enough $\abs{\lambda_0}$, we have that $B_{u,0} = \emptyset$, as the indicator function of $B$ is uniformly continuous and $B_{0,0} = \emptyset$. We remark that if $\sigma = (3), (1^3)$, then $I_0^\sigma = 0$ as $B$ does not contain any point with last coordinate equal to zero in these cases.

We turn to the summation over nonzero $d$. Note that $\#\{x\in (gB)_d\cap R^3\} = \abs{\lambda}^3\abs{t}^3\vol(B_{dt^3/\lambda})$, whence the corresponding integral equals
\begin{equation*}
    (q-1)\sum_{d\in R\setminus\{0\}}'\abs{\lambda_0}^3\int_{ \abs{t}\geq 1}\eta(\abs{t}) \vol(B_{dt^3/\lambda_0})dt.
\end{equation*}
This integral is slightly more delicate than the one handled previously. We would like to make the change of variables $u = dt^3/\lambda_0$, but we note that while the map $t\mapsto t^3$ is bijective on $1+\pi\mathcal{O}_\infty$, it is not surjective on all of $\mathcal{O}_\infty^*$. Indeed, the valuation of $t^3$ is always divisible by three. To solve this issue, we split the summation over $d$ depending on the congruence class of the degree modulo $3$ and write
\begin{equation*}
    (q-1)\sum_{\epsilon = 0}^2\sum_\sumstack{d\in R\setminus\{0\}\\ \deg(d) \equiv^3 \epsilon}'\abs{\lambda_0}^3\int_{ \abs{t}\geq 1}\eta(\abs{t}) \vol(B_{dt^3/\lambda_0})dt,
\end{equation*}
where $\equiv^3$ denotes congruence modulo three. Writing $u=dt^3/\lambda_0$, we see that the condition $\abs{t}\geq 1$ means that $\abs{u}\geq \abs{d/\lambda_0}$, and that $\deg(d) \equiv^3\epsilon$ is equivalent to $-v_\infty(u\lambda_0) \equiv^3 \epsilon$. We also have the condition that the unit part of $u$ is congruent to $1$, but we may drop this condition by making use of the factor $(q-1)$ outside the integral. To finish the substitution, we also note that $dt = \abs{\lambda_0d^{-1}}^{1/3}\abs{u}^{-2/3}du$. 

From the discussion above, we see that the integral of the sum over $d\neq 0$ is 
\begin{equation*}
\begin{split}
    \sum_{\epsilon = 0}^2&\sum_\sumstack{d\in R\setminus\{0\}\\ \deg(d) \equiv^3 \epsilon}'\abs{\lambda_0}^{10/3}\int_\sumstack{\abs{u}\geq\abs{d/\lambda_0} \\ -v_\infty(u\lambda_0)\equiv^3 \epsilon} \eta(\abs{u\lambda_0/d}^{1/3})\abs{d}^{-1/3}\vol(B_u)\abs{u}^{-2/3}du
    \\&=\sum_{\epsilon = 0}^2\abs{\lambda_0}^{10/3}\int_\sumstack{\abs{u}\geq\abs{1/\lambda_0} \\ -v_\infty(u\lambda_0)\equiv^3 \epsilon} \Bigg(\sum_\sumstack{d\in R\setminus\{0\}\\ \deg(d) \equiv^3 \epsilon \\ \abs{d}\leq \abs{u\lambda_0}}'\eta\big(\abs{u\lambda_0/d}^{1/3}\big)\abs{d}^{-1/3}\Bigg)\vol(B_u)\abs{u}^{-2/3}du.
\end{split}
\end{equation*}
We now seek to evaluate the innermost sum. Let us first write $\abs{u\lambda_0} = q^{3k+\epsilon}$. The sum is then equal to
\begin{equation*}
    \frac{q^{2\epsilon/3}}{q-1}\left(\sum_{0\leq \ell\leq k-1}q^{2\ell}+\frac{q^{2k}}{q+1}\right) = \frac{q^{2\epsilon/3}}{q-1}\left(\frac{q^{2k}-1}{q^2-1}+\frac{q^{2k}}{q+1}\right) = \frac{\abs{u\lambda_0}^{2/3}q}{(q-1)(q^2-1)} -\frac{q^{2\epsilon/3}}{(q-1)(q^2-1)}.
\end{equation*}
Integrating these two terms yields
\begin{equation*}
\begin{split}
    \frac{q}{(q-1)(q^2-1)}\sum_{\epsilon = 0}^2\abs{\lambda_0}^{4}\int_\sumstack{\abs{u}\geq\abs{1/\lambda_0} \\ -v_\infty(u\lambda_0)\equiv^3 \epsilon} \vol(B_u)du - \frac{1}{(q-1)(q^2-1)}\sum_{\epsilon = 0}^2q^{2\epsilon/3}\abs{\lambda_0}^{10/3}\int_\sumstack{\abs{u}\geq\abs{1/\lambda_0} \\ -v_\infty(u\lambda_0)\equiv^3 \epsilon} \vol(B_u)\abs{u}^{-2/3}du,
\end{split}
\end{equation*}
which, after separating the tail integrals, becomes
\begin{equation}\label{nonzerocontr}
\begin{split}
    \frac{q}{(q-1)(q^2-1)}&\abs{\lambda_0}^{4}\left(\vol(B)-\int_\sumstack{\abs{u}<\abs{1/\lambda_0} } \vol(B_u)du \right) \\&- \frac{1}{(q-1)(q^2-1)}\abs{\lambda_0}^{10/3}\left(I^\sigma_1(\lambda_0)-\sum_{\epsilon = 0}^2q^{2\epsilon/3}\int_\sumstack{\abs{u}< \abs{1/\lambda_0} \\ -v_\infty(u\lambda_0)\equiv^3 \epsilon} \vol(B_u)\abs{u}^{-2/3}du\right),
\end{split}
\end{equation}
with 
\begin{equation}\label{sectermint}
    I^\sigma_1(\lambda_0) = \sum_{\epsilon = 0}^2q^{2\epsilon/3}\int_\sumstack{ -v_\infty(u\lambda_0)\equiv^3 \epsilon} \vol(B_u)\abs{u}^{-2/3}du
\end{equation}
depending only on $v_\infty(\lambda_0)$ modulo three and on $v_\sigma$.

The tail integrals above can be simplified. First, for $\abs{u}< \abs{1/\lambda_0}$ we have that $\vol(B_u) = \vol(B_0)$. Letting $\mu$ denote the measure obtained from $du$ we find that
\begin{equation*}
    -\int_\sumstack{\abs{u}<\abs{1/\lambda_0} } \vol(B_u)du = -\vol(B_0)\mu(\pi^{v_\infty(\lambda_0)+1}\mathcal{O}_\infty) = -\vol(B_0)q^{v_\infty(\lambda_0)}=-\vol(B_0)\abs{\lambda_0}^{-1}.
\end{equation*}
Let us write $\epsilon(u\lambda_0)$ for the representative in $\{0,1,2\}$ of the congruence class modulo three of $-v_\infty(u\lambda_0)$. Then,  making the change of variables $r=u\lambda_0$, the second tail integral above can be written as
\begin{equation}\label{belowepsdef}
   \vol(B_0)\int_\sumstack{\abs{u}< \abs{1/\lambda_0}} q^{2\epsilon(u\lambda_0)/3}\abs{u}^{-2/3}du = \vol(B_0)\abs{\lambda_0}^{-1/3}\int_{\abs{r} < 1} q^{2\epsilon(r)/3}\abs{r}^{-2/3}dr.
\end{equation}
Moreover,
\begin{equation*}
    \int_{\abs{r} < 1} q^{2\epsilon(r)/3}\abs{r}^{-2/3}dr = \sum_{\epsilon = 0}^2 q^{2\epsilon/3}\sum_{k=1}^\infty \int_{\abs{r} = q^{-3k+\epsilon}}\abs{r}^{-2/3}dr = \sum_{\epsilon = 0}^2 q^{2\epsilon/3}\sum_{k=1}^\infty q^{2k-2\epsilon/3}\mu(\pi^{3k-\epsilon}\mathcal{O}_\infty^*),
\end{equation*}
which equals
\begin{equation*}
    (q-1)\sum_{\epsilon = 0}^2 q^{2\epsilon/3}\sum_{k=1}^\infty q^{-k+\epsilon/3} = 1+q+q^2.
\end{equation*}
It follows that the second term from \eqref{redintcontr} cancels against the two tails terms from \eqref{nonzerocontr}, and this almost concludes the proof of the proposition. As the choice of $v_\sigma$, and thus $\abs{\lambda_0}$ depends on $\log_q(X)$ modulo $4$, it remains to prove that $I^\sigma_1(\lambda_0)$ is independent of this choice. We postpone the proof of this fact to the end of the section.  
\end{proof}
The term in Proposition \ref{implicitformcount} involving $I_0^\sigma$ corresponds, in a sense, to reducible forms. First, any point with last coordinate equal to zero is reducible, as such a form has a root. Furthermore, one may prove an analogue of \cite[Lemma 21]{BST} with almost identical proof, showing that the number of reducible points with nonzero last coordinate in $\mathcal{F}_Xv_\sigma$ is $\ll X^{3/4+\epsilon}$. Hence, if one is only interested in irreducible forms, then one need not consider the slice of $gB$ where $d=0$. Copying the proof of the above proposition, but bounding the tail integrals as $\ll \abs{\lambda_0}^3$, one finds the following result.
\begin{lemma}\label{irrformcount}
    Let $V(R)^{\sigma, \text{irr}}$ denote the irreducible forms in $V(R)^\sigma$. Then,
    \begin{equation*}
        N(V(R)^{\sigma, \text{irr}}; X) = \frac{\nu(K)}{\vol(G_0)}\left(\frac{q}{(q-1)(q^2-1)}\abs{\lambda_0}^{4}\vol(B)- \frac{1}{(q-1)(q^2-1)}\abs{\lambda_0}^{10/3}I^\sigma_1(\lambda_0)\right) + \mathcal{O}_\epsilon\left(X^{3/4+\epsilon}\right).
    \end{equation*}
\end{lemma}
We will use this lemma later on when counting fields to avoid having to compute $I_0^\sigma$ explicitly.

\subsection{Volumes and automorphisms}\label{autch}
We now make Proposition \ref{implicitformcount} more explicit by computing the various volumes and integrals. First, as the Jacobian of the map $g\mapsto gv_\sigma$ is $\mathrm{Disc}(v_\sigma)$, and $G_0$ covers $B$ with multiplicity $\#(\mathrm{Aut}(\sigma))$ we see that $\vol(G_0) = \#(\mathrm{Aut}(\sigma))\vol(B)/\abs{\mathrm{Disc}(v_\sigma)}$. Recall also that $\abs{\lambda_0}^4\abs{\mathrm{Disc}(v_\sigma)} = X$.

We now compute $\nu(K)$, the measure of $K$ with respect to $dk = adc-cda$. Let $\Lambda'$ denote the set of $\lambda \in K_\infty$ with $\abs{\lambda} = 1$ and unit part congruent to $1$ modulo $P_\infty$. If $k\in K$ has the form
\begin{equation}\label{kform}
    \begin{pmatrix}
        a & c\alpha \\ c & a
    \end{pmatrix},
\end{equation}
we can consider the product map
\begin{equation}\label{productmap}
    \Lambda' \times K \to K_\infty^2\setminus\{(0,0)\},
\end{equation}
given by $\lambda k \mapsto (\lambda a, \lambda c)$. One sees that this map is injective and maps onto the set $ S = \{(a,c): \max\{\abs{a}, \abs{c}\} = 1\}$, and that the form corresponding to the form $\lambda\,d\lambda \wedge dk$ is $dx_1\wedge dx_2$. The measure of $\Lambda' \times K$ is $\nu(K)$, where we used that $\abs{\lambda}=1$, while the measure of $S$ is
\begin{equation*}
    2\int_{\abs{x_1}=1}\int_{\abs{x_2}\leq 1} dx_2dx_1-\int_{\abs{x_1}=1}\int_{\abs{x_2}= 1} dx_2dx_1 = 2(q-1)q-(q-1)(q-1) = q^2-1,
\end{equation*}
so that $\nu(K) = q^2-1$. Using this observation, we conclude that the expression in Proposition \ref{implicitformcount} can be rewritten as
\begin{equation}\label{stepformcount}
  \frac{q}{(q-1)\#\mathrm{Aut}(\sigma)}X- \frac{\abs{\lambda_0}^{10/3}}{(q-1)\vol(G_0)}I^\sigma_1(\lambda_0)+\frac{q(q+1)I_0^\sigma}{(q-1)\vol(B)\#\mathrm{Aut}(\sigma)}X.
\end{equation}

To evaluate $\#\mathrm{Aut}(\sigma)$, we should classify the various Étale extensions of $K_\infty$. These have the form
\begin{equation*}
    K_1 \oplus ... \oplus K_r,
\end{equation*}
where each $K_i$ is a finite extension of $K_\infty$ and where the various degrees add up to three. We thus see that to classify the cubic Étale extensions of $K_\infty$, we should classify the extensions of degree $1,2$ and $3$ of $K_\infty$.

Clearly, there is a unique extension of degree one, namely $K_\infty$ itself. For $\sigma$ representing the extension $K_\infty^3$ we clearly have $\#\mathrm{Aut}(\sigma) = 6$. Classifying the extensions of degree two is equivalent to studying $K_\infty^* / (K_\infty^*)^2$. Now, the decomposition $x=u\pi^n$ for elements $x\in K^*_\infty$ with $u\in \mathcal{O}_\infty^*$ shows that
\begin{equation}\label{Kinfdecomp}
    K^*_\infty \simeq \fq^*\times  \left( 1+\pi\mathcal{O}_\infty\right) \times \mathbb{Z},
\end{equation}
so that
\begin{equation*}
    (K_\infty^*)^2 \simeq (\fq^*)^2\times  \left( 1+\pi\mathcal{O}_\infty\right) \times 2\mathbb{Z},
\end{equation*}
where we used that $\left( 1+\pi\mathcal{O}_\infty\right) = \left( 1+\pi\mathcal{O}_\infty\right)^2$ by Hensel's lemma. Hence, the quotient is
\begin{equation}\label{squaredecomp}
    K_\infty^*/(K_\infty^*)^2 \simeq \frac{\fq^*}{(\fq^*)^2} \times \frac{\mathbb{Z}}{2\mathbb{Z}}.
\end{equation}
Letting $\alpha$ be a non-square in $\fq$ we therefore see that the quadratic field extensions of $K_\infty$, corresponding to the non-identity elements in \eqref{squaredecomp}, are $K_\infty(\sqrt{\alpha}), K_\infty(\sqrt{\pi})$ and $K_\infty(\sqrt{\alpha \pi})$. The first of these extensions is unramified, while the last two are totally ramified. For all three extensions, we have that $\#\mathrm{Aut}(\sigma) = 2$.

We turn to cubic field extensions of $K_\infty$. By general theory, there is a unique unramified extension, namely the splitting field of $X^{q^3}-X$ over $K_\infty$, or equivalently the compositum of $K_\infty$ with $\mathbb{F}_{q^3}$ (see \cite[p.173]{Neukirch}). The ramified extensions must be totally (tamely) ramified as $3$ is a prime. Then, by general considerations, these extensions are all obtained by appending the cube root of a uniformiser to $K_\infty$ so that they all have the form $K_\infty((u\pi)^{1/3})$, with $u \in \mathcal{O}_\infty^*$. By Hensel's lemma
\begin{equation*}
    \frac{\mathcal{O}_\infty^*}{\left(\mathcal{O}_\infty^*\right)^3} \simeq \frac{\fq^*}{(\fq^*)^3},
\end{equation*}
which has order one if $q\equiv^3 2$, and order three if $q \equiv^3 1$ (recall that $3\nmid q$). Hence, if $q\equiv^3 2$, then $K_\infty(\pi^{1/3})$ is the unique cubic ramified extension. On the other hand, if $q\equiv^3 1$, and $\beta \in \fq$ is a non-cube, the three non-isomorphic extensions are given by $K_\infty((\beta^i\pi)^{1/3})$ with $i=0,1,2$. When $q \equiv^3 2$, the extension is non-Galois, as $\fq$ contains no primitive third root of unity, and $\#\mathrm{Aut}(\sigma) = 1$. If instead $q\equiv^3 1$, then all three extensions are Galois with $\#\mathrm{Aut}(\sigma) = 3$.

\subsection{Computation of integrals}
Having classified the Étale extensions of $K_\infty$ of degree three, we turn to the evaluation of $I^\sigma_1(X)$ for various $\sigma$ and $X$. For a form $v\in V(K_\infty)$, we write $d(v)$ for the last coordinate of $v$. Then, we can write the integral defining $I_1^\sigma$ as 
\begin{equation*}
    I_1^\sigma(X) = \int_\sumstack{u} e^{2\epsilon(u\lambda_0)/3}\vol(B_u)\abs{u}^{-2/3}du = \int_{v\in B}q^{2\epsilon(d(v)\lambda_0)/3}\abs{d(v)}^{-2/3}dv = \frac{\abs{\mathrm{Disc}(v_\sigma)}}{\#\mathrm{Aut}(\sigma)}\int_{g\in G_0}q^{2\epsilon(d(gv_\sigma)\lambda_0)/3}\abs{d(gv_\sigma)}^{-2/3}dg. 
\end{equation*}
As $G_0$ is $K$-invariant, we may average over $K$ and write the above as
\begin{equation}\label{secaverage}
    \frac{\abs{\mathrm{Disc}(v_\sigma)}}{\#\mathrm{Aut}(\sigma)\nu(K)}\int_{g\in G_0}\int_{k\in K}q^{2\epsilon(d(kgv_\sigma)\lambda_0)/3}\abs{d(kgv_\sigma)}^{-2/3}dkdg.
\end{equation}
Recall that we defined the set $G_0$ somewhat explicitly previously. However, all our calculations remain true for any open, compact, left $K$-invariant set $G_0$, containing only matrices whose determinant has valuation equal to zero, as well as being invariant under multiplication with elements from $\mathrm{Stab}_{\GLtwo(K_\infty)}(v_\sigma)$ from the right. We show that by redefining the set $G_0$, the inner integral above becomes independent of $g$.

To find an appropriate set $G_0$, we consider the pairing
\begin{equation*}
    \langle x,y\rangle = x_1y_1-\alpha x_2y_2 \in K_\infty, 
\end{equation*}
where $x,y\in K_\infty^2$ and $\alpha\in \fq$ is a non-square. Note that $\langle x,x\rangle$ equals zero only if $x=0$. We can consider the subgroup of $\GLtwo(K_\infty)$ consisting of matrices $g$ such that $\abs{\det(g)} = 1$. By performing Gram-Schmidt orthogonalisation, one can bring such a $g$ into the form
\begin{equation*}
\begin{pmatrix}
    a & bc\alpha \\ c & ab
\end{pmatrix}
\begin{pmatrix}
    \gamma_1 & \beta_1 \\ 0 & \beta_2
\end{pmatrix},
\end{equation*}
where $a^2-\alpha c^2$ and $b^2\alpha(c^2\alpha -a^2)$ both have absolute value equal to one. In fact, by dividing $\gamma_1$ by $b$, multiplying the first column of the left-most matrix above by $b$, and changing variables, we can ensure that $b=1$. We then see that $\abs{\gamma_1\beta_2} = 1$. Recalling \eqref{Kinfdecomp} and that $(1+\pi\mathcal{O}_\infty)^2 = 1+\pi\mathcal{O}_\infty$, we see that after factoring an appropriate $\lambda$ with $\abs{\lambda} = 1$ from the matrix product, we may assume that $\gamma_1\beta_2 \in \{1,\alpha\}$ and that $a^2-\alpha c^2\in \fq$.

By decomposing the upper triangular matrix further, we see that we can write any $g$ as a product 
\begin{equation*}
     k\lambda\begin{pmatrix}
        t & 0\\ 0 &\alpha^i t^{-1}
    \end{pmatrix}\begin{pmatrix}
        1 & f\\ 0 &1
    \end{pmatrix} =: k\lambda a(t)n(f),
\end{equation*}
with $k\in K$ and $i\in\{0,1\}$. We now let $H$ be the set of all such products with $\abs{t} = \abs{\lambda} = 1$, $\abs{f} < 1$. We then define $G_0$ as the union of all $Hg_i$ where $g_i$ stabilises $v_\sigma$.

We now study \eqref{secaverage} and write $g\in G_0$ as $k\lambda a(t)n(f)g_i$. As $g_iv_\sigma = v_\sigma$ and $k\in K$, we can after a change of variables assume that $g = \lambda a(t)n(f)$. We also note that $\abs{d(\lambda v)} = \abs{\lambda}\abs{d(v)} = \abs{d(v)}$ if $\abs{\lambda} = 1$, and thus we may assume that $\lambda = 1$. 

Moreover, a computation shows that $d(gv)$ only depends on the bottom row of the matrix $g$ (and on $v$). This, together with \eqref{productmap}, and invariance under multiplication with $\lambda$ shows that instead of considering the inner integral in $\eqref{secaverage}$ we may study
\begin{equation*}
    \int_{\max\{\abs{x},\abs{y}\} = 1}q^{2\epsilon(d(h(x,y)gv_\sigma)\lambda_0)/3}\abs{d(h(x,y)gv_\sigma)}^{-2/3}dxdy,
\end{equation*}
where $h(x,y)$ is a matrix with bottom row equal to $(x,y)$.

We turn to invariance under multiplication by $a(t)$. One checks that the bottom row of $h(x,y)a(t)$ is given by $(tx, \alpha^i t^{-1}y)$. Making the change of variables $x' = tx$ and $y' = \alpha^i t^{-1}y$ leaves $dxdy$ invariant and shows invariance under multiplication with $a(t)$.

Finally, we note that the bottom row of $h(x,y)n(f)$ is $(x, y+fx)$. Note that $\max\{\abs{x}, \abs{y+fx}\} = \max\{\abs{x},\abs{y}\}$ as $\abs{f} < 1$. Making the change of variables $x'=x$ and $y' = y+fx$ leaves $dxdy$ invariant. This finally shows that \eqref{secaverage} is equal to
\begin{equation*}
\begin{split}
    \frac{\abs{\mathrm{Disc}(v_\sigma)}\vol(G_0)}{\#\mathrm{Aut}(\sigma)\nu(K)}&\int_{\max\{\abs{x},\abs{y}\} = 1}q^{2\epsilon(d(h(x,y)v_\sigma)\lambda_0)/3}\abs{d(h(x,y)v_\sigma)}^{-2/3}dxdy
    \\&=\frac{\abs{\mathrm{Disc}(v_\sigma)}\vol(G_0)}{\#\mathrm{Aut}(\sigma)\nu(K)}\int_{k\in K}q^{2\epsilon(d(kv_\sigma)\lambda_0)/3}\abs{d(kv_\sigma)}^{-2/3}dk.
\end{split}
\end{equation*}
We can thus conclude that the second term in \eqref{stepformcount} is equal to
\begin{equation}\label{tocmplres}
\begin{split}
    -\frac{\abs{\lambda_0}^{10/3}\abs{\mathrm{Disc}(v_\sigma)}}{\#\mathrm{Aut}(\sigma)\nu(K)(q-1)}&\int_{\max\{\abs{x},\abs{y}\} = 1}q^{2\epsilon(d(h(x,y)v_\sigma)\lambda_0)/3}\abs{d(h(x,y)v_\sigma)}^{-2/3}dxdy
    \\&=-\frac{X^{5/6}\abs{\mathrm{Disc}(v_\sigma)}^{1/6}}{\#\mathrm{Aut}(\sigma)(q^2-1)(q-1)}\int_{\max\{\abs{x},\abs{y}\} = 1}q^{2\epsilon(v_\sigma(x,y)\lambda_0)/3}\abs{v_\sigma(x,y)}^{-2/3}dxdy
\end{split}
\end{equation}
where we used that one obtains the last coordinate of a form by evaluating at $(0,1)$. We write $I_\sigma'(X)$ for the integral above including the factor $\abs{\mathrm{Disc}(v_\sigma)}^{1/6}/(q^2-1)$.

We have the following result.
\begin{prop}\label{secevalprop}
    Write $X = q^\ell$ for $\ell\geq 0$, with $\ell$ even unless $\sigma$ corresponds to the sum of $K_\infty$ and a ramified quadratic extension, in which case $\ell$ is odd. Then, the value of $I_\sigma'(X) =: C_2(\ell)$ is given by the following table:
\begin{center}
\begin{tabular}{ |c|c|c|c| } 
 \hline
 Type of $\sigma$ & $\ell \equiv^3 0$ & $\ell \equiv^3 1$ & $\ell \equiv^3 2$ \\ 
 \hline
 $(111)$ & $3q+1$ & $4q^{2/3}$ & $q^{1/3}(q+3)$ \\ 
 $(21)$ & $q+1$ & $2q^{2/3}$  & $q^{1/3}(q+1)$\\ 
 $(3)$ & $1$ & $q^{2/3}$  & $q^{4/3}$ \\
 $(1^21)$ & $2q^{1/2}$& $q^{1/6}(q+1)$  & $2q^{5/6}$ \\
 $(1^3)$
 & $q$ & $q^{2/3}$ & $q^{1/3}$ \\
 \hline
\end{tabular}
\end{center}
\end{prop}
\begin{proof}
    We first argue that $I_\sigma'(q^\ell)$ only depends on $\ell$ modulo $3$ and not modulo $4$. It is apparent that such an invariance holds for the integral for a fixed choice of $v_\sigma$, but recall that our choice of $v_\sigma$ depends on $\ell$ modulo $4$. 

    Suppose that for a fixed choice of $v_\sigma$, we have evaluated $I'_\sigma(q^\ell)$ for all $\ell$ such that the equation $q^\ell = \abs{\lambda_0}^4\abs{\mathrm{Disc}(v_\sigma)}$ is solvable for $\lambda_0$, i.e. $\ell$ lying in a certain congruence class modulo $4$ depending on $\abs{\mathrm{Disc}(v_\sigma)}$. We show how to bootstrap such a result to also evaluate $I_\sigma'(q^{\ell'})$ for $\ell' = \ell + 2$. 
    
    Write $q^{\ell'} = \abs{\lambda_0}^4\abs{\mathrm{Disc}(v_\sigma')}$ with
    \begin{equation*}
        v'_\sigma = \begin{pmatrix}
            \pi^{-1} & 0\\ 0 & 1
        \end{pmatrix}v_\sigma.
    \end{equation*}
We then study $(q^2-1)I'_\sigma(q^{\ell'})$, i.e.
    \begin{equation*}
    \begin{split}
        \abs{\mathrm{Disc}(v'_\sigma)}^{1/6}&\int_{\max\{\abs{x},\abs{y}\} = 1}q^{2\epsilon(v'_\sigma(x,y)\lambda_0)/3}\abs{v'_\sigma(x,y)}^{-2/3}dxdy
        \\&=q\abs{\mathrm{Disc}(v_\sigma)}^{1/6}\int_{\max\{\abs{x},\abs{y}\} = 1}q^{2\epsilon(\pi v_\sigma(\pi^{-1}x,y)\lambda_0)/3}\abs{v_\sigma(\pi^{-1}x,y)}^{-2/3}dxdy.
    \end{split}
    \end{equation*}
    We make the change of variables $x\mapsto \pi x$ and $y\mapsto y$ and rewrite the above as
    \begin{equation*}
        \abs{\mathrm{Disc}(v_\sigma)}^{1/6}\int_{\max\{\abs{\pi x},\abs{y}\} = 1}q^{2\epsilon(\pi v_\sigma(x,y)\lambda_0)/3}\abs{v_\sigma(x,y)}^{-2/3}dxdy.
    \end{equation*}
    We now split the integral into three regions. A first part where $\abs{y} = 1$ and $\abs{x}\leq 1$, a second where $\abs{y} \leq q^{-1}$ and $\abs{x} = q$ and a final one where $\abs{y} = 1$ and $\abs{x} = q$. In the second and third regions, we make the change of variables $x\mapsto \pi^{-1} x$, $y\mapsto \pi^{-1} y$. As $v_\sigma(\pi x,\pi y) = \pi^3v_\sigma(x,y)$ and $\epsilon$ is $3$-periodic with respect to the valuation of the argument, we end up with
    \begin{equation*}
        \abs{\mathrm{Disc}(v_\sigma)}^{1/6}\int_{\max\{\abs{x},\abs{y}\} = 1}q^{2\epsilon(\pi v_\sigma(x,y)\lambda_0)/3}\abs{v_\sigma(x,y)}^{-2/3}dxdy,
    \end{equation*}
    i.e. $(q^2-1)I'_\sigma(q^{\ell-4})$, by viewing $\pi\lambda_0$ as the new $\lambda_0$. Now, we simply note that $\ell-4 \equiv^3 \ell+2$.

    We now turn to the explicit evaluation of $I'_\sigma(q^\ell)$. When $\sigma$ corresponds to an unramified splitting type, we perform this evaluation for $\ell$ divisible by $4$. We begin with the totally split case $(111)$. As we saw before, there is a unique such Étale extension, and by studying the multiplication laws in the Levi--Delone--Faddeev correspondence, one checks that the form $v_\sigma = x(x+y)y$ corresponds to such an extension. The discriminant has absolute value equal to one, so that if $\ell = 4m$ we have $\abs{\lambda_0} = q^m$.

    Using symmetry, we rewrite $(q^2-1)I'_\sigma(q^\ell)$ as 
    \begin{equation}\label{totsplitint}
        2\int_{\abs{x} = 1}\int_{\abs{y} \leq  q^{-1}}q^{2\epsilon(x(x+y)y\lambda_0)/3}\abs{x(x+y)y}^{-2/3}dxdy + \int_{\abs{x} = 1}\int_{\abs{y} = 1}q^{2\epsilon(x(x+y)y\lambda_0)/3}\abs{x(x+y)y}^{-2/3}dxdy.
    \end{equation}
    In the first of these integrals $\abs{x(x+y)y} = \abs{y}$ as $\abs{x} = 1 > \abs{y}$. Hence, we obtain a contribution
    \begin{equation*}
        2(q-1)^2\int_{\abs{y} \leq q^{-1}}q^{2\epsilon(y\lambda_0)/3}\abs{y}^{-2/3}dy = 2(q-1)^2\sum_{k=1}^\infty q^{2\epsilon(\pi^k\lambda_0)/3}q^{-k/3}.
    \end{equation*}
    Separating depending on $k$ modulo three yields
    \begin{equation*}
    2(q-1)^2\sum_{i=1}^3 q^{2\epsilon(\pi^i\lambda_0)/3}q^{-i/3} \sum_{k=0}^\infty q^{-k} = 2q(q-1) \sum_{i=1}^3 q^{2\epsilon(\pi^i\lambda_0)/3}q^{-i/3}.
    \end{equation*}
    The second integral from \eqref{totsplitint} is
    \begin{equation*}
    \begin{split}
        \int_{\abs{x} = 1}\int_{\abs{y} = 1}q^{2\epsilon((x+y)\lambda_0)/3}\abs{(x+y)}^{-2/3}dxdy &= \int_{\abs{x} = 1}\sum_{k=1}^\infty \int_{y: \abs{y+x} = q^{-k}}q^{2\epsilon(\pi^k\lambda_0)/3}q^{2k/3}dxdy
        \\  + \int_{\abs{x} = 1}\int_{y: \abs{y+x} = 1, \abs{y} = 1}q^{2\epsilon(\lambda_0)/3}dxdy &= (q-1)^2\sum_{k=1}^\infty q^{2\epsilon(\pi^k\lambda_0)/3}q^{-k/3} + (q-1)(q-2)q^{2\epsilon(\lambda_0)/3}
        \\&= q(q-1)\sum_{i=1}^3 q^{2\epsilon(\pi^i\lambda_0)/3}q^{-i/3} +(q-1)(q-2)q^{2\epsilon(\lambda_0)/3}.
    \end{split}
    \end{equation*}

    Recalling the definition of the $\epsilon$-function above \eqref{belowepsdef}, one sees that
    \begin{equation}\label{firstcassum}
        \sum_{i=1}^3 q^{2\epsilon(\pi^{i-m})/3}q^{-i/3}=\sum_{i=1}^3 q^{2\epsilon(\pi^{i-\ell})/3}q^{-i/3} = \begin{cases}
            q+1+q^{-1}, \,\, &\ell\equiv^3 0 \\
            q^{-1/3} + q^{2/3}+ q^{-1/3},   \,\, &\ell\equiv^3 1,\\
            q^{1/3}+q^{-2/3}+q^{1/3},\,\, &\ell\equiv^3 2.
        \end{cases} = \begin{cases}
            q^{-1}(1+q+q^2), \,\, &\ell\equiv^3 0 \\
            q^{-1/3}(q+2),   \,\, &\ell\equiv^3 1\\
            q^{-2/3}(2q+1),\,\, &\ell\equiv^3 2.
        \end{cases}
    \end{equation}
    Furthermore,
    \begin{equation}\label{seccassum}
        q^{2\epsilon(\lambda_0)/3} = \begin{cases}
            1, \,\, &\ell\equiv^3 0 \\
            q^{2/3},\,\, &\ell\equiv^3 1, \\
            q^{4/3}   \,\, &\ell\equiv^3 2.
        \end{cases}
    \end{equation}
    This proves Proposition \ref{secevalprop} for $\sigma$ with splitting type $(111)$.

    We turn to the unramified quadratic case, $K_\infty \oplus K_\infty(\sqrt{\alpha})$, corresponding to the splitting type $(21)$. Then, one checks that the form $x(x^2-\alpha y^2)$ lies in the orbit corresponding to $\sigma$. We note that if the maximum of $\abs{x}$ and $\abs{y}$ is $1$, then $\abs{x^2-\alpha y^2} = 1$ so that the integral
    \begin{equation*}
        \int_{\max\{\abs{x},\abs{y}\} = 1}q^{2\epsilon(x\lambda_0)/3}\abs{x}^{-2/3}dxdy = q(q-1)q^{2\epsilon(\lambda_0)/3} + (q-1)^2\sum_{j=1}^\infty q^{2\epsilon(\pi^j\lambda_0)/3}q^{-j/3}.
    \end{equation*}
    By our earlier calculations, this is equal to
    \begin{equation*}
        q(q-1)q^{2\epsilon(\lambda_0)/3} + q(q-1)\sum_{i=1}^3 q^{2\epsilon(\pi^i\lambda_0)/3}q^{-i/3},
    \end{equation*}
    which proves the proposition for the splitting type $(21)$, using \eqref{firstcassum} and \eqref{seccassum}.

    We now turn to the case when $\sigma$ corresponds to the splitting type $(3)$, i.e. $\sigma$ corresponds to the unramified cubic extension of $K_\infty$. Then, if we let $x^3+cx+d$ be an irreducible polynomial over $\fq$, the form $v_\sigma = x^3+cxy^2+dy^3$ lies in $V(R)^\sigma$. As $x^3+cx+d$ is irreducible, one sees that $\abs{v_\sigma(x,y)} = 1$ in the set we are integrating over. We then have that
    \begin{equation*}
        q^{2\epsilon(\lambda_0)/3}\int_{\max\{\abs{x},\abs{y}\} = 1}1dxdy = q^{2\epsilon(\lambda_0)/3}(q^2-1),
    \end{equation*}
    which finishes the case $(3)$.

    Finally, we turn to the ramified cases. First, we claim that the two different orbits corresponding to the splitting type $(1^21)$ are represented by the forms $x(x^2-\alpha^r \pi y^2)$, with $r \in \{0,1\}$. As usual, this is quickly confirmed using the multiplication laws of the corresponding ring.

    We see that $\abs{x^2-\alpha^r \pi y^2} = \max\{\abs{x}^2,q^{-1}\abs{y}^2\}$, i.e. $1$ if $\abs{x}= 1$ and $\abs{y}\leq 1$, and $q^{-1}$ if $\abs{x} < 1$ and $\abs{y}=1$. Hence, the integral we should evaluate is
    \begin{equation*}
    \begin{split}
        \int_{\abs{x} = 1, \abs{y} \leq 1}&q^{2\epsilon(\lambda_0)/3}dxdy + q^{2/3}\int_{\abs{x} < 1, \abs{y} = 1}q^{2\epsilon(x\pi \lambda_0)/3}\abs{x}^{-2/3}dxdy
        =q^{2\epsilon(\lambda_0)/3}q(q-1)
        +q^{2/3}(q-1)^2\sum_{j=1}^\infty q^{2\epsilon(\pi^j\pi \lambda_0)/3}q^{-j/3} 
        \\=\,&  q^{2\epsilon(\lambda_0)/3}q(q-1)+q^2(q-1)\sum_{k=2}^4q^{2\epsilon(\pi^k \lambda_0)/3}q^{-k/3}
        \\=\,&q(q-1)\left(2q^{2\epsilon(\lambda_0)/3} + q^{2\epsilon(\pi\lambda_0)/3}q^{-1/3}+q^{2\epsilon(\pi^2\lambda_0)/3}q^{1/3}\right).
    \end{split}
    \end{equation*}
    Now, $\abs{\mathrm{Disc}(x(x^2-\alpha^r \pi y^2))} = \abs{\pi}^3 = q^{-3}$. Thus, if $\ell = 4m+1$, we see that $\abs{\lambda_0} = q^{m+1}$. In particular, $-v_\infty(\lambda_0) = m+1 \equiv^3 \ell$. To finish the partially ramified case, we compute
    \begin{equation*}
        2q^{2\epsilon(\lambda_0)/3} + q^{2\epsilon(\pi\lambda_0)/3}q^{-1/3}+q^{2\epsilon(\pi^2\lambda_0)/3}q^{1/3} = \begin{cases}
            2(q+1), &\ell \equiv^3 0, \\
            q^{-1/3}(q+1)^2, &\ell \equiv^3 1, \\
            2q^{1/3}(q+1), &\ell \equiv^3 2.
        \end{cases}
    \end{equation*}

    The last case to consider is the totally ramified case. Then, the orbits are represented by the forms of types $x^3-\beta^i\pi y^3$, $i\in\{0,1,2\}$ and with $\beta$ a non-cube in $\fq$. In the case that $\fq$ contains no non-cube, we consider the above form only with $i=0$. The absolute value of the discriminant of these forms is $\abs{\pi^2}= q^{-2}$. 
    
    Note that $\abs{x^3-\beta^i\pi y^3} = 1$ if $\abs{x} = 1$, and else it equals $q^{-1}$. Hence, the integral of interest is
    \begin{equation*}
        \int_{\abs{x} = 1, \abs{y} \leq 1}q^{2\epsilon(\lambda_0)/3}dxdy + q^{2/3}\int_{\abs{x} < 1, \abs{y} = 1}q^{2\epsilon(\pi\lambda_0)/3}dxdy = (q-1)\left(q^{2\epsilon(\lambda_0)/3}q+q^{2/3}q^{2\epsilon(\pi\lambda_0)/3}\right).
    \end{equation*}
        Now, if $\ell = 4m+2$, we have that $\abs{\lambda_0} = q^{m+1}$ so that $-v_\infty(\lambda_0) \equiv^3 m+1 \equiv^3 \ell - 1$. Hence, the above integral is equal to
    \begin{equation*}
        (q^2-1)\cdot
        \begin{cases}
            q, &-v_\infty(\lambda_0) \equiv^3 0, \text{ i.e. } \ell \equiv^3 1\\
            q^{2/3},&-v_\infty(\lambda_0) \equiv^3 1, \text{ i.e. } \ell \equiv^3 2\\
            q^{4/3},&-v_\infty(\lambda_0) \equiv^3 2 \text{ i.e. } \ell \equiv^3 0.
        \end{cases}
    \end{equation*}
    Multiplying by the sixth root $\abs{\mathrm{Disc}(v_\sigma)}^{1/6}=q^{-1/3}$ finishes the proof.
\end{proof}

Combining Proposition \ref{implicitformcount}, Proposition \ref{secevalprop}, Lemma \ref{irrformcount}, \eqref{stepformcount} and \eqref{tocmplres}, we see that we have proven Theorem \ref{explicitformcount}.
\section{Sieving for maximality}\label{fieldchapter}
We now show how to obtain asymptotic formulas for the number of maximal cubic forms by employing a certain discriminant-reducing sieve, used in \cite[Section 9]{BST}. By taking advantage of the function field setting, as well as precise results for the evaluation of various Fourier transforms, cf.~\cite{BTT}, we obtain an error term with the same quality as the best result for number fields, while keeping our argument mostly elementary.

The starting point of the calculations is the following simple inclusion-exclusion sieve. Let $U$ denote the set of all $R$-integral maximal binary cubic forms, and let $W_F$ denote the set of forms in $V(R)$ which are nonmaximal at every prime dividing the monic squarefree polynomial $F$. Letting $\mu$ denote the Möbius function, we have that
\begin{equation*}
    N(U\cap V(R)^\sigma; X) = \sum_{F}\mu(F)N\left(W_F\cap V(R)^\sigma; X\right),
\end{equation*}
by the inclusion-exclusion principle, where the sum runs over monic squarefree polynomials. When $\abs{F}$ is large, we estimate the summand above using the bound
\begin{equation}\label{cutoffBound}
    N\left(W_F\cap V(R)^\sigma; X\right) \ll \frac{X}{\abs{F}^{2-\epsilon}},
\end{equation}
proven in the same way as \cite[Lemma 34]{BST}. This also shows the absolute convergence of the above sum. Using this bound, we may estimate
\begin{equation}\label{cutoffFieldSum}
    N(U\cap V(R)^\sigma; X) = \sum_{\abs{F}\leq X^\delta}\mu(F)N\left(W_F\cap V(R)^\sigma; X\right) + \mathcal{O}\left(X^{1-\delta+\epsilon}\right), 
\end{equation}
for any $0 <\delta < 1$.

To handle $F$ of smaller degree, we rewrite 
\begin{equation*}
    N\left(W_F\cap V(R)^\sigma; X\right) = \sum_\sumstack{ fg \mid F \\ \alpha \in \mathbb{P}^1(R/fgR)} \mu(g)N\left(V_{fg,\alpha}; X \abs{f}^2/\abs{F}^4\right),
\end{equation*}
see \cite[Proposition 33, Eq. (70)]{BST} and \cite[Lemma 6.2]{SST}. Here $V_{fg,\alpha}$ denotes the forms in $V(R)$ whose reduction modulo every $P\mid fg$ has a root at $\alpha$. One proves this by using Lemmas \ref{subringformlemma} and \ref{overringlemma} to count pairs $(A,A')$ of cubic rings $A,A'$ with $A'$ being a $P$-overring of $A$. If we let $\omega_{fg}(x)$ denote the number of roots of the form $x$ modulo $fg$, then we see that the above is equal to
\begin{equation*}
     \sum_{fg\mid F}\mu(g)\sum_\sumstack{x\in V(R) \\ \abs{\mathrm{Disc}(x)} = X\abs{f}^2/\abs{F}^4} \omega_{fg}(x).
\end{equation*}
By a slight modification of the computations proving \eqref{intcountcoords}, we find that the innermost of these sums is equal to
\begin{equation}\label{rootcountint}
    \frac{\nu(K)}{\vol(G_0)}\int_{g\in \Lambda_{X\abs{f}^2/\abs{F}^4}S}\eta(\abs{t}) \left(\sum_{x\in gB \cap V(R)}\omega_{fg}(x)\right)\abs{t}^{-3}\abs{\lambda}^{-1}dtd\lambda df.
\end{equation}
\subsection{Fourier analysis and geometry of numbers}

We study the integrand in \eqref{rootcountint} by using finite Fourier analysis. For this, we will need an exact formula for the Fourier transform of $\omega_{fg}$, obtained in \cite{TTOrb}. Our argument is quite similar to \cite[Section 7]{BTT}, but we avoid using the theory of Shintani zeta functions.

Let $P$ be a prime polynomial. The character group of $(R/PR)^4$ can be parametrised by binary cubic forms $y = (y_1, y_2,y_3,y_4)$ with coefficients satisfying $\deg(y_i) < \deg P$. Indeed, let us first define $\chi_\infty: K_\infty \to \fq$ by 
\begin{equation}\label{charnewpar}
\sum_{n=-N}^\infty a_n\pi^n \mapsto \exp\left(-\frac{2\pi i}{p}\mathrm{Tr}_{\fq/\mathbb{F}_p}\left(a_{1}\right)\right),
\end{equation}
with $p\mid q$, see \cite[Eq. (2.1)]{Frankenhuijsen}. We then set 
\begin{equation*}
    \chi_y(x) = \chi_\infty\left(\frac{[x,y]}{P}\right)
\end{equation*}
where $[x,y] = x_1y_1+x_2y_2/3+x_3y_3/3+x_4y_4$ is the bilinear form from \cite[Eq. (12)]{TTOrb}. This is slightly different from the parametrisation used in \cite{TTOrb}, but it has nice properties that we use when bounding the error term. Note that $\chi_y(x)$ splits into a product $\chi_{y_1}(x_1)\chi_{y_2}(x_2)\chi_{y_3}(x_3)\chi_{y_4}(x_4)$ where the $\chi_{y_i}$ are characters on $R/PR$.

The function $h\mapsto \widehat{\omega}_h$ is multiplicative, where the circumflex denotes Fourier transform. Furthermore, by \cite[Proposition 1]{TTOrb}, whose proof goes through without changes even with our parametrisation, we have that
\begin{equation}\label{rootFouriereval}
    \widehat{\omega}_P(y) := \widehat{\omega}_P(\chi_y) = \begin{cases}
        1 + \abs{P}^{-1}, &\text{ if $y = 0$ modulo $P$},\\
        \abs{P}^{-1}, &\text{ if $y \neq 0$ has a triple root modulo $P$},
        \\0,  &\text{ otherwise.}
    \end{cases}
\end{equation}
In particular, we need only consider $y$ with at least a triple root modulo all $P\mid fg$.

We now decompose $gB$ into unions of boxes as in \eqref{gBdecomp}. Writing $x^{(i)} = gv_i$ in the notation from \eqref{gBdecomp} we have that
\begin{equation*}
    gB = \bigcup_{i\leq i_0} \left(x^{(i)} + C_{\lambda,t}\right).
\end{equation*}
By Fourier inversion, we obtain
\begin{equation*}
    \sum_{x\in gB \cap V(R)}\omega_{fg}(x) = \sum_{i\leq i_0}\sum_{y=(y_1,y_2,y_3,y_4) \,\,\mathrm{ mod} fg}\chi_y(-x^{(i)})\widehat{\omega}_{fg}(y)\sum_{x\in C_{\lambda,t}}\chi_y(x).
\end{equation*}
Let us separate the modulus $fg$ into two parts and write $r\mid fg$, with $r$ chosen such that $y$ is zero modulo $fg/r$, but nonzero modulo any prime dividing $r$ while also having a triple root modulo any such prime. We say that $y$ is of type $r$. Then, by \eqref{rootFouriereval},
\begin{equation*}
    \widehat{\omega}_{fg}(y) = \abs{r}^{-1}\prod_{P\mid fg/r}\left(1+\abs{P}^{-1}\right) = \frac{\sigma(fg/r)}{\abs{fg/r}\abs{r}} = \frac{\sigma(fg)}{\abs{fg}\sigma(r)},
\end{equation*}
where $\sigma(h)$ is the polynomial sum-of-divisors function.

Thus, we obtain
\begin{equation}\label{rootsum}
    \sum_{x\in gB \cap V(R)}\omega_{fg}(x) = \frac{\sigma(fg)}{\abs{fg}}\sum_{r\mid fg}\sigma(r)^{-1}\sum_{i\leq i_0}\sum_\sumstack{y \,\,\mathrm{ mod}\,\, r \\ y  \text{ of type $r$}}\chi_y(-x^{(i)})\sum_{x\in C_{\lambda,t}}\chi_y(x).
\end{equation}
Recall that $C_{\lambda,t}$ is a box whose side lengths depend on $\abs{\lambda}$ and $\abs{t}$. In particular, the first coordinate of any $x\in C_{\lambda,t}$ satisfies $\abs{x_1} < c\abs{\lambda t^3}$. To study the expression above, we split into cases depending on the size of $c\abs{\lambda t^3}$.

Let us first assume that $c\abs{\lambda t^3} \geq \abs{r}$. Then, when summing over $x$, we sum over a full set of representatives of $x_1$ modulo $r$. In particular, the innermost sum is zero unless $y_1 = 0$. Now, $y$ has a triple root modulo $r$, which means that we must also have that $y_2 = y_3 = 0$. Recalling that $C_{\lambda,t}$ is a box, and writing $\chi_y = \chi_{y_1}\chi_{y_2}\chi_{y_3}\chi_{y_4}$, we see that
\begin{equation*}
    \sum_{x\in C_{\lambda,t}}\chi_y(x) = \#\{(x_1,x_2,x_3)\in R^3: (x_1,x_2,x_3,0)\in C_{\lambda,t}\}\sum_{\abs{x_4} < c\abs{\lambda/t^3}}\chi_{y_4}(x_4).
\end{equation*}
We are left with
\begin{equation*}
 \#\{(x_1,x_2,x_3)\in R^3: (x_1,x_2,x_3,0)\in C_{\lambda,t}\}\frac{\sigma(fg)}{\abs{fg}}\sum_{r\mid fg}\sigma(r)^{-1}\sum_{i\leq i_0}\sum_\sumstack{y_4 \,\,\mathrm{ mod}\,\, r \\ y  \text{ of type $r$}}\chi_{y_4}(-x^{(i)})\sum_{\abs{x_4} < c\abs{\lambda/t^3}}\chi_{y_4}(x_4).
\end{equation*}
At this point, we may apply Möbius inversion to rewrite the expression above, excluding the prefactors, as
\begin{equation*}
    \sum_{r\mid fg}\sigma(r)^{-1}\mu(r)\sum_{r'\mid r}\mu(r')\sum_{i\leq i_0}\sum_\sumstack{y_4 \,\,\mathrm{ mod}\,\, r' }\chi_{y_4}(-x^{(i)})\sum_{\abs{x_4} < c\abs{\lambda/t^3}}\chi_{y_4}(x_4).
\end{equation*}
Note that the set of $x_4$ satisfying the condition in the innermost sum is an additive group, say $G$, that projects down to $R/r'R$. Hence, we need only consider the contribution from $y_4$ such that $\chi_{y_4}$ is trivial on the projection of $G$. The set of such $y_4$ is a subgroup of $R/r'R$ of size $\abs{r'}/\abs{G}$, assuming $\abs{G} \leq \abs{r'}$, else the size is $1$. In fact, this is immediate from our definition of $\chi_{y_4}$ as one can see that the set of such $y_4$ is in fact simply those satisfying $\abs{y_4} < c^{-1}\abs{r' t^3/\lambda}$ (when $c\abs{\lambda/t^3}\geq 1$). Indeed, the easiest way to see this is to note that $\chi_{y_4}(x_4)$ equalling $1$ for every $x_4$ of degree zero means that $\abs{y_4} < \abs{r}q^{-1}$, as $\chi_\infty$ only picks out the coefficient associated to $T^{-1}$. Once this is shown, one iterates and arrives at the result.

Let us now consider the two innermost sums for a fixed $i$. We can split $-x^{(i)}$ into one component lying in $G$ and one component which either lies outside of $G$, or is zero. If the second component is nonzero, then the contribution from such an $i$ is zero after summing over $y_4$. Else, if this component is zero, we obtain a nonzero contribution. This happens precisely when the fourth coordinate of $x^{(i)}$ lies in $G$, i.e. when $x^{(i)}+C_{\lambda,t}$ contains a point whose last coordinate is divisible with $r'$.

We combine the argument above with slicing over the third and fourth coordinate in $gB$ and obtain that, cf. \eqref{firstSlice},
\begin{equation}\label{sliceFields}
\begin{split}
     \sum_{x\in gB \cap R}&\omega_{fg}(x) =  \frac{\sigma(fg)}{\abs{fg}}\sum_{r\mid fg}\sigma(r)^{-1}\mu(r)\sum_{r'\mid r}\abs{r'}\mu(r')\\&\times\left((q-1)\sum_{r'\mid d\in R\setminus\{0\}}' \#\{x\in (gB)_d\cap R^3\} + (q-1)\sum_{c\in R\setminus\{0\}}' \#\{x\in (gB)_{c,0}\cap R^2\}\right),
\end{split}
\end{equation}
where the $'$ denotes that the summation is restricted to monics. Note that the second of the two sums above is one whose integral we have already evaluated in the proof of Proposition \ref{implicitformcount}. We remark that if $c\abs{\lambda t^3} < \abs{r}$, we still have a contribution from $y$ of the form $(0,0,0,y_4)$ and this contribution is precisely what is stated above. The contribution from $y$ not of this form when $c\abs{\lambda t^3} < \abs{r}$ will be bounded in the next section. 

The sum over $d$ above is very similar to one we have already studied in the proof of Proposition \ref{implicitformcount}, except for the condition $r'\mid d$. We can write $d=r'd'$ to see that this sum is
\begin{equation*}
    (q-1)\sum'_{d'\in R\setminus\{0\}} \abs{\lambda}^3\abs{t}^3\vol(B_{d'r't^3/\lambda}).
\end{equation*}
When integrating we may replace $\lambda$ with $\lambda_0$, where $\lambda_0\in \Lambda_{X\abs{f}^2/\abs{F}^4}$, and see that the integral is
\begin{equation*}
\begin{split}
    (q-1)&\sum_{\epsilon = 0}^2\sum_\sumstack{d\in R\setminus\{0\}\\ \deg(d'r') \equiv^3 \epsilon}'\abs{\lambda_0}^3\int_{ \abs{t}\geq 1}\eta(\abs{t}) \vol(B_{d'r't^3/\lambda_0})dt \\&= \abs{r'}^{-1/3}\sum_{\epsilon = 0}^2\abs{\lambda_0}^{10/3}\int_\sumstack{\abs{u}\geq\abs{r'/\lambda_0} \\ -v_\infty(u\lambda_0)\equiv^3 \epsilon} \Bigg(\sum_\sumstack{d'\in R\setminus\{0\}\\ \deg(d'r') \equiv^3 \epsilon\\\abs{d'}\leq \abs{u\lambda_0}/\abs{r'}}'\eta(\abs{u\lambda_0/d'r'}^{1/3})\abs{d'}^{-1/3}\Bigg)\vol(B_u)\abs{u}^{-2/3}du.
\end{split}
\end{equation*}
The condition $\deg(d'r')\equiv^3 \epsilon$ can be rewritten as $\deg(d') \equiv^3 -v_\infty(u\lambda_0 / r')$. Hence, we may use our previous results for the evaluation of the sum, with $\lambda_0$ replaced by $\lambda_0/r'$, to see that the inner sum above is equal to
\begin{equation*}
    \frac{\abs{u\lambda_0/r'}^{2/3}q}{(q-1)(q^2-1)}-\frac{q^{2\epsilon(u\lambda_0/r')/3}}{(q-1)(q^2-1)}.
\end{equation*}
Here, the function $\epsilon$ is the same as in the proof of Proposition \ref{implicitformcount}. Integrating the two terms above and separating the tails yields
\begin{equation}\label{mainandtail}
\begin{split}
    \frac{q\abs{r'}^{-1}}{(q-1)(q^2-1)}&\abs{\lambda_0}^{4}\left(\vol(B)-\int_\sumstack{\abs{u}<\abs{r'/\lambda_0} } \vol(B_u)du \right) \\&- \frac{\abs{r'}^{-1/3}}{(q-1)(q^2-1)}\abs{\lambda_0}^{10/3}\left(I^\sigma_1(\lambda_0/r')-\int_\sumstack{\abs{u}< \abs{r'/\lambda_0} } \vol(B_u)q^{2\epsilon(u\lambda_0/r')/3}\abs{u}^{-2/3}du\right).
\end{split}
\end{equation}

We now consider the contribution from the tail integrals above. If $\abs{r'}$ is smaller than some small absolute constant multiplied by $\abs{\lambda_0}$, then we may replace $B_u$ with $B_0$ above, and then just as in Proposition \ref{implicitformcount} the tail integrals above will cancel against terms coming from the second sum in \eqref{sliceFields}. On the other hand, if $\abs{r'}\gg \abs{\lambda_0}$ then we also have that $\abs{fg}\gg \abs{r}\gg \abs{\lambda_0}$. Furthermore, we can replace all the tails with an error term of size $\ll \abs{\lambda_0}^3$. Hence, we may remove the tail terms from \eqref{mainandtail} at the cost of an error $\mathcal{O}\left(\abs{\lambda_0}^3 \mathbf{1}_{\{\abs{fg}\gg \abs{\lambda_0}\}}\right)$.

Recalling that $\nu(K)=q^2-1$, we see that \eqref{rootcountint} is
\begin{equation}\label{fieldMainError}
\begin{split}
     \frac{1}{(q-1)\vol(G_0)}&\frac{\sigma(fg)}{\abs{fg}}\sum_{r\mid fg}\sigma(r)^{-1}\mu(r)\sum_{r'\mid r}\abs{r'}\mu(r')
    \\& \times \Bigg( q\abs{r'}^{-1}\abs{\lambda_0}^{4}\vol(B) - \abs{r'}^{-1/3}\abs{\lambda_0}^{10/3}I^\sigma_1(\lambda_0/r') + q(q+1)\abs{\lambda_0}^4I_0^\sigma+\mathcal{O}\left(\abs{\lambda}^3 \mathbf{1}_{\{\abs{fg}\gg \abs{\lambda_0}\}}\right)\Bigg).
\end{split}
\end{equation}
We first study the contribution of the error term to the entire sum \eqref{cutoffFieldSum}. As $\abs{\lambda_0} \asymp X^{1/4}\abs{f}^{1/2}/\abs{F}$ and $\abs{\sigma(r)}\ll \abs{r}^{1+\epsilon}$, this contribution is
\begin{equation*}
    \ll X^\epsilon \sum_{\abs{F}\leq X^\delta}\sum_{fg\mid F}\sum_{r\mid fg}\sum_{r'\mid r}\frac{X^{3/4}\abs{f}^{3/2}}{\abs{F}^3}\mathbf{1}_{\{\abs{fg}\gg X^{1/4}\abs{f}^{1/2}/\abs{F}\}}\ll X^{3/4+\epsilon} \sum_{\abs{F}\leq X^\delta}\sum_{fg\mid F}\frac{\abs{f}^{3/2}}{\abs{F}^3}\mathbf{1}_{\{\abs{fg}\gg X^{1/4}\abs{f}^{1/2}/\abs{F}\}}.
\end{equation*}
Write $F=fgh$ so that the condition in the indicator function becomes $\abs{f}^{3/2}\abs{g}^2\abs{h} \gg X^{1/4}$. We can then bound the sum above by
\begin{equation}\label{lambdacubesumbound}
    X^{3/4+\epsilon}\sum_{\abs{h}\leq X^\delta}\abs{h}^{-3}\sum_{\abs{g}\leq X^\delta/\abs{h}}\abs{g}^{-3}\sum_{\abs{f}\gg X^{1/6}/(\abs{h}^{2/3}\abs{g}^{4/3})} \abs{f}^{-3/2} \ll X^{2/3+\epsilon}\sum_{\abs{h}\leq X^\delta}\abs{h}^{-8/3}\sum_{\abs{g}\leq X^\delta/\abs{h}}\abs{g}^{-7/3},
\end{equation}
which is $\ll X^{2/3+\epsilon}$. 

We now turn to the contribution of the non-error terms from \eqref{fieldMainError}. We begin with the first of the three terms and note that the only non-zero contribution to
\begin{equation*}
     \frac{q\vol(B)\abs{\lambda_0}^4}{(q-1)\vol(G_0)}\frac{\sigma(fg)}{\abs{fg}}\sum_{r\mid fg}\sigma(r)^{-1}\mu(r)\sum_{r'\mid r}\mu(r')
\end{equation*}
is when $r = 1$, in which case the contribution is
\begin{equation*}
    \frac{q\vol(B)\abs{\lambda_0}^4}{(q-1)\vol(G_0)}\frac{\sigma(fg)}{\abs{fg}} = \frac{q}{(q-1)\#\mathrm{Aut}(\sigma)}X \cdot \frac{\abs{f}^2}{\abs{F}^4}\cdot\frac{\sigma(fg)}{\abs{fg}},
\end{equation*}
i.e. the main term from Theorem \ref{explicitformcount} multiplied with $\sigma(fg)\abs{fg}^{-1}\abs{f}^2\abs{F}^{-4}$. The total contribution from this term to \eqref{cutoffFieldSum} is then
\begin{equation*}
    \frac{q}{(q-1)\#\mathrm{Aut}(\sigma)}X\left(\sum_{F} \frac{\mu(F)}{\abs{F}^4}\sum_{fg\mid F}\mu(g)\frac{\abs{f}\sigma(fg)}{\abs{g}}- \sum_{\abs{F}> X^\delta} \frac{\mu(F)}{\abs{F}^4}\sum_{fg\mid F}\mu(g)\frac{\abs{f}\sigma(fg)}{\abs{g}}\right).
\end{equation*}
The tail sum is 
\begin{equation*}
    \ll X\sum_{\abs{F} > X^{\delta}} \frac{1}{\abs{F}^{2-\epsilon}} \ll X^{1-\delta+\epsilon}.
\end{equation*}
By using multiplicativity, one sees that the main term contributes
\begin{equation}\label{prezeta}
    \frac{q}{(q-1)\#\mathrm{Aut}(\sigma)}X \prod_{P}\left(1-\abs{P}^{-2}-\abs{P}^{-3}+\abs{P}^{-5}\right) = \frac{q}{(q-1)\#\mathrm{Aut}(\sigma)}X \prod_{P}\left(1-\abs{P}^{-2}\right)\prod_{P}\left(1-\abs{P}^{-3}\right).
\end{equation}
If we let
\begin{equation*}
    \zeta_R(s) :=  \prod_{P}\left(1-\abs{P}^{-s}\right)^{-1} = \sum_{f}\abs{f}^{-s} = \frac{1}{1-q^{1-s}}
\end{equation*}
be the $R$-semilocal Riemann zeta function, then \eqref{prezeta} is
\begin{equation*}
     \frac{q}{(q-1)\zeta_R(2)\zeta_R(3)\#\mathrm{Aut}(\sigma)}X = \frac{q^2-1}{q^2\#\mathrm{Aut}(\sigma)}X.
\end{equation*}

We now study the contribution of the second term from \eqref{fieldMainError}. This equals
\begin{equation*}
\begin{split}
    - \frac{1}{(q-1)\vol(G_0)}&\frac{\sigma(fg)}{\abs{fg}}\sum_{r\mid fg}\sigma(r)^{-1}\mu(r)\sum_{r'\mid r}\abs{r'}^{2/3}\mu(r')\abs{\lambda_0}^{10/3}I^\sigma_1(\lambda_0/r') =- \frac{1}{(q-1)\#\mathrm{Aut}(\sigma)}\frac{\sigma(fg)}{\abs{fg}}\cdot \left(\frac{X\abs{f}^2}{\abs{F}^4}\right)^{5/6}
    \\&\times\sum_{r\mid fg}\sigma(r)^{-1}\mu(r)\sum_{r'\mid r}\abs{r'}^{2/3}\mu(r')C_2^\sigma\big(\ell+2\deg(f)-4\deg(F) -4\deg(r')\big).
\end{split}
\end{equation*}
Recall that $C_2^\sigma$ is given explicitly in Proposition \ref{secevalprop} and that its value depends only on the degree of its argument, modulo three. When summing the above over $fg\mid F$ and $F$, we may separate the tail part of the sum over $F$, as we did for the main term. The tail contribution is
\begin{equation*}
    \ll X^{5/6}\sum_{\abs{F} > X^\delta} \abs{F}^{-10/3}\sum_{fg\mid F} \abs{f}^{5/3}\sum_{r\mid fg}\sum_{r'\mid r}1\ll X^{5/6+\epsilon} \sum_{\abs{F} > X^\delta}\abs{F}^{-5/3} \ll X^{5/6-2\delta/3+\epsilon}.
\end{equation*}

Write $F=P_1...P_n$, $f=f_1...f_n$ and $g=g_1...g_n$ with $f_i = (f,P_i)$ and $g_i = (g,P_i)$. The contribution from the non-tail terms when summing over $fg\mid F$ is then
\begin{equation}\label{halfmultsum}
\begin{split}
    - X^{5/6}&\abs{F}^{-10/3} \frac{1}{(q-1)\#\mathrm{Aut}(\sigma)} \sum_{f_1g_1 \mid P_1} S_1\cdot ...\cdot \sum_{f_ng_n \mid P_n} S_n,
\end{split}
\end{equation}
where
\begin{equation*}
    S_i = \mu(g_i)\frac{\sigma(f_ig_i)\abs{f_i}^{5/3}}{\abs{f_ig_i}}\sum_{r_i\mid f_ig_i}\sigma(r_i)^{-1}\mu(r_i)\sum_{r'_i\mid r_i}\abs{r'_i}^{2/3}\mu(r'_i),
\end{equation*}
for $i < n$ and with $S_n$ also containing the factor $C_2^\sigma\big(\ell+2\deg(f)-4\deg(F) -4\deg(r')\big)$. Write
\begin{equation*}
\begin{split}
    \ell &+2\deg(f)-4\deg(F)-4\deg(r') = \bigg(\ell +2\deg(f/f_n)-4\deg(F/P_n)-4\deg(r'/r'_n)\bigg)\\&+ \bigg( 2\deg(f_n)-4\deg(P_n)-4\deg(r'_n)\bigg),
\end{split}
\end{equation*}
and let us temporarily write $\ell_n = \ell +2\deg(f/f_n)-4\deg(F/P_n)-4\deg(r'/r'_n)$. Adding the three terms coming from the condition $f_ng_n\mid P_n$ shows that
\begin{equation*}
    \abs{P_n}^{-10/3}\sum_{f_ng_n\mid P_n}S_n = \abs{P_n}^{-2}C_2^\sigma\big(\ell_n\big)+ \abs{P_n}^{-5/3}\left(1-\abs{P}^{-2}\right)C_2^\sigma\big(\ell_n-2\deg(P_n)\big).
\end{equation*}
Using induction, we see that \eqref{halfmultsum} equals
\begin{equation}\label{fieldinduct}
    - \frac{X^{5/6}}{(q-1)\#\mathrm{Aut}(\sigma)} \sum_{fg = F}C_2^\sigma\big(\ell - 2\deg(f)\big) \left(\prod_{P\mid f} \left(1-\abs{P}^{-2}\right)\abs{P}^{-5/3}\right)\prod_{P\mid g}\abs{P}^{-2}.
\end{equation}

Multiplying the above with $\mu(F)$ and summing over all squarefree $F$ we obtain
\begin{equation*}
\begin{split}
     - \frac{X^{5/6}}{(q-1)\#\mathrm{Aut}(\sigma)} \sum_{f}C_2^\sigma(\ell-2\deg(f))\mu(f)\left(\prod_{P\mid f} \left(1-\abs{P}^{-2}\right)\abs{P}^{-5/3}\right) \sum_{g: (g,f) = 1}\mu(g)\abs{g}^{-2}.
\end{split}
\end{equation*}
The innermost sum is 
\begin{equation*}
    \zeta_R(2)^{-1}\prod_{P\mid f}\left(1-\abs{P}^{-2}\right)^{-1},
\end{equation*}
so that the sum above equals
\begin{equation*}
    - \frac{X^{5/6}}{(q-1)\#\mathrm{Aut}(\sigma)\zeta_R(2)} \sum_{f}C_2^\sigma(\ell-2\deg(f))\mu(f)\abs{f}^{-5/3}.
\end{equation*}
Now, by studying the zeta function $1/\zeta_R(s)$ associated to the Möbius function, one obtains the well-known result for the sum over monic polynomials that
\begin{equation*}
    \sum_{\deg(f) = n}\mu(f) = \begin{cases}
        1, \,\, &n=0, \\
        -q, \,\, &n=1, \\
        0, \,\, &n\geq 2.
    \end{cases}
\end{equation*}
We thus finally arrive at
\begin{equation*}
    - \frac{X^{5/6}}{(q-1)\#\mathrm{Aut}(\sigma)\zeta_R(2)}  \cdot \left(C_2^\sigma(\ell)-q^{-2/3}C_2^\sigma(\ell-2)\right)
\end{equation*}
as the secondary term.

Finally, we investigate the contribution from the last non-error term in \eqref{fieldMainError}. Using multiplicativity, we see that this is
\begin{equation}
\begin{split}
     \frac{q(q+1)I_0^\sigma}{(q-1)\vol(B)\#\mathrm{Aut}(\sigma)}&X\frac{\abs{f}^2}{\abs{F}^4}\frac{\sigma(fg)}{\abs{fg}}\sum_{r\mid fg}\sigma(r)^{-1}\mu(r)\sum_{r'\mid r}\abs{r'}\mu(r')
     \\& = \frac{q(q+1)I_0^\sigma}{(q-1)\vol(B)\#\mathrm{Aut}(\sigma)}X\frac{\abs{f}\sigma(fg)}{\abs{F}^4\abs{g}}\sum_{r\mid fg}\sigma(r)^{-1}\Phi(r) = \frac{q(q+1)I_0^\sigma}{(q-1)\vol(B)\#\mathrm{Aut}(\sigma)}X\frac{\abs{f}^2}{\abs{F}^4} 2^{\tau'(fg)},
\end{split}
\end{equation}
where $\Phi$ is the Euler phi function and $\tau'$ is the number-of-prime-divisors function. Note that $2^{\tau'(fg)}$ is the number of divisors of $fg$, which is $\ll \abs{fg}^\epsilon$. Summing this over $fg\mid F$ and over $F$ against $\mu(F)\mu(g)$ yields (up to a tail term that can be absorbed into the error terms that we have already accumulated):
\begin{equation*}
    \frac{q(q+1)I_0^\sigma}{(q-1)\vol(B)\#\mathrm{Aut}(\sigma)}X\sum_{F}\mu(F)\abs{F}^{-4}\sum_{fg\mid F}\mu(g)\abs{f}^2 2^{\tau'(fg)} = \frac{q(q+1)I_0^\sigma}{(q-1)\vol(B)\#\mathrm{Aut}(\sigma)}X\sum_{F}\mu(F)\abs{F}^{-2}\prod_{P\mid F}\left(2-\abs{P}^{-2}\right).
\end{equation*}
Again, using multiplicativity, we see that this is
\begin{equation*}
    \frac{q(q+1)I_0^\sigma}{(q-1)\vol(B)\#\mathrm{Aut}(\sigma)}X\prod_{P}\left(1-\abs{P}^{-2}\right)^2 = \frac{q(q+1)I_0^\sigma}{(q-1)\vol(B)\#\mathrm{Aut}(\sigma)\zeta_R(2)^2}X = \frac{(q^2-1)I_0^\sigma}{q\vol(B)\#\mathrm{Aut}(\sigma)}X.
\end{equation*}
\subsection{Bounding the remaining error terms}\label{fielderrorsect}
We have now isolated main terms and secondary terms coming from characters associated to $y$ of the form $(0,0,0,y_4)$. We have also bounded the error term coming from such characters. It remains to study the contribution to \eqref{rootsum} when $c\abs{\lambda_0 t^3} < \abs{r}$ and $y$ does not have the form $(0,0,0,y_4)$. Note that $c\abs{\lambda_0t^3}< \abs{r}$ implies in particular that $\abs{\lambda_0}\ll \abs{fg}$.

Let us first begin with the case when $\abs{t}^3 \gg \abs{\lambda_0}$ so that the length of the side of $C_{\lambda_0,t}$ in the fourth coordinate direction is small. As $B$ contains no element with discriminant zero, we may assume that $\abs{t}\ll \abs{\lambda_0}$, as before. We can estimate
\begin{equation*}
    \sum_{x\in C_{\lambda_0,t}} 1 \ll \abs{\lambda_0}^3\abs{t}^3.
\end{equation*}

Next, we note that the summation over $x$ in \eqref{rootsum} restricts $y_1$ to one of $c^{-1}\abs{r/\lambda_0 t^3}$ congruence classes and $y_2$ to one of $c^{-1}\abs{r/\lambda_0 t}$ congruence classes. We thus have $\ll \abs{r^2/\lambda_0^2t^4}$ ways to pick $y_1$ and $y_2$ modulo $r$, with $y_1$ a unit in $R/rR$. As $y$ should have a triple root, determining $y_1$ and $y_2$ in this manner already determines the values of $y_3$ and $y_4$. 

Let us instead suppose that $y_1$ is zero modulo some $r'\mid r$ and a unit modulo $r'' := r/r'$, with $r'\neq r$. We must then pick $y_2,y_3$ congruent to zero modulo $r'$, and we have $\ll \abs{r'}$ choices for $y_4$ modulo $r'$. Note that for $r''$ we must have that $\abs{r''} \geq c\abs{\lambda_0 t^3}$, else $y_1$ would just be zero. Hence, modulo $r''$ we have $\ll \abs{r''^2/\lambda_0^2 t^4}$ choices. In total, this gives $\ll \abs{r^2/\lambda_0^2 t^4}$ choices for $y$, which is the same bound as in the other case. 

Multiplying this bound with $\abs{\lambda_0}^3\abs{t}^3$ yields $\abs{r}^2\abs{\lambda_0}\abs{t}^{-1}\ll \abs{r}^2\abs{t}^2$. Recall that we should integrate this against $dg$ up to $\abs{t}\ll \abs{\lambda_0}.$ Using also that $\abs{r}\leq \abs{F}$ we see that the result is that the contribution to \eqref{rootsum} of these terms is $\ll X^\epsilon \abs{F}$. Summing this up to $\abs{F}\leq X^\delta$ yields a term $\ll X^{2\delta+\epsilon}$. At this point, we have error terms containing $X$ to the powers $1-\delta, 5/6-2\delta/3, 2/3$ and $2\delta$. We see that these exponents are minimised when $\delta = 1/3$, in which case they are all bounded by $2/3$.

We now turn to the case when $\abs{t}^3\ll \abs{\lambda_0}$. We then estimate
\begin{equation*}
    \sum_{x\in C_{\lambda_0,t}} 1 \ll \abs{\lambda_0}^4.
\end{equation*}
To estimate the contribution from the sum over $y$ requires some more effort than above. The form $y$ having a triple root modulo $r$ implies that, in particular, $r^2\mid \mathrm{Disc}(y)$. We split the summation over $y$ into two cases, depending on whether $\mathrm{Disc}(y)$ is nonzero or not. 

We begin by estimating the contribution from forms whose discriminant vanishes, which means that they have a double root in $R$. If this double root is $[1:0]$, then using the indeterminates $w$ and $z$ instead of $x$ and $y$ to avoid confusion, $y$ has the form $y=ewz^2+dz^3$ with $e =re'$, which means that $e = 0$ as $\abs{y_3} =\abs{e} < \abs{r}$. We have already taken forms of this shape into account in a previous section, where they contributed to the non-error terms. On the other hand, if the double root is $[0:1]$, then $y$ has the form $aw^3$, which provides us with $\ll \abs{r/\lambda_0 t^3}$ choices for the coefficient. This gives us an error term $\ll \abs{\lambda_0}^3 \mathbf{1}_{\{\abs{fg}\gg \abs{\lambda_0}\}}$, which we have already bounded, see $\eqref{lambdacubesumbound}$ and the discussion leading up to it.

We now turn to the case where the root is $[\ell:1]$ for some nonzero $\ell$. Then, $y$ has the form $(aw+bz)(w-\ell z)^2 = aw^3+(b-2a\ell)w^2z+(a\ell^2-2b\ell)wz^2+b\ell^2z^3$ with $a\ell+b\equiv 0 \,\,\mathrm{mod}\,\, r$. Now, $\abs{b} \leq \abs{b\ell^2} < \abs{r}$. Therefore, $\abs{b} < \abs{r}$ and thus we see from examining the third coefficient that $\abs{a\ell} < \abs{r}$. Hence, we must in fact have that $b=-a\ell$ and the last coefficient is thus $-a\ell^3$. Now, the summation over $x$ restricts us to $\abs{a} < c^{-1}\abs{r/\lambda_0 t^3}$ and $\abs{a\ell^3 } < c^{-1}\abs{rt^3/\lambda_0}$. We can thus bound the number of forms by
\begin{equation*}
    \sum_{\abs{a} < c^{-1}\abs{r/\lambda_0 t^3}} \left(\frac{\abs{rt^3}}{\abs{a\lambda_0 }}\right)^{1/3} \ll \frac{\abs{r}}{\abs{\lambda_0 t}},
\end{equation*}
which also provides an acceptable error. This finishes our estimation of the contribution from $y$ with $\mathrm{Disc}(y) = 0$.

Now, we study the contribution from nondegenerate forms. Note that if $y$ has a triple root, then $r^2 \mid \mathrm{Disc}(y)$. The contribution to \eqref{rootcountint} is then
\begin{equation}\label{nondegencontrib}
\begin{split}
 \ll X^\epsilon \abs{\lambda_0}^4 \sum_{r\mid fg}\abs{r}^{-1}\int_{g\in \Lambda_{X\abs{f}^2/\abs{F}^4}S}\sum_\sumstack{\abs{y_k} < c^{-1}\abs{r/\lambda t^{j(k)}}\\ \mathrm{Disc}(y) \neq 0 \\ r^2 \mid \mathrm{Disc}(y)} 1 \,dg',
\end{split}
\end{equation}
where $j(1) = 3, j(2) = 1, j(3) = -1$ and $j(4) = -3$ and $dg' = \abs{t}^{-3}\abs{\lambda}^{-1}dfdtd\lambda$. Now, the integrand does not depend on $f$, which means that we only need to integrate over $\lambda$ and $t$. We may rewrite the condition $\abs{y_i} < c^{-1}\abs{r/\lambda t^{j(i)}}$ as $y\in r(\lambda a(t))^{-1}C^{-1}$, where $C^{-1} := \{(x_1,x_2,x_3,x_4): \abs{x_i} < c^{-1}\}$. The integral is then
\begin{equation*}
   \ll X^\epsilon \abs{\lambda_0}^4 \sum_{r\mid fg}\abs{r}^{-1}\sum_\sumstack{ \mathrm{Disc}(y) \neq 0\\r^2 \mid \mathrm{Disc}(y)}\int_{\lambda,t} \mathbf{1}_{r^{-1}\lambda a(t) y\in C^{-1}} \abs{t}^{-3}\abs{\lambda}^{-1}d\lambda dt.
\end{equation*}
At the cost of constant factors, we may extend the integration back over the $f$ and $k$ such that $\lambda a(t)n(f)k$ lies in the fundamental domain. Here we use that $C^{-1}$ is open. Next, we may restrict the summation over $y$ to one representative from each $\GLtwo(R)$-orbit as long as we extend the integration to all of $\GLtwo(K_\infty)_{X\abs{f}^2/\abs{F}^4}$. We may also add the condition $\abs{\mathrm{Disc}(y)} \leq \abs{r}^4$ to the sum over $y$. We can then bound the above as
\begin{equation*}
    \ll X^\epsilon \abs{\lambda_0}^4 \sum_{r\mid fg}\abs{r}^{-1}\sum_\sumstack{y\in \GLtwo(R)\setminus V(R)\\ \mathrm{Disc}(y) \neq 0 \\ \abs{\mathrm{Disc(y)}} \leq \abs{r}^4\\ r^2 \mid \mathrm{Disc}(y)}\int_{g\in \GLtwo(K_\infty)_{X\abs{f}^2/\abs{F}^4}} \mathbf{1}_{r^{-1}g y\in C^{-1}} dg.
\end{equation*}
Now, the discriminant of $r^{-1}gy$ has absolute value $\asymp \abs{\mathrm{Disc}(y)}\abs{\lambda_0}^4/\abs{r}^4$. Recalling that changing the integral over $\GLtwo(K_\infty)$ to one over $V(R)$ changes $dg$ to $\abs{\mathrm{Disc}(r^{-1}gy)}^{-1}dx$, cf. \cite[Prop. 23]{BST}, we obtain a bound
\begin{equation*}
    \ll X^\epsilon \sum_{r\mid fg}\abs{r}^{3}\sum_\sumstack{y\in \GLtwo(R)\setminus V(R)\\ \mathrm{Disc}(y) \neq 0\\ \abs{\mathrm{Disc(y)} }\leq \abs{r}^4\\r^2 \mid \mathrm{Disc}(y)}\frac{1}{\abs{\mathrm{Disc}(y)}}\int_{ C^{-1}} 1 dx.
\end{equation*}
We now use the function field analogue of \cite[Proposition 4.5]{BTT} asserting that the number of orbits of forms with $r^2\mid \mathrm{Disc}(y)$ with $\abs{\mathrm{Disc}(y)} = Y$ is $\ll Y/\abs{r}^{2-\epsilon}$. Applying this result, we can immediately bound the above by 
\begin{equation*}
    X^\epsilon \sum_{r\mid fg}\abs{r} \ll X^\epsilon \abs{fg}\ll  X^\epsilon \abs{F}.
\end{equation*}
Summing this over $fg\mid F$ and $\abs{F}\leq X^\delta$, using $\delta = 2/3$, yields an error term $\ll X^{2/3+\epsilon}$, which finishes our analysis of the error term. Indeed, we have now proven the following result.
\begin{prop}\label{maxringcount}
    We have that
    \begin{equation*}
        N(U\cap V(R)^\sigma; X) = \frac{q^2-1}{q^2\#\mathrm{Aut}(\sigma)}X- \frac{X^{5/6}}{q\#\mathrm{Aut}(\sigma)}  \cdot \left(C_2^\sigma(\ell)-q^{-2/3}C_2^\sigma(\ell-2)\right) + \frac{(q^2-1)I_0^\sigma}{q\vol(B)\#\mathrm{Aut}(\sigma)}X + \mathcal{O}\left(X^{2/3+\epsilon}\right),
    \end{equation*}
    where the forms $v$ in the left-hand side are counted with weight $\abs{\mathrm{Stab}_{\GLtwo(R)}(v)}^{-1}$.
\end{prop}
\subsection{Counting cubic fields}
Proposition \ref{maxringcount} allows us to count maximal cubic rings. Similarly to Lemma \ref{irrformcount}, we may restrict to maximal cubic orders at the cost of replacing the third term above with $\mathcal{O}\left(X^{3/4+\epsilon}\right)$. Aside from $R\oplus R\oplus R$, the maximal rings that are not cubic orders correspond to the sum of $R$ and the integral closure of $R$ inside some quadratic field. Now, quadratic fields with prescribed behaviour at $P_\infty$ can be studied directly by studying certain squarefree polynomials. One obtains a main term of order $X$ and an error term that is smaller than $\mathcal{O}(X^{2/3+\epsilon})$, cf. \cite[Lemma 8.1]{BTT}. By combining these two results, we can bootstrap Proposition \ref{maxringcount} to an estimate for the counting function of $U\cap V(R)^{(\sigma, \text{ irr})}$ by simply removing the third term, while keeping the same error term.  

We now count the number of cubic fields with discriminant $Y=q^{M}$, $2\mid M$, by summing over the various isomorphism classes $\sigma$ of the completion at $P_\infty$. If $\sigma$ is unramified, then there is no contribution to the global discriminant from $P_\infty$, so that the semilocal discriminant $X$ and the global discriminant agree. However, if $\sigma$ has splitting type $(1^21)$, then the semilocal discriminant is $X = q^{M-1}$ and if $\sigma$ has splitting type $(1^3)$, then the semilocal discriminant is $X = q^{M-2}$. We write $X(\sigma,Y)$ for $X$ as a function of $\sigma$ and $Y$.

Summing over $\sigma$, and using the computations of $\#\mathrm{Aut}(\sigma)$ from Section \ref{autch}, we find that
\begin{equation*}
    \sum_{\sigma}\frac{X(Y,\sigma)}{\#\mathrm{Aut}(\sigma)} = \left(\frac{1}{6} + \frac{1}{2} + \frac{1}{3}\right)Y + \left(2\cdot \frac{1}{2}\right)Yq^{-1} + Yq^{-2} = \left(1+q^{-1}+q^{-2}\right)Y = \frac{q^3-1}{q^2(q-1)}Y,
\end{equation*}
which allows us to write down the main term for the total number of cubic fields of discriminant $Y$. For the secondary term, we have from Proposition \ref{secevalprop} that $C_2^\sigma(\ell)-q^{-2/3}C_2^\sigma(\ell-2)$ is given by the following table.
\begin{center}
\begin{tabular}{ |c|c|c|c| } 
 \hline
 Type of $\sigma$ & $\ell \equiv^3 0$ & $\ell \equiv^3 1$ & $\ell \equiv^3 2$ \\ 
 \hline
 $(111)$ & $3(q-1)$ & $3q^{-1/3}(q-1)$ & $q^{-2/3}(q^2-1)$ \\ 
 $(21)$ & $q-1$ & $q^{-1/3}(q-1)$  & $q^{-2/3}(q^2-1)$\\ 
 $(3)$ & $0$ & $0$  & $q^{-2/3}(q^2-1)$ \\
 $(1^21)$ & $q^{-1/2}(q-1)$& $q^{1/6}(q-1)$  & $2q^{-1/6}(q-1)$ \\
 $(1^3)$
 & $q-1$ & $q^{-1/3}(q-1)$ & $0$ \\
 \hline
\end{tabular}
\end{center}
Hence,
\begin{equation*}
    \sum_{\sigma}\frac{\big(X(Y,\sigma)\big)^{5/6}}{\#\mathrm{Aut}(\sigma)}\left(C_2^\sigma(\ell)-q^{-2/3}C_2^\sigma(\ell-2)\right) = Y^{5/6}(q-1)\cdot \begin{cases}
    q^{-2}(q+1)^2,    &M \equiv^3 0,\\
    q^{-4/3}(q+1),       &M \equiv^3 1,\\
    q^{-5/3}(q+1)^2, &M \equiv^3 2,
    \end{cases}
\end{equation*}
where $\ell = M$ for unramified $\sigma$, $\ell = M-1$ for $\sigma$ of type $(1^21)$ while $\ell = M-2$ for $\sigma$ of type $(1^3)$. This completes the proof of Theorem \ref{allfieldthm}.

\section{The one-level density}\label{onelevch}
Before proving Theorem \ref{allfieldSplitCondThm}, concerning the number of cubic fields with splitting conditions on finitely many primes, we showcase an application of our results. Specifically, we show how to obtain a lower bound for the error term in \eqref{corspliteq} and, by extension, for the error term in Theorem \ref{allfieldSplitCondThm}. This is accomplished by studying low-lying zeros of certain $L$-functions using the so-called one-level density. We use the same methods as in \cite[Theorem 1.1]{CFLS}, but as the GRH is a theorem in our setting, our results are unconditional.

In this section, we will only use the one-level density as a tool for obtaining Theorem \ref{omegaresultthm}. However, we remark that the one-level density is an interesting object in its own right and has been studied for many different families of $L$-functions by several authors, see e.g. \cite{Hughes-Rudnick}, \cite{ILS}, \cite{Ozluk-Snyder}, \cite{Rudnick}, \cite{Young}. The one-level density associated with the Dedekind zeta functions of cubic fields over $\mathbb{Q}$ was first studied by \cite{Yang}, and later by \cite{CK}, \cite{SST1} and \cite{CFLS}. 

Conjecturally, see e.g. \cite{Katz-Sarnak}, the main term of the one-level density is governed by the so-called symmetry type of the $L$-functions under consideration. For $L$-functions associated to cubic number fields over $\mathbb{Q}$, this symmetry type is known to be symplectic, from work of \cite{Yang}. Our computations in this section show that this is also the case over $\fq(T)$, as expected.

\subsection{Preliminaries for the one-level density}
We let $\mathcal{F}(Y)$ denote the set of cubic $S_3$-extensions of $K=\fq(T)$ with discriminant equal to $Y=q^M$, with $M$ even. As there are very few $C_3$-cubic fields, one could include these and obtain similar results. As usual, we require that $2,3\nmid q$.

We now use $\mathcal{F}(Y)$ to parametrise certain $L$-functions. Specifically, given $L\in \mathcal{F}(Y)$, we let $\mathcal{L}(s;L)$ denote a certain Artin $L$-function associated to $L$, namely the quotient of the (global) Dedekind zeta function $$\zeta_L(s) = \prod_{\mathfrak{P}}\left(1-q^{-s\deg(\mathfrak{P})}\right)^{-1}$$ by the (global) Riemann zeta function $$\zeta_K(s) = \prod_{P}\left(1-q^{-s\deg(P)}\right)^{-1} = \frac{1}{(1-q^{-s})(1-q^{1-s})}.$$ Here, the products range over all primes, including the infinite ones, in $L$ and $K$ respectively. Now, by \cite[Theorem 5.9]{Rosen}, if one makes the change of variables $u=q^{-s}$ after dividing and writes $P_L(u)$ for the result, then this is a polynomial in $u$. The degree of this polynomial is $2g$, where $g$ is the genus of the extension. The Riemann--Hurwitz formula connects the discriminant with the genus and asserts that
\begin{equation*}
    2g-2 = 3 \cdot (-2)+M \iff 2g+4 = M,
\end{equation*}
see \cite[Theorem 7.16]{Rosen}.

Now, the Riemann Hypothesis, which is a theorem over function fields, states that every root of $P_L(u)$ has absolute value equal to $q^{-1/2}$. Hence, all roots have the form $q^{-1/2}e^{i\theta}$. Using this observation, we fix a real even Schwartz function $\psi$ and let
\begin{equation*}
    D_L(\psi) = \sum_{\theta}\psi\left(2g_L\frac{\theta}{2\pi}\right),
\end{equation*}
where the sum ranges over all $\theta$, counted with multiplicity, such that $q^{-1/2}e^{i\theta}$ is a root of $P_L(u)$. Here $g_L$ is the genus of $L$. Sometimes we write $N_L = 2g_L$ for simplicity. Note that $N_L$ only depends on $Y$ and not on the specific choice of cubic extension $L$. The one-level density is then defined as
\begin{equation*}
    \frac{1}{\#\mathcal{F}(Y)}\sum_{L\in \mathcal
F(Y)}D_L(\psi).
\end{equation*}
\subsection{The explicit formula}
The one-level density is often studied through a so-called explicit formula. In the function field setting, this explicit formula is obtained through Poisson summation. We write the roots of $P_L(u)$ as $q^{-1/2}e^{i\theta_j}$ with $1\leq j\leq N_L$. Then,
\begin{equation*}
    D_L(\psi) = \sum_{j=1}^{N_L}\sum_{n\in \mathbb{Z}}\psi\left(N_L\frac{\theta_j+2\pi n}{2\pi}\right).
\end{equation*}
Now, the Fourier transform of
\begin{equation*}
    f(x) = \psi\left(\frac{N_L\theta_j}{2\pi } + N_Lx\right)
\end{equation*}
is 
\begin{equation*}
    \widehat{f}(\xi) = \frac{1}{N_L}\widehat{\psi}\left(\frac{\xi}{N_L}\right)e^{i\theta_j\xi},
\end{equation*}
so that Poisson summation shows that
\begin{equation*}
    D_L(\psi) = \frac{1}{N_L}\sum_{n\in \mathbb{Z}}\widehat{\psi}\left(\frac{n}{N_L}\right)\left(\sum_{j=1}^{N_L}e^{i\theta_j n}\right).
\end{equation*}
Denoting the innermost sum by $c^L_n$, we note that the functional equation implies that $c^L_n = c^L_{-n}$, see \cite[p.55]{Rosen}. Hence, we may restrict our attention to $n\geq 0$, recalling that $\psi$ is real and even.

The $c_n^L$ appear naturally in the logarithmic derivative of $P_L(u)$. Indeed, write 
\begin{equation*}
    P_L(u) = \prod_{i=1}^{N_L}\left(1-\pi_i u\right),
\end{equation*}
where the $\pi_i$ are the inverse roots. Then, we see that, for small enough $\abs{u}$,
\begin{equation*}
    \frac{P'_L}{P_L}(u) = -\frac{1}{u}\sum_{n=1}^\infty u^n\sum_{i=1}^{N_L} \pi_i^n.
\end{equation*}
As $\pi_i = q^{1/2}e^{-i\theta_i}$, we see that
\begin{equation*}
    \frac{P'_L}{P_L}(u) = -\frac{1}{u}\sum_{n=1}^\infty q^{n/2}c^L_{n}u^n.
\end{equation*}

We also note that
\begin{equation}\label{firstlogder}
    \frac{\zeta_L'}{\zeta_L}(u) = \frac{1}{1-u} + \frac{q}{1-qu} + \frac{P'_L}{P_L}(u) = \frac{1}{u}\sum_{n=1}^\infty\left(q^n+1 -q^{n/2}c^L_{n}\right)u^n.
\end{equation}
On the other hand,
\begin{equation*}
    \zeta_L(u) = \prod_{P} \prod_{\mathfrak{P}\mid P}\left(1-u^{\deg\mathfrak{P}}\right)^{-1}.
\end{equation*}
Above, $P$ ranges over all primes, including $P_\infty$. Hence, for the logarithmic derivative, we find that
\begin{equation}\label{seclogder}
    \frac{\zeta_L'}{\zeta_L}(u) = \frac{1}{u}\sum_{P}\sum_{\mathfrak{P}\mid P}\sum_{k=1}^\infty u^{k\deg \mathfrak{P}}\deg\mathfrak{P}.
\end{equation}
Combining \eqref{firstlogder} and \eqref{seclogder}, we conclude that
\begin{equation}\label{traceeq}
    c^L_n  = q^{-n/2}\Bigg(q^n + 1 -\sum_{P}\deg P\sum_\sumstack{\mathfrak{P}\mid P \\ \deg\mathfrak{P} \mid n} f(\mathfrak{P}/P)\Bigg),
\end{equation}
where $f(\mathfrak{P}/P)$ is the inertial degree, see Section \ref{prelch}. An equality of the type \eqref{traceeq} is often referred to as an "explicit formula".
\subsection{Averaging over $L$}
We now apply \eqref{traceeq} to the computation of the one-level density and conclude that
\begin{equation}\label{onelevexpr}
    \frac{1}{\#\mathcal{F}(Y)}\sum_{L\in \mathcal{F}(Y)}D_L(\psi) = \widehat{\psi}(0) + \frac{2}{N_L\#\mathcal{F}(Y)}\sum_{n\geq 1}\widehat{\psi}\left(\frac{n}{N_L}\right)\Bigg(\bigg(q^{n/2}+q^{-n/2}\bigg) \#\mathcal{F}(Y)-q^{-n/2}\sum_{P}\deg P\sum_{L\in \mathcal{F}(Y)}\sum_\sumstack{\mathfrak{P}\mid P \\ \deg\mathfrak{P} \mid n} f(\mathfrak{P}/P)\Bigg).
\end{equation}
Now,
\begin{equation}\label{sumdegsplit}
    \sum_{\mathfrak{P}\mid P}\mathbf{1}_{\deg \mathfrak{P}\mid n}f(\mathfrak{P}/P) = \delta_{\deg P \mid n} \cdot\begin{cases}
        3,& \text{ for $P$ of type $(111)$},
        \\1+2\delta_{2\deg P\mid n},& \text{ for $P$ of type $(21)$},
        \\3\delta_{3\deg P\mid n},& \text{ for $P$ of type $(3)$},
        \\2,& \text{ for $P$ of type $(1^21)$},
        \\1,&\text{ for $P$ of type $(1^3)$}.
    \end{cases}
\end{equation}
If $S$ is one of the five splitting types above, we write $\mathcal{F}_{P,S}(Y)$ for the subcollection of $\mathcal{F}(Y)$ where $P$ splits according to $S$.

We note that by the Polynomial Prime Number Theorem, \cite[Proposition 2.1]{Rosen}, we have that
\begin{equation*}
    q^{-n/2}\sum_{\deg P \mid n} \deg P = q^{n/2} + q^{-n/2},
\end{equation*}
where we also used that $P_\infty$ has degree one. Thus, by adding and
subtracting $1$ from \eqref{sumdegsplit}, we find that
\begin{equation}\label{onelevtermsplit}
\begin{split}
    \bigg(q^{n/2}+&q^{-n/2}\bigg) \#\mathcal{F}(Y)-q^{-n/2}\sum_{P}\deg P\sum_{L\in \mathcal{F}(Y)}\sum_\sumstack{\mathfrak{P}\mid P \\ \deg\mathfrak{P} \mid n} f(\mathfrak{P}/P) \\& =-q^{-n/2}\sum_{\deg P \mid n}\deg P\big(2\#\mathcal{F}_{P,(111)}(Y)+2\delta_{2\deg P \mid n}\#\mathcal{F}_{P,(21)}(Y) + (3\delta_{3\deg P \mid n}-1)\#\mathcal{F}_{P,(3)}(Y)+\#\mathcal{F}_{P,(1^21)}(Y)\big).
\end{split}
\end{equation}
Having made this observation, we are ready to prove Theorem \ref{omegaresultthm}.
\begin{proof}[Proof of Theorem \ref{omegaresultthm}]
    Our standing assumption throughout the proof is that
    \begin{equation}\label{splitfamasymp}
        \#\mathcal{F}_{P,S}(Y) = C_{1,P,S}Y + C_{2,P,S}Y^{5/6} + \mathcal{O}\left(Y^{\theta+\epsilon} \abs{P}^\omega\right),
    \end{equation}
    with $C_{1,P,S}$ as in \eqref{firstsplitcoeffeq} and $C_{2,P,S}$ as given in Theorem \ref{allfieldSplitCondThm}. The main idea is to use this relation to study the one-level density when $\psi$ is a real even Schwartz function whose Fourier transform is supported in $[-\sigma,\sigma]$ for some $\sigma > 0$. All error terms below are allowed to depend on $\psi$.

    Starting from \eqref{splitfamasymp}, we begin by bounding the error-term contribution to \eqref{onelevtermsplit} as being
    \begin{equation*}
        \ll Y^{\theta+\epsilon}q^{-n/2}\sum_{\deg P\mid n}(\deg P) \abs{P}^\omega \ll   Y^{\theta+\epsilon}q^{n(\omega + 1/2)}.
    \end{equation*}
    Hence, using that $\widehat{\psi}\left(n/N_L\right)$ is zero unless $\abs{n}\leq \sigma N_L$, we may bound the error-term contribution to \eqref{onelevexpr} as
    \begin{equation*}
        \ll \frac{Y^{\theta+\epsilon-1} }{N_L}\cdot q^{\sigma N_L(\omega + 1/2)} \ll Y^{\theta-1+\sigma(\omega+1/2)+\epsilon}.
    \end{equation*}

    We now turn our attention to the main term coming from \eqref{splitfamasymp}. We let
    \begin{equation*}
        C^*_1 = \frac{(q^2-1)(q^3-1)}{q^4(q-1)},
    \end{equation*}
    so that, using \eqref{firstsplitcoeffeq}, we see that the main-term contribution to \eqref{onelevtermsplit} is
    \begin{equation*}
        -C^*_1Yq^{-n/2}\sum_{\deg P\mid n} x_P(\deg P)\left(\delta_{2\deg P\mid n} + \delta_{3\deg P\mid n}+ 1/\abs{P}\right) = -C^*_1Yq^{-n/2}\sum_{\deg P\mid n} (\deg P)\delta_{2\deg P\mid n} + \mathcal{O}\left(Yq^{-n/6}\right),
    \end{equation*}
    where $x_P = \big(1+\abs{P}^{-1}+\abs{P}^{-2}\big)^{-1}$. Now, again by the Polynomial Prime Number Theorem, the sum over $P$ above equals $\delta_{2\mid n}\big(q^{n/2}+1\big)$. Hence, summing the expression above in \eqref{onelevexpr} gives us 
    \begin{equation*}
        -C_1^* Y \frac{2}{N_L\#\mathcal{F}(Y)}\sum_{n\geq 1}\widehat{\psi}\left(\frac{n}{N_L}\right)\delta_{2\mid n} + \mathcal{O}\left(\frac{1}{N_L}\sum_{n\geq 1}\widehat{\psi}\left(\frac{n}{N_L}\right)q^{-n/6}\right).
    \end{equation*}
    The error term is easily seen to be $\ll N_L^{-1}$ as $\widehat{\psi}$ is bounded. For the main term, we can use that $Y/\#\mathcal{F}(Y) = 1 +\mathcal{O}\big(Y^{-1/6}\big)$, together with Poisson summation, to write the above as
    \begin{equation*}
        -\frac{1}{N_L}\sum_{n\in \mathbb{Z}}\widehat{\psi}\left(\frac{2n}{N_L}\right) + \mathcal{O}\left(N_L^{-1}\right) = -\frac{1}{2}\sum_{n\in \mathbb{Z}}\psi\left(\frac{nN_L}{2}\right) + \mathcal{O}\left(N_L^{-1}\right) = -\frac{1}{2}\psi(0) + \mathcal{O}\left(N_L^{-1}\right).
    \end{equation*}
    This already provides us with the symplectic main term for the one-level density, as expected, cf. \cite{CFLS}.

    We now turn to the secondary term. This is slightly more delicate than the main term. First, we note that the contribution from a single prime to \eqref{onelevtermsplit} can be absorbed into the error term, and we may therefore restrict the summation to primes in $R$. Now, in \eqref{onelevtermsplit}, for the secondary term, the terms involving $\delta_{2\deg P \mid n}$ and $\delta_{3\deg P\mid n}$ can all be absorbed into the error term. Hence, the terms of interests are those coming from $2\#\mathcal{F}_{P,(111)}-\#\mathcal{F}_{P,(3)} + \#\mathcal{F}_{P,(1^21)}$. Now, we remark that with $C_{2,P,S}$ from Theorem \ref{allfieldSplitCondThm}, one has that
    \begin{equation}\label{splitpreponelevcor}
        2C_{2,P,(111)}-C_{2,P,(3)} + C_{2,P,(1^21)} \leq -D\abs{P}^{-1/3} + \mathcal{O}(\abs{P}^{-1}),
    \end{equation}
    for a constant $D > 0$. We delay the proof of this inequality to the end of Section \ref{splitcondchapt}.
    
    We conclude that the relevant contribution to \eqref{onelevtermsplit} from the secondary term is bounded from above by
        \begin{equation*}
        -DY^{5/6}\abs{P}^{-1/3} + \mathcal{O}\left(Y^{5/6}\abs{P}^{-1}\right).
    \end{equation*}
    One checks that the contribution from the error term above can be absorbed into our previous error term. Thus, the non-error contribution from \eqref{onelevtermsplit} coming from the secondary term is bounded by
    \begin{equation*}
         -DY^{5/6}q^{-n/2}\sum_{\deg P\mid n} (\deg P)\abs{P}^{-1/3} \leq -D'Y^{5/6}q^{n/6},
    \end{equation*}
    for some new constant $D'$. Let us now specify $\psi$ further. We choose $\psi$ to be a real even Schwartz function whose Fourier transform is real, even, nonnegative, and equal to $1$ on $[-\sigma+\epsilon', \sigma-\epsilon']$ for some $\epsilon' >0$. Using this in \eqref{onelevexpr} gives a term bounded from above by
    \begin{equation*}
        -D'Y^{5/6}\frac{2}{N_L\#\mathcal{F}(Y)}\sum_{n\geq 1}\widehat{\psi}\left(\frac{n}{N_L}\right)q^{n/6} \leq -D'Y^{5/6}\frac{2}{N_L\#\mathcal{F}(Y)}q^{(\sigma-2\epsilon')N_L/6} \leq -D''\frac{Y^{(\sigma-2\epsilon'-1)/6}}{N_L}.
    \end{equation*}

    To summarise, we have shown that
    \begin{equation*}
        \frac{1}{\#\mathcal{F}(Y)}\sum_{L\in \mathcal{F}(Y)}D_L(\psi) \leq \widehat{\psi}(0)-\frac{1}{2}\psi(0) + \mathcal{O}\left(N_L^{-1} + Y^{\theta-1+\sigma(\omega + 1/2)+\epsilon}\right) - D''\frac{Y^{(\sigma-2\epsilon'-1)/6}}{N_L}.
    \end{equation*}
    Recall also that $N_L\asymp \log_q(Y)$.
    
    To obtain a contradiction, let us assume that $\omega + \theta < 1/2$. This means that
    \begin{equation*}
        \frac{1-\theta}{\omega+ 1/2} > 1,
    \end{equation*}
    so that we may pick $\sigma = \sigma_0 > 1$ such that $\theta-1+\sigma_0(\omega + 1/2) + \epsilon < 0$, as long as $\epsilon$ is taken small enough. Write $\eta = \sigma_0-2\epsilon' > 1$. Then, we see that with this choice of $\sigma$, we have that
    \begin{equation*}
        \frac{1}{\#\mathcal{F}(Y)}\sum_{L\in \mathcal{F}(Y)}D_L(\psi) \leq -D'' \frac{Y^{(\eta-1)/6}}{N_L},
    \end{equation*}
    which in absolute value grows faster than some power of $Y$. However, from the definition, we immediately have that
    \begin{equation*}
        D_L(\psi) = \sum_{j=1}^{N_L}\sum_{n\in \mathbb{Z}}\psi\left(N_L\frac{\theta_j+2\pi n}{2\pi}\right) \ll N_L,
    \end{equation*}
    so that
    \begin{equation*}
        \frac{1}{\#\mathcal{F}(Y)}\sum_{L\in \mathcal{F}(Y)}D_L(\psi) \ll N_L \ll \log_q(Y),
    \end{equation*}
    which gives the desired contradiction.
\end{proof}

\section{Splitting conditions}\label{splitcondchapt}
We now extend the sieve used in Section \ref{fieldchapter} to also take splitting conditions for prime polynomials into account. Specifically, we fix finitely many prime polynomials $P_1,...,P_n$ and we wish to count the number of cubic fields where each $P_i$ splits according to the splitting type $S_i$. For maximal forms $v$, the splitting type of the corresponding maximal ring can be computed by reducing modulo $P$, which means we might as well study the reduction of $v$ modulo $P$. Now, in $V(R/PR)$ there are exactly five nonzero orbits, corresponding to the various splittings of the forms modulo $P$, see \cite[Section 4]{BST}.

If $\overline{P} := P_{r_1}....P_{r_m}\mid P_1...P_n$, then we let $T_{\overline{P}}(\overline{S})$ denote the set of forms whose reductions modulo $P_{r_j}$ splits according to $S_{r_j}$. We write $\overline{S}^c$ for the pointwise complement of $S$, so that $T_{\overline{P}}(\overline{S}^c)$ denotes the set of forms where none of the reductions modulo $P_{r_j}$ splits as $S_{r_j}$. Sometimes we write $S_{P_{r_j}}$ for $S_{r_j}$.

To simplify the calculations somewhat, we want to avoid treating the splitting type $(1^21)$ directly. To this end, we reorder the primes so that the only primes associated with the splitting type $(1^21)$ are $P_{n_0+1},...,P_n$ and write $\Tilde{P}_{(1^21)} $ for the product of these primes. Then, by inclusion-exclusion, we have that
\begin{equation*}
    N\big(U\cap T_{P_1...P_{n}}(S_1,...,S_n);X\big) = \sum_{\overline{P}\mid P_1...P_{n_0}}\mu(\overline{P})N\big(U\cap T_{\overline{P}, \Tilde{P}_{(1^21)}}\big(\overline{S}^c, S_{(1^21)}\big);X\big),
\end{equation*}
with $T_{\overline{P}, \Tilde{P}_{(1^21)}}\big(\overline{S}^c, S_{(1^21)}\big)$ denoting the forms whose reductions modulo primes dividing $\overline{P}$ splits according to $\overline{S}^c$, while the reduction modulo primes dividing $\Tilde{P}_{(1^21)}$ splits according to type $(1^21)$.

Again, by inclusion-exclusion, we see that
\begin{equation*}
    N\big(U\cap T_{\overline{P}, \Tilde{P}_{(1^21)}}(\overline{S}^c), S^{(1^21)};X\big) = \sum_{F}\mu(F)N(Z_{F,\overline{P}};X).
\end{equation*}
Here $Z_{P,\overline{P}}$ is simply $W_P$ if $P\nmid P_1...P_n$, else if $P\mid \overline{P}$, then $Z_{P,\overline{P}}$ is the set of forms which are either nonmaximal at $P$, or maximal at $P$ but with splitting type $S_P$. If instead $P\mid \Tilde{P}_{(1^21)}$, then $Z_{P,\overline{P}}$ are the sets of forms which are either nonmaximal at $P$, or with splitting type not equal to $(1^21)$. Finally, we define $Z_{F,\overline{P}}$ as the intersection of all $Z_{P,\overline{P}}$ with $P\mid F$. Note that our notation suppresses the dependence on $\Tilde{P}_{(1^21)}$.

We now make a few observations. If $S_P$ corresponds to an unramified splitting type, then $N(Z_{P,\overline{P}};X) = N(W_P;X) + N(T_P(S_P))$ as a form which is unramified at $P$ cannot be nonmaximal at $P$, which one sees by considering the discriminant. On the other hand, if $S_P=(1^3)$ we have that
\begin{equation*}
     N(Z_{P,\overline{P}};X) = N(W_P;X) + N(T_P(1^3); X)-N(W_P \cap T_P(1^3);X),
\end{equation*}
while if $S_P = (1^21)$ we instead find that
\begin{equation*}
    N(Z_{P,\overline{P}};X) = N(W_P;X) + N(T_P\big((1^21)^c\big); X)-N(W_P \cap T_P(1^3);X).
\end{equation*}
Note that we do not include the type $(0)$ in $(1^21)^c$ as these forms are a subset of $W_P$. Writing $V_P$ for the union of unramified splitting types and using the observations above, we may write
\begin{equation}\label{splitnonmaxdecomp}
    Z_{F,\overline{P}} = \Bigg(\bigcap_\sumstack{P\mid F\\ P\nmid P_1...P_n}W_P\Bigg) \bigcap \Bigg(\bigcap_\sumstack{P\mid (F,\overline{P})\\ S_P \text{ unramified }}\left(W_P \sqcup T_P(S_P)\right)\Bigg) \bigcap \Bigg(\bigcap_\sumstack{P\mid (F,\Tilde{P}_{(1^21)})}\left(V_P\sqcup R_{P}\right)\Bigg) \bigcap \left(\bigcap_\sumstack{P\mid (F,\overline{P})\\ S_P = (1^3)} R_{P}\right),
\end{equation}
where 
\begin{equation*}
   R_{P} =  \big(W_P \cup T_P((1^3)\big)\setminus \big(W_P \cap T_P(1^3)\big)
\end{equation*}
is considered as a set difference of a multiset.

We now introduce some notation. First, let $F_u$ be the product of the primes dividing both $\overline{P}$ and $F$ such that $S_P$ is unramified. Write $F_1$ for the product of primes dividing both $\overline{P}$ and $F$ with $S_P = (1^21)$ and similarly for $F_2$ but with the splitting type $(1^3)$. Writing $F = F'F_uF_1F_2$ we may then use \eqref{splitnonmaxdecomp} to write $N\big(Z_{F,\overline{P}};X\big)$ as
\begin{equation*}
\begin{split}
     \sum_\sumstack{h\ell\ell_1\ell_2 = F'F_uF_1 \\ h\mid F'F_u, \ell \mid F_u, \ell_1\ell_2\mid F_1} &N\bigg(W_h \cap T_{\ell}(S_\ell) \cap T_{\ell_1}(V_{\ell_1})\cap R_{\ell_2F_2} ;X\bigg) \\&= \sum_\sumstack{h\ell\ell_1\ell_2 = F'F_uF_1 \\ h\mid F'F_u, \ell \mid F_u, \ell_1\ell_2\mid F_1}\sum_\sumstack{fg\mid h\\ \alpha \in \mathbb{P}^1(R/fgR)}\mu(g) N\bigg(V_{fg,\alpha} \cap T_{\ell}(S_\ell) \cap T_{\ell_1}(V_{\ell_1})\cap R_{\ell_2F_2} ;X\abs{f}^2/\abs{h}^4\bigg).
\end{split}
\end{equation*}
To decompose $R_{\ell_2F_2}$, we use the same sieve as in \cite[Eq. (71)]{BST}, counting subrings and overrings, to see that the above equals
\begin{equation}\label{splitmiddlesums}
     \sum_\sumstack{h\ell\ell_1\ell_2 = F'F_uF_1 \\ h\mid F'F_u, \ell \mid F_u, \ell_1\ell_2\mid F_1}\sum_\sumstack{fg\mid h\\ \alpha \in \mathbb{P}^1(R/fgR)}\sum_\sumstack{abcd=\ell_2F_2\\\beta\in \mathbb{P}^1(R/(abR))} \mu(gb)N\bigg(V_{fg,\alpha}\cap V_{a,\beta}\cap V^2_{b,\beta} \cap T_{\ell}(S_\ell)\cap T_{d}(1^3) \cap  T_{\ell_1}(V_{\ell_1});X\abs{df}^2/\big(\abs{\ell_2F_2}^2\abs{c}^2\abs{h}^4\big)\bigg).
\end{equation}
Here $V^2_{b,\beta}$ denotes the forms modulo $b$ with a root $\beta$ of multiplicity at least two.

Finally, we make a few observations allowing us to separate a tail sum. As in \eqref{cutoffBound}, we have the bound
\begin{equation*}
    N(Z_{F,\overline{P}};X) \ll \frac{X}{\abs{F'}^{2-\epsilon}}.
\end{equation*}
We remark that one can obtain a better bound by keeping track of which splitting types are totally ramified, but as we are only interested in uniform bounds for all splitting types, we use this crude bound. In particular, we have that
\begin{equation}\label{splitcutoffsum}
    N\big(U\cap T_{\overline{P}, P_{(1^21)}}(\overline{S}^c), S^{(1^21)};X\big) = \sum_{F: \,\abs{F'} \leq X^{1/3}/\abs{F_uF_1F_2}^\delta}\mu(F)N(Z_{F,\overline{P}};X) + X^{2/3+\epsilon}\abs{(F,P_1...P_n)}^{\delta},
\end{equation}
where we used that the number of choices for $F_u,F_1,F_2$ is $\ll X^\epsilon$ as long as $\abs{\overline{P}}$ is bounded by some power of $X$.
\subsection{More Fourier analysis}
We now seek to evaluate the sum
\begin{equation}\label{splittingcount}
    \sum_\sumstack{ \alpha \in \mathbb{P}^1(R/fgR)}\sum_\sumstack{\beta\in \mathbb{P}^1(R/(abR))}N\bigg(V_{fg,\alpha}\cap V_{a,\beta}\cap V^2_{b,\beta} \cap T_{\ell}(S_\ell)\cap T_{d}(1^3)\cap  T_{\ell_1}(V_{\ell_1});X\abs{df}^2/\big(\abs{\ell_2F_2}^2\abs{c}^2\abs{h}^4\big)\bigg).
\end{equation}
The union over the $\beta$ of $V^2_{P,\beta}$ almost becomes the set of singular forms, except that the zero form is counted one time for every root. Hence, for $x\in V(R)$, we let $\Tilde{\omega}_P(x)$ be the indicator function of forms which are singular modulo $P$, except that $\Tilde{\omega_P}(0) = \#\mathbb{P}^1(R/PR)= \abs{P}+1$. Furthermore, we define $\Tilde{\omega}_b(x)$ multiplicatively in $b$. Then, similarly to how we arrived at \eqref{rootcountint}, we find that \eqref{splittingcount} is equal to
\begin{equation}\label{splittingcountint}
    \frac{\nu(K)}{\vol(G_0)}\int_{g\in \Lambda_{X'}S}\eta(\abs{t}) \left(\sum_{x\in gB \cap V(R)}\omega_{afg}(x)\Tilde{\omega}_b(x)\mathbf{1}_{T_{\ell}(S_\ell)}(x)\mathbf{1}_{T_{d}(1^3)}(x)\mathbf{1}_{T_{\ell_1}(V_{\ell_1})}(x)\right)\abs{t}^{-3}\abs{\lambda}^{-1}dtd\lambda df,
\end{equation}
with $X' = X\abs{df}^2/\big(\abs{\ell_2F_2}^2\abs{c}^2\abs{h}^4\big)$.

To estimate the inner sum above, we will need information about the Fourier transforms of the various functions inside the sum. Once again, we write $\chi = \chi_y$ for the character associated to $y$ and for a function $f$ we write $\widehat{f}(\chi_y) =: \widehat{f}(y)$. By \cite[Corollary 12]{TTOrb}, we have that 
\begin{equation}\label{singtransf}
    \widehat{\Tilde{\omega}}_P(y) = \abs{P}^{-1}\mathbf{1}_{\{0\}}(y) + \abs{P}^{-2}\mathbf{1}_{\{\mathrm{Disc}(y) = 0 \}}. 
\end{equation}
Moreover, \cite[Theorem 11]{TTOrb} provides exact values for the Fourier transform of the indicator functions of the various splitting types. In particular, we have that
\begin{equation*}
    \widehat{\mathbf{1}}_{T_P(1^3)}\ll \abs{P}^{-2}\mathbf{1}_{\{0\}}+\abs{P}^{-3},\,\,\,  \widehat{\mathbf{1}}_{T_P(S_P)}\ll \mathbf{1}_{\{0\}} + \abs{P}^{-1}\mathbf{1}_{T_P(1^3)} + \abs{P}^{-2},
\end{equation*}
where $S_P$ is an unramified splitting type.

Next, we use Fourier inversion to rewrite the sum in \eqref{splittingcountint} as
\begin{equation*}
    \sum_{i\leq i_0}\sum_{y \nsmod{afgd\ell\ell_1}}\chi_y\big(-x^{(i)}\big)\widehat{\omega}_{afg}(y)\widehat{\Tilde{\omega}_b}(y)\widehat{\mathbf{1}}_{T_{\ell}(S_\ell)}(y)\widehat{\mathbf{1}}_{T_{d}(1^3)}(y)\widehat{\mathbf{1}}_{T_{\ell_1}(V_{\ell_1})}(y)\sum_{x\in C_{\lambda,t}}\chi_y(x),
\end{equation*}
where the notation is from \eqref{gBdecomp}. We now refine the summation over $y$ slightly. Let $r_1\mid afg$, $r_2\mid b$, $r_3r_4\mid\ell\ell_1$ and $r_5\mid d$. We say that $y$ is of type $(r_1,...,r_5)$ if the reduction of $y$ modulo $afgd\ell\ell_1/(r_1...r_5)$ is zero, while all $y_i$ are units modulo $r_1...r_5$. Furthermore, its reduction modulo all primes dividing $r_1$ has a triple root, and its reduction modulo $r_2$ is singular. Its reduction modulo all primes dividing $r_3$ has a triple root, while its reduction modulo primes dividing $r_4$ does not have a triple root. This refinement will allow us to use the bound we have stated above for the various transforms. We thus rewrite the above as
\begin{equation}\label{splitsumfourier}
    \sum_{i\leq i_0}\sum_\sumstack{r_1\mid afg,\,\, r_2\mid b\\ r_3r_4\mid \ell \ell_1,\,\, r_5\mid d}\sum_\sumstack{y \nsmod{r_1...r_5}\\ y \text{ of type }(r_1,...,r_5)}\chi_y\big(-x^{(i)}\big)\widehat{\omega}_{afg}(y)\widehat{\Tilde{\omega}_b}(y)\widehat{\mathbf{1}}_{T_{\ell}(S_\ell)}(y)\widehat{\mathbf{1}}_{T_{d}(1^3)}(y)\widehat{\mathbf{1}}_{T_{\ell_1}(V_{\ell_1})}(y)\sum_{x\in C_{\lambda,t}}\chi_y(x).
\end{equation}

\subsection{Error terms}
When we considered the counting function for forms nonmaximal at $F$, the non-error contribution came from dual forms $y$ with all first three coordinates equal to zero. In this new setting with splitting conditions, we shall see that we find non-error contributions from $y$ of the form $(0,0,y_3,y_4)$. We begin by bounding the contribution from $y$ not of this shape.

Recall that the side-lengths of $C_{\lambda,t}$ are $c\abs{\lambda t^3},c\abs{\lambda t}, c\abs{\lambda/t}$ and $c\abs{\lambda/t^3}$. We write these lengths as $N_1, N_2, N_3$ and $N_4$ suppressing the dependence on $\lambda$ and $t$. We split into cases depending on the sizes of the various $N_j$. 

We first assume that $N_4 \geq 1$. The innermost sum over $x$ is then $\ll \abs{\lambda_0}^4$. Combined with the estimates for the various Fourier transforms, we find that \eqref{splitsumfourier} is
\begin{equation}\label{largeN4bound}
    \ll X^\epsilon\abs{\lambda_0}^4\sum_\sumstack{r_1\mid afg,\,\, r_2\mid b\\ r_3r_4\mid \ell \ell_1,\,\, r_5\mid d}\abs{r_1}^{-1}\abs{b}^{-1}\abs{r_2}^{-1}\abs{r_3}^{-1}\abs{r_4}^{-2}\abs{d}^{-2}\abs{r_5}^{-1}\sum_\sumstack{y \nsmod{r_1...r_5}\\ y \text{ of type }(r_1,...,r_5)\\ y\in C_{\lambda,t}^\perp}1.
\end{equation}
Here, by $C_{\lambda,t}^\perp$ we mean the set of $y$ such that $\chi_y$ is trivial on $C_{\lambda,t}$.

\subsubsection{Error term contribution from small $t$}

We now consider the range where in addition to $N_4\geq 1$, we also have that $N_1 < R:= \abs{r_1...r_5}$. The condition $y\in C_{\lambda,t}^\perp$ is then equivalent to $\abs{y_1} < c^{-1}\abs{r/\lambda t^3}, \abs{y_2} < c^{-1}\abs{r/\lambda t}, \abs{y_3} < c^{-1}\abs{rt/\lambda }$ and $ \abs{y_4} < c^{-1}\abs{rt^3/\lambda }$. We now proceed similarly to when we counted fields without splitting conditions. First, we bound the contribution from $y$ with discriminant zero, except those of the form $(0,0,y_3,y_4)$. In fact, we will find an acceptable error term contribution from forms not of the shape $(0,0,0,y_4)$ in this range.

All $y=y(w,z)$ with discriminant equal to zero have a double root. If this root is $[1:0]$, then $y$ has the form $er_1r_3wz^2+dz^3$. If $N_3\geq \abs{r_1r_3}$ this means that $e = 0$ so that the forms are of the shape $(0,0,0,y_4)$. Else, there are $\ll R^2\abs{t}^4/\abs{r_1r_3\lambda_0^2}$ different $y$ to choose from. As $N_1 < R$ we have that $\abs{t}\ll \abs{R/\lambda_0}^{1/3}$. Integrating the contribution from such $y$ in \eqref{splittingcountint} thus yields
\begin{equation*}
\begin{split}
    \ll X^\epsilon &\abs{\lambda_0}^{4/3} R^{8/3} \sum_\sumstack{r_1\mid afg,\,\, r_2\mid b\\ r_3r_4\mid \ell \ell_1,\,\, r_5\mid d}\abs{r_1}^{-2}\abs{b}^{-1}\abs{r_2}^{-1}\abs{r_3}^{-2}\abs{r_4}^{-2}\abs{d}^{-2}\abs{r_5}^{-1} \\ &\ll X^\epsilon \cdot \frac{X^{1/3}\abs{df}^{2/3}\abs{afgb\ell\ell_1}^{2/3}}{\abs{\ell_2F_2c}^{2/3}\abs{h}^{4/3}} \ll X^{1/3+\epsilon}\abs{\ell\ell_1}^{2/3} \ll X^{1/3+\epsilon}\abs{P_1...P_n}^{2/3}.
\end{split}
\end{equation*}
The contribution from these terms to the sum in \eqref{splitcutoffsum} is thus $\ll X^{2/3+\epsilon}\abs{P_1...P_n}^{2/3-\delta}$. The error from also summing over the various $\overline{P}$ can be absorbed in the $\epsilon$. 

We now look at the contribution from $y$ with a double root $[0:1]$. They have the form $aw^3+r_1r_3bw^2z$. If $\abs{R/r_1r_3}\leq c\abs{\lambda_0 t}$, there are $\ll R/\abs{\lambda_0 t^3}$ choices. Else, there are $R^2/\abs{\lambda_0^2 t^4}$ choices. The latter of these is clearly smaller than the error term we obtained above. The first of these instead gives a contribution $\ll X^\epsilon \abs{\lambda_0}^3$ when summed in \eqref{splitsumfourier} and then integrated in \eqref{splittingcountint}. Now,
\begin{equation*}
    \abs{\lambda_0}^3 = \frac{X^{3/4}\abs{df}^{3/2}}{\abs{\ell_2F_2c}^{3/2}\abs{h}^3} \ll \frac{X^{3/4}\abs{f}^{3/2}}{\abs{h}^3}.
\end{equation*}
In the range we are studying $\abs{\lambda_0} \ll N_1 < R\leq \abs{afgb\ell\ell_1 d}$. Write $h=fgg'$ so that
\begin{equation*}
    \frac{X^{1/4}\abs{d}^{1/2}\abs{agb\ell\ell_1}
    }{\abs{\ell_2F_2 c}^{1/2}} \leq \abs{f}^{3/2}\abs{g}^2\abs{g'}.
\end{equation*}
Similar to \eqref{lambdacubesumbound}, the contribution from $\abs{\lambda_0}^3$ can be bounded as
\begin{equation*}
    \ll X^{2/3+\epsilon}\frac{\abs{\ell_2F_2c}^{1/3}}{\abs{d}^{1/3}\abs{agb\ell\ell_1}^{2/3}} \ll X^{2/3+\epsilon}\abs{P_1...P_n}^{1/3}.
\end{equation*}

Finally, we turn to the contribution from $y$ with a double root at $[\ell:1]$, $\ell \neq 0$. These have the shape $aw^3+(b-2a\ell)w^2z + (a\ell^2-2b\ell)wz^2+b\ell^2 z^3$, where $a\ell+b \equiv 0$ modulo $r_1r_3$. Let us first assume that $\abs{\ell}^2 \geq c^{-1}R\abs{t}^3/\abs{\lambda_0r_1r_3}$ so that $\abs{b} < \abs{r_1r_3}$ by the size condition on the last coordinate. In this case, the congruence condition implies that $b$ is completely determined by $a,\ell$ and that $\abs{a\ell}\geq \abs{b}$. Furthermore, from the congruence we also see that $\abs{a\ell-b} = \abs{a\ell}$, as $2\nmid q$. By studying the first and third coefficients above, we can bound the contribution from forms of this shape as
\begin{equation*}
     \ll \sum_{\abs{a} \leq R/(\lambda t^3)} \frac{R^{1/2}\abs{t}^{1/2}}{\abs{\lambda_0  a}^{1/2}}\ll \frac{R}{\abs{\lambda_0}}.
\end{equation*}
The contribution to \eqref{largeN4bound} is $\ll X^\epsilon\abs{\lambda_0}^3$, and we have already bounded the sum of this above.

Let us now assume that $\abs{\ell}^2 < c^{-1}R\abs{t}^3/\abs{\lambda_0r_1r_3}$. Then, the first coordinate allows us to bound the number of different $a$ and the last coordinate allows us to bound the number of pairs $b,\ell$. We obtain an error
\begin{equation*}
        \ll \frac{R}{\abs{\lambda_0 t^3}} \sum_{\abs{\ell} <  R^{1/2}\abs{t}^{3/2}/(\abs{\lambda_0 r_1r_3}^{1/2})} \frac{R\abs{t^3}}{\abs{\ell^2\lambda_0 r_1r_3}} \ll \frac{R^2}{\abs{\lambda_0}^2 \abs{r_1r_3}}.
\end{equation*}
The contribution to \eqref{largeN4bound} is $\ll \abs{\lambda_0}^2$, which is smaller than the other error terms we have estimated.

To finish the study of the range where $N_1 < R$ and $N_4\geq 1$, we need to study the contribution from non-degenerate $y$. These forms have $(r_1r_3)^2\mid \mathrm{Disc}(y)$ and bounding the contribution from these $y$ as in Section \ref{fielderrorsect}, we obtain a contribution 
\begin{equation*}
    \ll \abs{ab^2dfg}\abs{\ell}^2,
\end{equation*}
to \eqref{largeN4bound}. Summing this over the various variables and over $\abs{F'}\leq X^{1/3}/\abs{F_uF_1F_2}^\delta$ gives an error $\ll X^{2/3+\epsilon}\abs{P_1...P_n}^{2-2\delta}$. At this point, we see that $\delta = 2/3$ is the best choice so far, and this will indeed turn out to be the optimal choice overall.

We now turn to the somewhat similar range when $N_1 < R$, and $N_4 < 1$. As usual, we may assume that $N_3\gg 1$. In this case, one modifies \eqref{largeN4bound} by replacing $\abs{\lambda_0}^4$ with $\abs{\lambda_0}^3\abs{t}^3$. Let us write $R = (r_1'r_1'')...(r_5'r_5'')$ with $r_i'r_i'' = r_i$ and let us assume that $y_1$ is zero modulo all $r_i'$ and a unit modulo all $r_i''$. Then, we need to pick $y_2$ and $y_3$ as zero modulo $r_1'r_3'$. In total we have 
\begin{equation*}
    \frac{\abs{r_1'r_3'r_2r_4r_5}}{1} \cdot \frac{\abs{r_1''r_3''r_2r_4r_5}^2}{\abs{\lambda_0^2t^4}} \cdot \abs{r_2r_4r_5}^2 \ll \frac{\abs{r_1r_3}^2\abs{r_2r_4r_5}^4}{\abs{\lambda_0}^3\abs{t}}.
\end{equation*}
The contribution to the analogue of \eqref{largeN4bound} with $\abs{\lambda_0^3t^3}$ in place of $\abs{\lambda_0^4}$ is thus $\ll \abs{t}^2\abs{r_1r_3}\abs{r_2^2r_4^2r_5}$. Integrating against $\abs{t}^{-3}$ up to $\abs{t}\ll \abs{\lambda}$ and summing yields a bound $\ll X^{2/3+\epsilon}\abs{P_1...P_n}^{2-2\delta}$.

\subsubsection{Error term contribution from larger $t$}
We now study the range where $N_1 \geq R$. Then, $y_1 = 0$ which means that $y_2,y_3 \equiv 0 \nsmod{r_1r_3}$. To study this range, we split it into several subranges.

Let us first assume that $N_2, N_3 < R/\abs{r_1r_3}$. This provides us with
\begin{equation*}
    \frac{\abs{r_2r_4r_5}^2}{\abs{\lambda_0}^2}\cdot \min\{R, R\abs{t}^3/\abs{\lambda_0}\}
\end{equation*}
choices for $y$. No matter if $N_4 \geq 1$ or not, the contribution to \eqref{largeN4bound} (or its analogue when $N_4 < 1$) is
\begin{equation*}
    \ll \abs{t}^3\abs{\lambda_0}\abs{r_2r_4}.
\end{equation*}
Now as $N_2 \leq \abs{r_2r_4r_5}$ we may integrate this against $\abs{t}^{-3}$ up to $\abs{t}\leq \abs{r_2r_4r_5/\lambda_0}$ obtaining a term $\ll \abs{r_2^2r_4^2r_5}\ll \abs{P_1...P_n}^2$. Summing over $\abs{F'}\leq X^{1/3}/\abs{F_uF_1F_2}^\delta$ yields an error $X^{1/3+\epsilon}\abs{P_1...P_n}^{2-\delta} \ll \abs{X}^{2/3+\epsilon}\abs{P_1...P_n}^{4/3-\delta}$, assuming $\abs{P_1...P_n}\ll X^{1/2}$. We note that our results are trivial if this is not the case, as the error term would then be $\gg X$.

We turn to the range where $N_2\geq R/\abs{r_1r_3}$ but $N_3 < R/\abs{r_1r_3}$. Then, $y_2 = 0$ as well. Hence, the number of ways to pick $y$ is
\begin{equation*}
    \ll \frac{\abs{r_2r_4r_5}\abs{t}}{\abs{\lambda_0}}\cdot\min\{R, R\abs{t}^3/\abs{\lambda_0}\}.
\end{equation*}
Similar to before, no matter if $N_4 \geq 1$ or not, the contribution to \eqref{largeN4bound} (or its analogue when $N_4 < 1$) is
\begin{equation*}
    \ll \abs{t}^4\abs{\lambda_0}^2.
\end{equation*}
Let us assume that $N_4\geq 1$ so that $\abs{t} \ll \abs{\lambda_0}^{1/3}$. Integrating the expression above against $\abs{t}^{-3}$ up to the indicated limit for $t$ yields $\ll \abs{\lambda_0}^{8/3}$. Now,
\begin{equation*}
    \abs{\lambda_0}^{8/3} = \frac{X^{2/3}\abs{df}^{4/3}}{\abs{\ell_2F_2c}^{4/3}\abs{h}^{8/3}} \ll \frac{X^{2/3}}{\abs{h}^{4/3}}.
\end{equation*}
Summing this yields a contribution $\ll X^{2/3+\epsilon}$ which is acceptable.

The case when $N_4 < 1$ above will be handled when we study the contribution of $y$ of the form $(0,0,y_3,y_4)$ in the next section. These $y$ will provide terms that cannot be absorbed into the error. Similarly, in the range where $N_1\geq R$, $N_2,N_3\geq \abs{r_2r_4r_5}$ only $y$ of the shape $(0,0,0,y_4)$ contributes, and we estimate this contribution in the next section.

\subsection{Main term contributions}
We now study contributions to \eqref{splitsumfourier} coming from $y$ of the shape $(0,0,y_3,y_4)$. By the arguments in the previous section, we may even assume that $y_3 = 0$ unless $N_4 < 1$. In fact, we may even assume that $N_4$ is smaller than some small constant in this case.

We now refine the type of $y$ even further. We write $r_5=r_6r_7$ and require that $y_3$ is congruent to zero modulo $r_6$, while it is a unit modulo $r_7$. We say that $y = (0,0,y_3,y_4)$ has type $(r_1,r_2,r_3,r_4,r_6,r_7)$. Knowing the type of $y$ then allows us to determine the splitting type of $y$ modulo the various divisors of $r_1...r_7$, which means that we may evaluate the Fourier transforms using \cite[Theorem 11]{TTOrb}. Using \eqref{rootFouriereval} and \eqref{singtransf}, we rewrite \eqref{splitsumfourier} as
\begin{equation*}
\begin{split}
    \frac{\sigma(afg)}{\abs{afg}\abs{b}}\sum_{i\leq i_0}&\sum_\sumstack{r_1\mid afg,\,\, r_2\mid b\\ r_3r_4\mid \ell \ell_1,\,\, r_6r_7\mid d}\chi_y\big(-x^{(i)}\big)\sigma(r_1)^{-1}\abs{r_2}^{-1}\prod_{P\mid b/r_2}\left(1+\abs{P}^{-1}\right)
    \\& \times \nu_1\big(\Tilde{S}_{d\ell\ell_1/(r_3r_4r_6r_7)}\big)\nu_2\big(\Tilde{S}_{r_3r_6}\big)\nu_3\big(\Tilde{S}_{r_4r_7}\big)\sum_\sumstack{y=(0,0,y_3,y_4)\\ y_3 \text{ mod } r_2r_4r_7, \,\,y_4 \text{ mod } r_1r_2r_3r_4r_6r_7\\ y_3 \in (R/r_4r_7)^*, y_4\in (R/r_1r_3r_6)^*\\ y \text{ of type $(r_1,r_2,r_3,r_4,r_6,r_7)$}}\sum_{x\in C_{\lambda,t}}\chi_y(x).
\end{split}
\end{equation*}
Here, the various $\nu_i$ functions are multiplicative with respect to the index of the argument. We have let $\Tilde{S}_P$ denote the indicator function of $(1^3)_P$ if $P\mid d$, the indicator function of $S_P$ if $P\mid \ell$ and the indicator function of $V_P$, i.e. all unramified forms modulo $P$, if $P\mid \ell_1$. The function $\nu_1(\Tilde{S}_P)$ is defined as the Fourier transform of $\Tilde{S}_P$ evaluated at the zero form. Similarly, $\nu_2(\Tilde{S}_P)$ is the Fourier transform evaluated at a form with a triple root modulo $P$, while $\nu_3(\Tilde{S}_P)$ is the Fourier transform evaluated at a form with splitting type $(1^21)$ modulo $P$. Here, we have used the fact that, see \cite[Theorem 11]{TTOrb}, the Fourier transform of $\Tilde{S}_P$, evaluated at $y$, only depends on the splitting type of $y$.

The sum over $y$ above can almost be separated into two independent sums over $y_3$ and $y_4$, except modulo $r_2$. We therefore split $r_2=r_{2,1}r_{2,2}$ and require that $y_3$ is zero modulo $r_{2,1}$, but a unit modulo $r_{2,2}$. Then the above becomes
\begin{equation}\label{splitmobiusprep}
\begin{split}
    \frac{\sigma(afg)}{\abs{afg}\abs{b}}\sum_{i\leq i_0}&\sum_\sumstack{r_1\mid afg,\,\, r_{2,1}r_{2,2}\mid b\\ r_3r_4\mid \ell \ell_1,\,\, r_6r_7\mid d}\chi_y\big(-x^{(i)}\big)\sigma(r_1)^{-1}\abs{r_{2,1}r_{2,2}}^{-1}\prod_{P\mid b/(r_{2,1}r_{2,2})}\left(1+\abs{P}^{-1}\right)
    \\& \times \nu_1\big(\Tilde{S}_{d\ell\ell_1/(r_3r_4r_6r_7)}\big)\nu_2\big(\Tilde{S}_{r_3r_6}\big)\nu_3\big(\Tilde{S}_{r_4r_7}\big)\sum_\sumstack{y=(0,0,y_3,y_4)\\ y_3 \text{ mod } r_{2,2}r_4r_7, \,\,y_4 \text{ mod } r_1r_{2,1}r_{2,2}r_3r_4r_6r_7\\ y_3 \in (R/r_{2,2}r_4r_7)^*, y_4\in (R/r_1r_{2,1}r_3r_6)^*}\sum_{x\in C_{\lambda,t}}\chi_y(x).
\end{split}
\end{equation}
When estimating subexpressions of \eqref{splitmobiusprep} in the sequel, we will make frequent use of bounds for the various Fourier transforms $\nu_j$, cf. \eqref{largeN4bound}.

Now that the summations over $y_3$ and $y_4$ above are independent of each other, we may apply Möbius inversion to remove the condition that $y_3$ and $y_4$ be units. We then see that the two innermost sums above, together with the sum over $i$, equal
\begin{equation}\label{youfoundme}
    \mu(r_1r_{2,1}r_{2,2}r_3r_4r_6r_7)\sum_{k\mid r_{2,2}r_4r_7, m\mid r_1r_{2,1}r_3r_6 }\mu(k)\mu(m)\sum_\sumstack{y=(0,0,y_3,y_4)\\ y_3 \text{ mod } k, \,\,y_4 \text{ mod } r_{2,2}r_4r_7m}\sum_{i\leq i_0}\chi_y\big(-x^{(i)}\big)\sum_{x\in C_{\lambda,t}}\chi_y(-x).
\end{equation}

By our work in the previous sections on error terms, we need only consider $y$ with $y_3 \neq 0$ if $c\abs{\lambda_0/t^3}$ is smaller than some small constant. Let us first assume that $t$ is such that $c\abs{\lambda_0/t^3}\gg 1$, whence we can absorb the terms with $y_3 \neq 0$ into the error term and consider the contribution, cf \eqref{sliceFields}, from 
\begin{equation}\label{tobemod}
  \mathbf{1}_{\{c\abs{\lambda_0}\gg \abs{t}^3\}}  \mu(r_1r_{2,1}r_{2,2}r_3r_4r_6r_7)\sum_{k\mid r_{2,2}r_4r_7, m\mid r_1r_{2,1}r_3r_6 }\mu(k)\mu(m)\abs{r_{2,2}r_4r_7m}\sum_{x_4: r_{2,2}r_4r_7m\mid x_4}\vol\big((gB)_{x_4}\big),
\end{equation}
where we slice over the last coordinate in $gB$. 

Now, we note that the contribution from the range where $c\abs{\lambda_0/t^3}\gg 1$ and $c\abs{\lambda_0/t} < \abs{k}$ is negligible, in the sense that we could absorb the entire expression above into the error term in this range. Indeed, in this range, we must have that $\abs{\lambda_0} \ll \abs{k}^{3/2}\leq \abs{r_2r_4r_7}^{3/2}$. The contribution from $y$ of the form $(0,0,0,y_4)$ to \eqref{splitmobiusprep} can thus be bounded by 
\begin{equation*}
    \abs{\lambda_0}^{7/3}\abs{t}^3,
\end{equation*}
using that $\abs{\lambda_0}^{2/3} \ll \abs{k}$. Noting that $N_4\geq 1$ and integrating over $t$ yields $\ll \abs{\lambda_0}^{8/3}$, whose contribution we have already bounded.

Instead of discarding the terms from the range above entirely, we modify the expression \eqref{tobemod} so that the above becomes, up to an acceptable error, equal to
\begin{equation*}
\begin{split}
    &\mathbf{1}_\sumstack{\{c\abs{\lambda_0}\gg \abs{t}^3\}}  \mu(r_1r_{2,1}r_{2,2}r_3r_4r_6r_7)\sum_{k\mid r_{2,2}r_4r_7, m\mid r_1r_{2,1}r_3r_6 }\mu(k)\mu(m)\abs{r_{2,2}r_4r_7m}\sum_{x_4: r_{2,2}r_4r_7m\mid x_4}\vol\big((gB)_{x_4}\big)
    \\& + \mu(r_1r_{2,1}r_{2,2}r_3r_4r_6r_7)\sum_{k\mid r_{2,2}r_4r_7, m\mid r_1r_{2,1}r_3r_6 }\mu(k)\mu(m)\abs{r_{2,2}r_4r_7mk}\mathbf{1}_\sumstack{\{c\abs{\lambda_0}\gg\abs{t}^3,\,\,c\abs{\lambda_0/t} < \abs{k}\}} \sum_{x_3\neq 0: k\mid x_3}\vol\big((gB)_{x_3,0}\big).
\end{split}
\end{equation*}
Here we used that the entire last sum can be absorbed into the error term. Indeed, we can bound the second double sum above as $\ll R\abs{r_{2,2}r_4r_7}\abs{\lambda}^2\abs{t}^4$. The contribution to \eqref{splitmobiusprep} is thus $\ll \abs{\lambda_0}^2\abs{t}^4$, which integrates to $\ll \abs{\lambda_0}^{8/3}$ as $N_4\geq 1$. Note that when $c\abs{\lambda_0/t}\geq k$, the part of the innermost sum on the first row above where $x_4=0$ is equal to $\abs{k}$ multiplied by the innermost sum on the second row. Hence, we may add the condition $x_4\neq 0$ to the first sum and remove the condition on $k$ in the indicator function in the second sum.

If $c\abs{\lambda_0/t^3}$ is small, then the only possible last coordinate in $gB$ is zero. This means that the contribution in this range from \eqref{youfoundme} is
\begin{equation*}
\begin{split}
    \mathbf{1}_\sumstack{\{c\abs{\lambda_0/t^3} \text{ small}\} } \mu(r_1r_{2,1}r_{2,2}r_3r_4r_6r_7)\sum_{k\mid r_{2,2}r_4r_7, m\mid r_1r_{2,1}r_3r_6 }\mu(k)\mu(m)\abs{r_{2,2}r_4r_7mk}\sum_{x_3\neq 0: \,k\mid x_3}\vol\big((gB)_{x_3,0}\big).
\end{split}
\end{equation*}
Adding together these two expressions yields, up to an acceptable error term, two sums
\begin{equation}\label{splitnonerrorsums}
\begin{split}
    &\mu(r_1r_{2,1}r_{2,2}r_3r_4r_6r_7)\sum_{k\mid r_{2,2}r_4r_7, m\mid r_1r_{2,1}r_3r_6 }\mu(k)\mu(m)\abs{r_{2,2}r_4r_7m}\sum_{x_4\neq 0: r_{2,2}r_4r_7m\mid x_4}\vol\big((gB)_{x_4}\big)
    \\& + \mu(r_1r_{2,1}r_{2,2}r_3r_4r_6r_7)\sum_{k\mid r_{2,2}r_4r_7, m\mid r_1r_{2,1}r_3r_6 }\mu(k)\mu(m)\abs{r_{2,2}r_4r_7mk}\sum_{x_3\neq 0: k\mid x_3}\vol\big((gB)_{x_3,0}\big),
\end{split}
\end{equation}
which we now estimate. As there is no $k$-dependence in the first row, except for $\mu(k)$, we can sum over $k$ and see that it is nonzero only if $r_{2,2}r_4r_7 = 1$, in which case it equals
\begin{equation}\label{splitmainsum}
    \mu(r_1r_{2,1}r_3r_6)\sum_{ m\mid r_1r_{2,1}r_3r_6 }\mu(m)\abs{m}\sum_{x_4\neq 0: m\mid x_4}\vol\big((gB)_{x_4}\big).
\end{equation}

The innermost sum in \eqref{splitmainsum} has already been evaluated when studying the sum over $d$ in \eqref{sliceFields}, but here we have $m$ in place of $r'$. The integral over $\lambda,t,f$ of this sum is equal to, cf \eqref{mainandtail},
\begin{equation}\label{splitmainandtail}
\begin{split}
    \frac{q\abs{m}^{-1}}{(q-1)(q^2-1)}&\abs{\lambda_0}^{4}\left(\vol(B)-\int_\sumstack{\abs{u}<\abs{m/\lambda_0} } \vol(B_u)du \right) \\&- \frac{\abs{m}^{-1/3}}{(q-1)(q^2-1)}\abs{\lambda_0}^{10/3}\left(I^\sigma_1(\lambda_0/m)-\int_\sumstack{\abs{u}< \abs{m/\lambda_0} } \vol(B_u)q^{2\epsilon(u\lambda_0/m)/3}\abs{u}^{-2/3}du\right).
\end{split}
\end{equation}

We turn our attention to the innermost sum in the last row of \eqref{splitnonerrorsums}. Previously, this sum has only been evaluated when $k=1$. For a general $k$, we have that
\begin{equation*}
   \sum_{x_3\neq 0: k\mid x_3} \int_{g\in \Lambda_{X'}S}\eta(\abs{t})\vol\big((gB)_{x_3,0}\big)\abs{t}^{-3}\abs{\lambda}^{-1}dtd\lambda =  (q-1)\sum_{c\in R\setminus\{0\}}'\abs{\lambda_0}^2\int_{ \abs{t}\geq 1}\eta(\abs{t}) \vol(B_{(ckt/\lambda_0),\,0})\abs{t}dt.
\end{equation*}
From this point onwards, we may copy the computations from Proposition \ref{implicitformcount} for the integral with $\lambda_0$ replaced with $\lambda_0/k$. We thus see that the above is equal to
\begin{equation*}
    \frac{q\abs{\lambda_0}^4}{(q-1)^2\abs{k}^2}\int_{\abs{u}\geq \abs{k/\lambda_0}}\vol(B_u)\abs{u}du- \frac{(q^2+1)\abs{\lambda_0}^3}{(q^2-1)(q-1)\abs{k}}\int_{\abs{u}\geq \abs{k/\lambda_0}}\vol(B_u)du.
\end{equation*}
The contribution from the tail of either of these integrals is $\ll \abs{\lambda_0}^2$, which one checks gives a contribution that can be absorbed into the error term. Hence, up to a small error, we have obtained
\begin{equation}\label{splitredterms}
    \frac{q\abs{\lambda_0}^4}{(q-1)^2\abs{k}^2}I_0^\sigma- \frac{(q^2+1)\abs{\lambda_0}^3}{(q^2-1)(q-1)\abs{k}}\vol(B_0).
\end{equation}

Now, we note that if $\abs{m}$ is smaller than $\abs{\lambda_0}$ multiplied with some small constant, then one checks that the tail terms from \eqref{splitmainandtail} cancels perfectly against the second term in \eqref{splitredterms}, similar to when we studied all cubic fields. On the other hand, if $\abs{m}\gg \abs{\lambda_0}$, then the contribution from the tail terms in \eqref{splitmainandtail} and the second term in \eqref{splitredterms} is $\ll \abs{\lambda_0}^3$ and using our previous calculations, this can be absorbed into the error term. Hence, our non-error terms will come from
\begin{equation}\label{nonerrorcorelist}
    \frac{q\abs{m}^{-1}}{(q-1)(q^2-1)}\abs{\lambda_0}^{4}\vol(B)- \frac{\abs{m}^{-1/3}}{(q-1)(q^2-1)}\abs{\lambda_0}^{10/3}I^\sigma_1(\lambda_0/m),\,\,\, \text{  and  }\,\,\, \frac{q}{(q-1)^2\abs{k}^2}\abs{\lambda_0}^4I_0^\sigma.
\end{equation}

The contribution to \eqref{splitnonerrorsums} from the first term above is
\begin{equation}\label{innermainsplit}
    \frac{q}{(q-1)(q^2-1)}\abs{\lambda_0}^{4}\vol(B)\mu(r_1r_{2,1}r_3r_6)\sum_{ m\mid r_1r_{2,1}r_3r_6 }\mu(m) = \frac{q}{(q-1)(q^2-1)}\abs{\lambda_0}^{4}\vol(B)
\end{equation}
if $r_1r_{2,1}r_3r_6 = 1 $ and zero else. As the first sum in \eqref{splitnonerrorsums} vanishes unless also $r_{2,2}r_4r_7 = 1$, this means that all $r_i$ equals one, which means that this is the contribution from the trivial character. Using this in \eqref{splitmobiusprep}, we obtain
\begin{equation}
\begin{split}
    \frac{q}{(q-1)(q^2-1)}\abs{\lambda_0}^{4}\vol(B)\frac{\sigma(afg)}{\abs{afg}\abs{b}}&\prod_{P\mid b}\left(1+\abs{P}^{-1}\right)\nu_1\big(\Tilde{S}_{d\ell\ell_1}\big).
\end{split}
\end{equation}
We remark that by the definition of the Fourier transform, $\nu_1(\Tilde{S}_P)$ is the density of the set that $\Tilde{S}_P$ indicates modulo $P$.

Turning to the second term from \eqref{nonerrorcorelist}, we see that its contribution to \eqref{splitnonerrorsums} is
\begin{equation*}
    - \frac{1}{(q-1)(q^2-1)}\abs{\lambda_0}^{10/3}\mu(r_1r_{2,1}r_3r_6)\sum_{ m\mid r_1r_{2,1}r_3r_6 }\mu(m)\abs{m}^{2/3}I^\sigma_1(\lambda_0/m),
\end{equation*}
under the assumption that $r_{2,2}r_4r_7 = 1$. Using this in \eqref{splitmobiusprep} gives us 
\begin{equation}\label{splitsecorderrsum}
\begin{split}
     - \frac{\abs{\lambda_0}^{10/3}}{(q-1)(q^2-1)}\frac{\sigma(afg)}{\abs{afg}\abs{b}}&\sum_\sumstack{r_1\mid afg,\,\, r_{2,1}\mid b\\ r_3\mid \ell \ell_1,\,\, r_6\mid d}\mu(r_1r_{2,1}r_3r_6)\sigma(r_1)^{-1}\abs{r_{2,1}}^{-1}\prod_{P\mid b/r_{2,1}}\left(1+\abs{P}^{-1}\right)
    \\& \times \nu_1\big(\Tilde{S}_{d\ell\ell_1/(r_3r_6)}\big)\nu_2\big(\Tilde{S}_{r_3r_6}\big)\sum_{ m\mid r_1r_{2,1}r_3r_6 }\mu(m)\abs{m}^{2/3}I^\sigma_1(\lambda_0/m).
\end{split}
\end{equation}

Finally, we examine the third term from \eqref{nonerrorcorelist} and see that its contribution to \eqref{splitnonerrorsums} is
\begin{equation*}
\begin{split}
    \frac{q\abs{\lambda_0}^4}{(q-1)^2}&I_0^\sigma\mu(r_1r_{2,1}r_{2,2}r_3r_4r_6r_7)\abs{r_{2,2}r_4r_7}\sum_{k\mid r_{2,2}r_4r_7,\,\, m\mid r_1r_{2,1}r_3r_6 }\mu(k)\mu(m)\abs{m}\abs{k}^{-1}
    \\& = \frac{q\abs{\lambda_0}^4}{(q-1)^2}I_0^\sigma\Phi(r_1r_{2,1}r_{2,2}r_3r_4r_6r_7)\mu(r_{2,2}r_4r_7).
\end{split}
\end{equation*}
Using this in \eqref{splitmobiusprep} yields
\begin{equation}\label{redsplitrsum}
\begin{split}
    \frac{q\abs{\lambda_0}^4}{(q-1)^2}I_0^\sigma\frac{\sigma(afg)}{\abs{afg}\abs{b}}&\sum_\sumstack{r_1\mid afg,\,\, r_{2,1}r_{2,2}\mid b\\ r_3r_4\mid \ell \ell_1,\,\, r_6r_7\mid d}\sigma(r_1)^{-1}\abs{r_{2,1}r_{2,2}}^{-1}\prod_{P\mid b/(r_{2,1}r_{2,2})}\left(1+\abs{P}^{-1}\right)
    \\& \times \nu_1\big(\Tilde{S}_{d\ell\ell_1/(r_3r_4r_6r_7)}\big)\nu_2\big(\Tilde{S}_{r_3r_6}\big)\nu_3\big(\Tilde{S}_{r_4r_7}\big)\Phi(r_1r_{2,1}r_{2,2}r_3r_4r_6r_7)\mu(r_{2,2}r_4r_7).
\end{split}
\end{equation}

We now compute the non-error terms. The total contribution from \eqref{innermainsplit} to $N(Z_F,\overline{P};X)$ is
\begin{equation*}
    \frac{q}{(q-1)\#\mathrm{Aut}(\sigma)}X\sum_\sumstack{h\ell\ell_1\ell_2 = F'F_uF_1 \\ h\mid F'F_u, \ell \mid F_u, \ell_1\ell_2\mid F_1}\sum_\sumstack{fg\mid h}\sum_\sumstack{abcd=\ell_2F_2} \mu(gb)\frac{\abs{df}^2}{\abs{\ell_2F_2}^2\abs{c}^2\abs{h}^4}\frac{\sigma(afg)}{\abs{afg}\abs{b}}\prod_{P\mid b}\left(1+\abs{P}^{-1}\right)\nu_1\big(\Tilde{S}_{d\ell\ell_1}\big).
\end{equation*}
Write $h=F'h_u$ with $h_u\mid F_u$. Then, the above is $Xq/\big((q-1)\#\mathrm{Aut}(\sigma)\big)$ multiplied with
\begin{equation*}
\begin{split}
    \sum_\sumstack{h_uh_1\ell\ell_1\ell_2 = F_uF_1 \\ h_u\mid F_u, \,\ell \mid F_u, \ell_1\ell_2\mid F_1}&\left(\prod_{P\mid F'h_u} \abs{P}^{-2}+\abs{P}^{-3}-\abs{P}^{-5} \right)\left(\prod_{P\mid \ell} \nu_1(S_P)\right)\left(\prod_{P\mid \ell_1} \nu_1(V_P)\right)\left(\prod_{P\mid F_2\ell_2} 2\abs{P}^{-2}-\abs{P}^{-4} \right)
    \\& = \left(\prod_{P\mid F_2} 2\abs{P}^{-2}-\abs{P}^{-4} \right)\left(\prod_{P\mid F'} \abs{P}^{-2}+\abs{P}^{-3}-\abs{P}^{-5} \right)\left(\prod_{P\mid F_u} \abs{P}^{-2}+\abs{P}^{-3}-\abs{P}^{-5} + \nu_1(S_P) \right)
    \\&\times\left(\prod_{P\mid F_1}1-\abs{P}^{-1} + \abs{P}^{-2}+\abs{P}^{-3}-\abs{P}^{-4} \right).
\end{split}
\end{equation*}
For the estimation of the tail sum over $F$, we note that the above is $\ll X \abs{F_2}^{-2}\abs{F'}^{-2}$. Summing over $\abs{F'}\geq X^{1/3}/\abs{F_uF_1F_2}^\delta$ results in an error $X^{2/3+\epsilon}\abs{P_1...P_n}^\delta$, which is acceptable with the choice $\delta = 2/3$.

For the rest of the sum over $\overline{P}$ and $F$, we introduce some notation. Write $\Tilde{P}_u$ for the product of the $P_1,...,P_n$ which corresponds to unramified splitting types. Similarly, we write $\Tilde{P}_{(1^3)}$ and $\Tilde{P}_{(1^21)}$ for the product of totally ramified and partially ramified splitting types. Then, we note that
\begin{equation}\label{FPbarchange}
    \sum_{\overline{P}\mid \Tilde{P}_{u}\Tilde{P}_{(1^3)}}\mu(\overline{P})\sum_{F}\mu(F) = \sum_{F_u\mid \Tilde{P}_u}\mu(F_u)\sum_{F_1\mid \Tilde{P}_{(1^21)}}\mu(F_1)\sum_{F_2\mid \Tilde{P}_{(1^3)}}\mu(F_2)\sum_{F': \, \big(F', F_uF_1\Tilde{P}_{(1^21)}\big)=1}\mu(F')\sum_\sumstack{\overline{P} \mid \Tilde{P}_u\Tilde{P}_{(1^3)}, \,\, (F', \overline{P})=1\\ F_u \mid\overline{P}_u, \,\,F_2\mid\overline{P}_{(1^3)} }\mu(\overline{P}).
\end{equation}
The innermost sum is zero unless $F_uF_2(F',\Tilde{P}_u\Tilde{P}_{(1^3)}) = \Tilde{P}_u\Tilde{P}_{(1^3)}$, in which case it equals $\mu(F_uF_2)$. Hence, summing the above, including the prefactor, yields
\begin{equation}\label{maintermsplitconds}
\begin{split}
    \frac{q^2-1}{q^2\#\mathrm{Aut}(\sigma)}&X\left(\prod_{P\mid \Tilde{P}_{u} }\nu_1(S_P)\left(1-\abs{P}^{-2}-\abs{P}^{-3}+\abs{P}^{-5}\right)^{-1}\right)
 \left(\prod_{P\mid \Tilde{P}_{(1^21)} }\abs{P}^{-1}\left(1+\abs{P}^{-1}+\abs{P}^{-2}\right)^{-1}\right)
    \\& \times \left(\prod_{P\mid \Tilde{P}_{(1^3)} }\abs{P}^{-2}\left(1+\abs{P}^{-1}+\abs{P}^{-2}\right)^{-1}\right),
\end{split}
\end{equation}
where we used that
\begin{equation*}
\begin{split}
    \sum_\sumstack{F': \, \big(F', F_uF_1\Tilde{P}_{(1^21)}\big)=1\\ (\Tilde{P}_u\Tilde{P}_{(1^3)}/F_uF_2)\mid F'}&\mu(F')\left(\prod_{P\mid F'} \abs{P}^{-2}+\abs{P}^{-3}-\abs{P}^{-5}\right) = \zeta_R(2)^{-1}\zeta_R(3)^{-1}\left(\prod_{P\mid \Tilde{P}_u\Tilde{P}_{(1^3)}/(F_uF_2)}\big(-\abs{P}^{-2}-\abs{P}^{-3}+\abs{P}^{-5}\big)\right)
    \\& \times \left(\prod_{P\mid P_1...P_n }\left(1-\abs{P}^{-2}-\abs{P}^{-3}+\abs{P}^{-5}\right)^{-1}\right).
\end{split}
\end{equation*}
One checks that $\nu_1(S_P)\left(1-\abs{P}^{-2}-\abs{P}^{-3}+\abs{P}^{-5}\right)^{-1}$ is $1+\abs{P}^{-1}+\abs{P}^{-2}$ multiplied with either $1/6, 1/2$ or $1/3$ depending on whether $S_P$ is of type $(111), (21)$ or $(3)$, respectively, see \cite[Theorem 11]{TTOrb}.

We turn to the reducible terms and simplify \eqref{redsplitrsum}. The presence of the factor $\mu(r_{2,2})$ means that the terms when $r_{2,1}r_{2,2} \neq 1$ all cancel each other. What remains is
\begin{equation*}
\begin{split}
    \frac{q\abs{\lambda_0}^4}{(q-1)^2}I_0^\sigma\frac{\sigma(abfg)}{\abs{afg}\abs{b}^2}&\sum_\sumstack{r_1\mid afg\\ r_3r_4\mid \ell \ell_1,\,\, r_6r_7\mid d}\sigma(r_1)^{-1}
\nu_1\big(\Tilde{S}_{d\ell\ell_1/(r_3r_4r_6r_7)}\big)\nu_2\big(\Tilde{S}_{r_3r_6}\big)\nu_3\big(\Tilde{S}_{r_4r_7}\big)\Phi(r_3r_4r_6r_7)\mu(r_4r_7).
\end{split}
\end{equation*}
We do not need to evaluate this sum for the purpose of counting cubic fields. Instead, we remark that the tail when summing over $F$ is $\ll X^{2/3+\epsilon}\abs{P_1...P_n}^{2/3}$ by the same reasoning as before for the non-reducible main term. Hence, after summing the above, we obtain a term
\begin{equation*}
    XC_3(\sigma,S_1,...,S_n),
\end{equation*}
with $C_3$ only depending on $\sigma$ and the splitting types $S_i$ (and their associated primes).

Finally, we should study the second-order term, coming from \eqref{splitsecorderrsum}. Here, the main difficulty is that $I^\sigma(\lambda_0/m)$ depends on the valuation of $\lambda_0/m$ modulo three. As before, we may bound the tail term arising from the summation over $F$.

We begin by studying the contribution coming from terms associated with $F'$ from \eqref{FPbarchange}. For this purpose, we consider the contribution of the $r_1$ dividing the part of $fg$ which divides $F'$, to \eqref{splitsecorderrsum}. First, we note that
\begin{equation*}
    -\frac{\nu(K)}{\vol(G_0)(q-1)(q^2-1)}\abs{\lambda}^{10/3}I^\sigma_1(\lambda_0/m) = -\frac{X^{5/6}}{\#\mathrm{Aut}(\sigma)(q-1)}\cdot \frac{\abs{df}^{5/3}}{\abs{\ell_2F_2c}^{5/3}\abs{h}^{10/3}}C_2(\ell_0-4\deg(m)),
\end{equation*}
with $X= q^\ell$ and
\begin{equation*}
    \ell_0 = \ell + 2\deg(f'f_u)+2\deg(d) - 2\deg(\ell_2F_2 c)-4\deg(F'h_u) =: \ell_0' + 2\deg(f')-4\deg(F'),
\end{equation*}
where we have written $h=F'h_u$ and $f',g'$ for the maximal divisor of $f,g$ such that $f'g'$ divides $F'$. 

We now compute from \eqref{splitsecorderrsum}, including the factor $\mu(g')$ from \eqref{splitmiddlesums}, that
\begin{equation}
\begin{split}
     \sum_{f'g'\mid F'}\mu(g')&\frac{\sigma(f'g')\abs{f'}^{5/3}}{\abs{f'g'}\abs{F'}^{10/3}}\sum_\sumstack{r'_1\mid f'g'}\sigma(r_1)^{-1}\sum_{ m\mid r'_1}\mu(m)\abs{m}^{2/3}C_2\big(\ell'_0 + 2\deg(f')-4\deg(F)-4\deg(m)\big)=
     \\& = \sum_{f'g' = F'}C_2^\sigma\big(\ell'_0 - 2\deg(f)\big) \left(\prod_{P\mid f'} \left(1-\abs{P}^{-2}\right)\abs{P}^{-5/3}\right)\prod_{P\mid g'}\abs{P}^{-2},
\end{split}
\end{equation}
as in \eqref{fieldinduct}. Summing this over $f'g'\mid F'$ and $F'$, using \eqref{FPbarchange}, yields
\begin{equation}\label{nonmaxcontribsplit}
\begin{split}
    &\sum_\sumstack{F': \, \big(F', F_uF_2\Tilde{P}_{(1^21)}\big)=1\\ (\Tilde{P}_u\Tilde{P}_{(1^3)}/F_uF_2)\mid F'}\mu(F')\sum_{f'g'= F'}C_2^\sigma\big(\ell'_0 - 2\deg(f')\big) \left(\prod_{P\mid f'} \left(1-\abs{P}^{-2}\right)\abs{P}^{-5/3}\right)\prod_{P\mid g'}\abs{P}^{-2}
    \\& =\sum_{(f'_0,P_1...P_n)}\sum_{f'_1 \mid (\Tilde{P}_u\Tilde{P}_{(1^3)}/F_uF_2)}\sum_{(g'_0,f_0'P_1...P_n)}C_2^\sigma\big(\ell'_0 - 2\deg(f')\big) \mu(f'_0)\left(\prod_{P\mid f'_0f'_1} \left(1-\abs{P}^{-2}\right)\abs{P}^{-5/3}\right)\mu(g'_0)\prod_{P\mid g'_0(\Tilde{P}_u\Tilde{P}_{(1^3)}/F_uF_2f'_1)}\abs{P}^{-2}
    \\& = \frac{\abs{F_uF_2}^2}{\abs{\Tilde{P}_u\Tilde{P}_{(1^3)}}^2\zeta_R(2)}\prod_{P\mid P_1...P_n}\left(1-\abs{P}^{-2}\right)^{-1}\sum_{(f'_0,P_1...P_n)}\sum_{f'_1\mid \Tilde{P}_u\Tilde{P}_{(1^3)}/(F_uF_2)}C_2^\sigma\big(\ell'_0 - 2\deg(f'_0f'_1)\big) \mu(f'_0)\abs{f'_0}^{-5/3}\\&\times\abs{f'_1}^{1/3}\mu\left(\frac{\Tilde{P}_u\Tilde{P}_{(1^3)}}{F_uF_2}\right)\prod_{P\mid f_1'}\left(1-\abs{P}^{-2}\right).
\end{split}
\end{equation}
Now, the above is not multiplicative because of the factor $C_2^\sigma$. Nevertheless, we introduce multiplicative notation that will allow us to rewrite the above into something that appears more tractable.

We first define a formal product $\otimes$ by defining $\phi_\ell(n) \otimes \phi_\ell(m) = \phi_\ell(m+n)$, where $\phi_\ell(n), \phi_\ell(m)$ are a priori simply symbols. We define an evaluation map, mapping the symbol $\phi_\ell(n)$ to $C_2^\sigma(\ell+n)$ so that the point of the definition of $\otimes$ is that \eqref{nonmaxcontribsplit} is the image under the evaluation map of
\begin{equation}\label{nonmaxtimes}
\begin{split}
    \frac{\mu\left(\frac{\Tilde{P}_u\Tilde{P}_{(1^3)}}{F_uF_2}\right)\abs{F_uF_2}^2}{\abs{\Tilde{P}_u\Tilde{P}_{(1^3)}}^2\zeta_R(2)}&\prod_{P\mid P_1...P_n}\left(1-\abs{P}^{-2}\right)^{-1}\phi_\ell\big(\ell'_0-\ell\big) \bigotimes_{P\nmid P_1...P_n}\left(\phi_\ell(0)-\phi_\ell\big(-2\deg(P)\big)\abs{P}^{-5/3}\right)\\&\bigotimes_{P\mid \Tilde{P}_u\Tilde{P}_{(1^3)}/(F_uF_2)} \left(\phi_\ell(0) + \phi_\ell\big(-2\deg(P)\big)\abs{P}^{1/3}\big(1-\abs{P}^{-2}\big)\right).
\end{split}
\end{equation}
Usually, we will simply identify expressions such as the one above with their image under the evaluation map. Note that the evaluation map does not commute with $\otimes$.

We now compute the contribution to the sum over $F$ from \eqref{splitsecorderrsum} coming from terms involving $F_2$, including the relevant contribution from \eqref{nonmaxtimes}, using the decomposition \eqref{FPbarchange}. We write $abcd=a_1a_2...d_1d_2$, where $a_2,b_2,c_2,d_2 \mid F_2$. Then, we see through a computation using the almost-multiplicativity that for any fixed integer $n$, we have that
\begin{equation*}
\begin{split}
   \sum_{F_2\mid \Tilde{P}_{(1^3)}}\frac{\mu\left(\frac{\Tilde{P}_{(1^3)}}{F_2}\right)\abs{F_2}^2}{\abs{\Tilde{P}_{(1^3)}}^2} &\bigotimes_{P\mid \Tilde{P}_{(1^3)}/F_2} \left(\phi_\ell(0) + \phi_\ell\big(-2\deg(P)\big)\abs{P}^{1/3}\big(1-\abs{P}^{-2}\big)\right)\otimes\sum_{a_1b_1c_1d_1 = F_2}\frac{\abs{d_1}^{5/3}}{\abs{Pc_1}^{5/3}}\mu(b_1)\frac{\sigma(a_1)}{\abs{a_1b_1}}\\&\times\sum_\sumstack{r_1\mid a_1,\,\, r_{2,1}\mid b_1\\ \,\, r_6\mid d_1}\mu(r_1r_{2,1}r_6)\sigma(r_1)^{-1}\abs{r_{2,1}}^{-1}\prod_{P\mid b_1/r_{2,1}}\left(1+\abs{P}^{-1}\right)
\nu_1\big(\Tilde{S}_{d_1/(r_6)}\big)\nu_2\big(\Tilde{S}_{r_6}\big)\\&\times\sum_{ m\mid r_1r_{2,1}r_6 }\mu(m)\abs{m}^{2/3}\phi_\ell\big(n+2\deg(d)-2\deg(Pc)-4\deg(m)\big)
\\& = \phi_\ell(n)\otimes \bigotimes_{P\mid \Tilde{P}_ {(1^3)}}\abs{P}^{-2}\left(1-\abs{P}^{-1}\right)\left(\phi_\ell(0)-\abs{P}^{-2/3}\phi_\ell(-2\deg P)\right).
\end{split}
\end{equation*}
The factor $\abs{P}^{-2}\left(1-\abs{P}^{-1}\right)\left(\phi_\ell(0)-\abs{P}^{-2/3}\phi_\ell(-2\deg P)\right)$ is thus akin to a local factor at $P\mid \Tilde{P}_{(1^3)}$.

Similarly, we can find the local factor at a prime dividing $\Tilde{P}_{(1^21)}$. We then compute
\begin{equation*}
\begin{split}
    \sum_{F_1\mid \Tilde{P}_{(1^21)}}&\mu(F_1)\sum_\sumstack{\ell_1\ell_2 = F_1  }\sum_\sumstack{abcd=\ell_2} \frac{\abs{d_1}^{5/3}}{\abs{\ell_2c_1}^{5/3}} \mu(b)\frac{\sigma(a)}{\abs{a}\abs{b}}\sum_\sumstack{r_1\mid a,\,\, r_{2,1}\mid b\\ r_3\mid \ell_1,\,\, r_6\mid d}\mu(r_1r_{2,1}r_3r_6)\sigma(r_1)^{-1}\abs{r_{2,1}}^{-1}\prod_{P\mid b/r_{2,1}}\left(1+\abs{P}^{-1}\right)
    \\& \times \nu_1\big(\Tilde{S}_{d\ell_1/(r_3r_6)}\big)\nu_2\big(\Tilde{S}_{r_3r_6}\big)\sum_{ m\mid r_1r_{2,1}r_3r_6 }\mu(m)\abs{m}^{2/3}\phi_\ell\big(2\deg(d)-2\deg(\ell_2c)-4\deg(m)\big)
    \\& = \bigotimes_{P\mid \Tilde{P}_{(1^21)}}\left(1-\abs{P}^{-1}\right)\left(\phi_\ell(0)\abs{P}^{-1}\left(1-\abs{P}^{-1}\right)-\phi_\ell\big(-2\deg(P)\big)\abs{P}^{-5/3} + \phi_\ell\big(-4\deg(P)\big)\abs{P}^{-4/3}\right).
\end{split}
\end{equation*}

Finally, we turn to the unramified splitting types and study
\begin{equation*}
\begin{split}
    \sum_{F_u\mid \Tilde{P}_{u}}&\frac{\mu\left(\frac{\Tilde{P}_{u}}{F_u}\right)\abs{F_u}^2}{\abs{\Tilde{P}_{u}}^2} \bigotimes_{P\mid \Tilde{P}_{u}/F_u} \left(\phi_\ell(0) + \phi_\ell\big(-2\deg(P)\big)\abs{P}^{1/3}\big(1-\abs{P}^{-2}\big)\right)\otimes\sum_\sumstack{h_u\ell = F_u }\sum_\sumstack{fg\mid h_u}\mu(g)\frac{\sigma(fg)\abs{f}^{5/3}}{\abs{fg}\abs{h_u}^{10/3}}\\ \times &\sum_\sumstack{r_1\mid fg,\,\, r_3\mid \ell }\mu(r_1r_3)\sigma(r_1)^{-1}\nu_1\big(\Tilde{S}_{\ell/r_3}\big)\nu_2\big(\Tilde{S}_{r_3}\big)\sum_{ m\mid r_1r_3 }\mu(m)\abs{m}^{2/3}\phi_\ell\big(2\deg(f)-4\deg(h_u)-4\deg(m)\big)
    \\& = \bigotimes_{P\mid \Tilde{P}_u} \left(\phi_\ell(0) \big(\nu_1(S_P)-\nu_2(S_P)\big) + \phi_\ell(-4\deg P)\abs{P}^{2/3}\nu_2(S_P)\right).
\end{split}
\end{equation*}
Now one checks, using \cite[Theorem 11]{TTOrb}, for the splitting types $(111), (21)$ and $(3)$, that \newline$\left(\phi_\ell(0) \big(\nu_1(S_P)-\nu_2(S_P)\big) + \phi_\ell(-4\deg P)\abs{P}^{2/3}\nu_2(S_P)\right)$ is equal to
\begin{equation*}
\begin{split}
    &\frac{1}{6}\left(1-\abs{P}^{-1}\right)\left(1-2\abs{P}^{-1}\right)\phi_\ell(0) + \frac{1}{6}\left(1-\abs{P}^{-1}\right)\abs{P}^{-1/3}\left(2-\abs{P}^{-1}\right)\phi_\ell(-4\deg P),
    \\& \frac{1}{2}\left(1-\abs{P}^{-1}\right)\phi_\ell(0) - \frac{1}{2}\left(1-\abs{P}^{-1}\right)\abs{P}^{-4/3}\phi_\ell(-4\deg P), \text{ and}
    \\& \frac{1}{3} \left(1-\abs{P}^{-2}\right)\phi_\ell(0) - \abs{P}^{-1/3}\left(1-\abs{P}^{-2}\right)\phi_\ell(-4\deg P) \text{ respectively.}
\end{split}
\end{equation*}

Removing the reducible forms as before, we have proven the following theorem.
\begin{theorem}
    The number of $S_3$-cubic fields whose $R$-semilocal discriminant is $X = q^\ell$, for some admissable $\ell$, where the local specification at $P_\infty$ is $\sigma$, and where $P_1,...,P_n \in R$ splits according to $S_{P_1},...,S_{P_n}$ is equal to
    \begin{equation*}
    \begin{split}
        \frac{q^2-1}{q^2\#\mathrm{Aut}(\sigma)}&X\prod_{P\mid P_1...P_n}c(P)x_P -\frac{1}{q\#\mathrm{Aut}(\sigma)}X^{5/6}\prod_{P\mid P_1...P_n} \left(1-\abs{P}^{-1}\right)\bigotimes_{P\nmid P_1...P_n} \left(\phi_\ell(0)-\phi_\ell\big(-2\deg(P)\abs{P}^{-5/3}\big)\right)     
        \\& \bigotimes_{P\mid P_1...P_n} d_P(\phi_\ell)+ \mathcal{O}\left(X^{2/3+\epsilon}\abs{P_1...P_n}^{2/3}\right),
    \end{split}
    \end{equation*}
    where $c_P = 1/6,\, 1/2,\, 1/3,\, 1/\abs{P}$ or $1/\abs{P}^2$ depending on if $S_P$ is of type $(111), (21), (3), (1^21)$, or $(1^3)$ respectively. Furthermore, $x_P = \big(1+\abs{P}^{-1}+\abs{P}^{-2}\big)^{-1}$ and $d_P(\phi_\ell)$ equals
    \begin{equation*}
\begin{split}
    &\frac{1}{6}\left(1-2\abs{P}^{-1}\right)\phi_\ell(0) + \frac{1}{6}\abs{P}^{-1/3}\left(2-\abs{P}^{-1}\right)\phi_\ell(-4\deg P),
    \,\, \frac{1}{2}\phi_\ell(0) - \frac{1}{2}\abs{P}^{-4/3}\phi_\ell(-4\deg P),
    \\& \frac{1}{3} \left(1+\abs{P}^{-1}\right)\phi_\ell(0) - \abs{P}^{-1/3}\left(1+\abs{P}^{-1}\right)\phi_\ell(-4\deg P),
    \\& \abs{P}^{-1}\left(1-\abs{P}^{-1}\right)\phi_\ell(0)-\abs{P}^{-5/3}\phi_\ell\big(-2\deg(P)\big) + \abs{P}^{-4/3}\phi_\ell\big(-4\deg(P)\big), \text{ or}
    \\& \abs{P}^{-2}\left(\phi_\ell(0) - \abs{P}^{-2/3}\phi_\ell\big(-2\deg P\big)\right),
\end{split}
\end{equation*}
depending on if the splitting type is $(111), (21), (3), (1^21)$ or $(1^3)$.
\end{theorem}

We remark that if one wants to remove the condition on $P_\infty$, and count fields with global discriminant $Y=q^{M}$, then one may simply sum the above over $\sigma$ with $X$ chosen appropriately. We obtain the following result.
\begin{theorem}\label{allfieldSplitCondThm}
       The number of $S_3$-cubic fields whose global discriminant is $Y = q^{M}$, with $2\mid M$, and where $P_1,...,P_n \in R$ splits according to $S_{P_1},...,S_{P_n}$ is equal to
    \begin{equation*}
    \begin{split}
        \frac{(q^2-1)(q^3-1)}{q^4(q-1)}&Y\prod_{P\mid P_1...P_n}c(P)x_P -\frac{1}{q}Y^{5/6}\prod_{P\mid P_1...P_n} \left(1-\abs{P}^{-1}\right)\bigotimes_{P\nmid P_1...P_n} \left(\phi^*_M(0)-\phi^*_M\big(-2\deg(P)\abs{P}^{-5/3}\big)\right)
        \\& \bigotimes_{P\mid P_1...P_n} d_P(\phi^*_M)+ \mathcal{O}\left(X^{2/3+\epsilon}\abs{P_1...P_n}^{2/3}\right),
    \end{split}
    \end{equation*}
    where $\phi^*_M(n)$ is defined through $C_2^*(M+n)$, with 
    \begin{equation*}
        C_2^*(k) = \sum_{\sigma} \frac{q^{-5\gamma(\sigma)/6}}{\#\mathrm{Aut}(\sigma)}C_2^\sigma\big(k-\gamma(\sigma)\big),
    \end{equation*}
    where $\gamma(\sigma) = 0$ if $\sigma$ is unramified, $1$ if $\sigma$ is partially ramified and $2$ else. Explicitly, $C_2^*(k)$ is given by the following table:
\begin{center}
\begin{tabular}{ |c|c|c| } 
 \hline
  $k \equiv^3 0$ & $k \equiv^3 1$ & $k \equiv^3 2$ \\ 
 \hline
  $q+3+q^{-1}$ & $2q^{2/3}+2q^{-1/3}+q^{-4/3}$ & $q^{4/3}+2q^{1/3}+2q^{-2/3}$ \\ 
 \hline
\end{tabular}
\end{center}
Recall also that $\otimes$ is defined as a formal product acting on the symbols $\phi_M^*$ such that $\phi_M^*(n) \otimes \phi_M^*(n) = \phi_M^*(m+n)$. Furthermore, we associate the symbol $\phi_M^*(n)$ with the real number $C_2^*(M+n)$.
\end{theorem}

We end this section by proving the inequality \eqref{splitpreponelevcor} that we used when studying the one-level density. We begin by noting that
    \begin{equation*}
        \bigotimes_{Q\neq P} \left(\phi^*_M(0)-\phi^*_M\big(-2\deg(Q)\big)\abs{Q}^{-5/3}\right) = \bigotimes_{Q} \left(\phi^*_M(0)-\phi^*_M\big(-2\deg(Q)\big)\abs{Q}^{-5/3}\right)  \otimes \sum_{n=0}^\infty \phi^*_M\big(-2n\deg(P)\big)\abs{P}^{-5n/3},
    \end{equation*}
    as one checks that
    \begin{equation*}
        \left(\phi^*_M(0)-\phi^*_M\big(-2\deg(P)\big)\abs{P}^{-5/3}\right)  \otimes \sum_{n=0}^\infty \phi^*_M\big(-2n\deg(P)\big)\abs{P}^{-5n/3} = \phi^*_M(0)
    \end{equation*}
    from the definition of $\otimes$. Now, from the calculations concerning the secondary term for all cubic fields, we know that
    \begin{equation*}
        \bigotimes_{Q} \left(\phi^*_M(0)-\phi^*_M\big(-2\deg(Q)\big)\abs{Q}^{-5/3}\right) = \phi^*_M(0) - \phi^*_M(-2)q^{-2/3}.
    \end{equation*}
    We may, therefore, conclude that
    \begin{equation*}
        \bigotimes_{Q\neq P} \left(\phi^*_M(0)-\phi^*_M\big(-2\deg(Q)\big)\abs{Q}^{-5/3}\right) = \phi^*_M(0) - \phi^*_M(-2)q^{-2/3} + \mathcal{O}\left(\abs{P}^{-5/3}\right).
    \end{equation*}

    Furthermore,
    \begin{equation*}
        d_P(\phi^*_M) = \mathcal{O}\left(\abs{P}^{-1}\right) +\begin{cases}
            \frac{1}{6}\phi^*_M(0) + \frac{1}{3}\phi^*_M(-4\deg P)\abs{P}^{-1/3}, &\text{ for $P$ of type $(111)$,}
            \\ \frac{1}{2}\phi^*_M(0), &\text{ for $P$ of type $(21)$,}
            \\ \frac{1}{3}\phi^*_M(0)-\frac{1}{3}\phi^*_M(-4\deg P)\abs{P}^{-1/3},&\text{ for $P$ of type $(3)$,}
            \\ 0, &\text{ for $P$ of type $(1^21)$,}
            \\ 0,&\text{ for $P$ of type $(1^3)$.}
        \end{cases}
    \end{equation*}
    Hence, we conclude that $2C_{2,P,(111)}-C_{2,P,(3)} + C_{2,P,(1^21)}$ is equal to
    \begin{equation}\label{onelevtobound}
        -\frac{1}{q}\frac{\abs{P}^{-1/3}}{3}\left(C_2^*(M-4\deg P)-C_2^*(M-2-4\deg P)q^{-2/3}\right) + \mathcal{O}\big(\abs{P}^{-1}\big).
    \end{equation}

    At this point, a calculation shows that no matter the value of $M$ or $\deg P$, we have that
    \begin{equation*}
        C_2^*(M-4\deg P)-C_2^*(M-2-4\deg P)q^{-2/3} > 0.
    \end{equation*}
    Indeed, one checks the three different cases by using the table from Theorem \ref{allfieldSplitCondThm}. We conclude that there is some constant $D > 0$ such that \eqref{onelevtobound} is 
    \begin{equation*}
        \leq -D\abs{P}^{-1/3} + \mathcal{O}\big(\abs{P}^{-1}\big),
    \end{equation*}
    as desired.


\addcontentsline{toc}{chapter}{Bibliography}
\begin{thebibliography}{}
\bibitem[BBP]{BBP}K. Belabas, M. Bhargava, and C. Pomerance, “Error estimates for the Davenport-Heilbronn theorems”, \textit{Duke Math. J.}, vol. 153, no. 1, pp. 173–210, 2010.
\bibitem[BST]{BST}M. Bhargava, A. Shankar, and J. Tsimerman, “On the Davenport–Heilbronn theorems and second order terms,” \textit{Invent. Math.}, vol. 193, no. 2, pp. 439–499, 2013.
\bibitem[BSW]{BSW} M. Bhargava, A. Shankar, and X. Wang, “Geometry-of-numbers methods over global fields I: Prehomogeneous vector spaces,” Preprint, 2015, arXiv:1512.03035.
 \bibitem[BTT]{BTT}M. Bhargava, T. Taniguchi, and F. Thorne, “Improved error estimates for the Davenport–Heilbronn theorems,” \textit{Math. Ann}., vol. 389, no. 4, pp. 3471–3512, 2024.
\bibitem[C]{Chang} K. Chang, “Hurwitz spaces, Nichols algebras, and Igusa zeta functions,” Preprint, 2023, arXiv:2306.10446.
\bibitem[CFLS]{CFLS}P. J. Cho, D. Fiorilli, Y. Lee, and A. Södergren, “Omega results for cubic field counts via lower-order terms in the one-level density,” \textit{Forum Math. Sigma}, vol. 10, no. e80, 33 pp., 2022.
\bibitem[CK]{CK}P. J. Cho and H. H. Kim, "Low lying zeros of Artin $L$-functions", \textit{Math. Z.}, vol. 279, no. 3–4, pp. 669–688, 2015.
\bibitem[DH]{DH}H. Davenport and H. Heilbronn, “On the density of discriminants of cubic fields. II,”\textit{ Proc. R.
Soc. Lond. Ser. A, Math. Phys. Sci.}, vol. 322, no. 1551, pp. 405–420, 1971.
\bibitem[DW]{DW} B. Datskovsky and D. Wright, “Density of discriminants of cubic extensions,” \textit{J. Reine Angew. Math.}, vol. 386, pp. 116–138, 1988.
\bibitem[F]{Frankenhuijsen} M. van Frankenhuijsen, The Riemann hypothesis for function fields, \textit{London Mathematics Society
Student Texts} vol. 80, Cambridge University Press, 2014.
 \bibitem[H1]{Hough}R. Hough, “The shape of cubic fields,” \textit{Res. Math. Sci}., vol. 6, no. 23, 25 pp., 2019.
 \bibitem[H2]{Hough2}R. Hough, “Equidistribution of bounded torsion CM points,” \textit{J. Anal. Math}., vol. 138, no. 2, pp. 765–797, 2019.

 \bibitem[HR]{Hughes-Rudnick}C. P. Hughes and Z. Rudnick, “Linear statistics of low-lying zeros of L-functions,” \textit{Q. J. Math}., vol. 54, no. 3, pp. 309–333, 2003.
\bibitem[ILS]{ILS}H. Iwaniec, W. Luo, and P. Sarnak, “Low lying zeros of families of L-functions,”\textit{ Publ. Math. Inst. Hautes Etudes Sci}., vol. 91, pp. 55–131, 2000.
\bibitem[K]{Kural} M. Kural, "The geometry of secondary terms in arithmetic statistics," Preprint, 2025, arXiv: 2504.17909.
\bibitem[KS]{Katz-Sarnak}N. M. Katz and P. Sarnak, “Zeroes of zeta functions and symmetry,” \textit{Bull. Amer. Math. Soc.}, vol. 36, no. 1, pp. 1–26, 1999.
\bibitem[M]{Mori} S. Mori, "Orbital Gauss sums associated with the space of binary cubic forms over a finite field", \textit{RIMS
Kôkyûroku}, vol. 1715, pp. 32-36, 2010.
\bibitem[MS]{MeisSod}P. Meisner and A. Södergren, “Low-lying zeros in families of elliptic curve L-functions over function fields,” Finite Fields Appl., vol. 84, Paper no. 102096, 46pp, 2022.
\bibitem[N]{Neukirch}J. Neukirch, \textit{Algebraic number theory}. Berlin, Germany: Springer-Verlag, 1999.
 \bibitem[ÖS]{Ozluk-Snyder}A. E. Özlük, and C. Snyder. "On the distribution of the nontrivial zeros of quadratic L-functions close to the real axis." \textit{Acta Arith}., vol. 91, no. 3, pp. 209-228, 1999.

\bibitem[Rb]{Roberts}D. Roberts, “Density of cubic field discriminants,” \textit{Math. Comp}., vol. 70, no. 236, pp. 1699–1705, 2001.
 \bibitem[Ro]{Rosen}M. Rosen, \textit{Number theory in function fields}. New York, NY: Springer-Verlag, 2002. 
\bibitem[Rk]{Rudnick}Z. Rudnick, “Traces of high powers of the Frobenius class in the hyperelliptic ensemble,” \textit{Acta Arith}., vol. 143, no. 1, pp. 81-99, 2010.
\bibitem[S]{Serre} J.P. Serre: \textit{Trees}. Berlin, Heidelberg, Germany: Springer-Verlag, 1980.
\bibitem[SST1]{SST1} A. Shankar, A. Södergren, and N. Templier, “Sato-Tate equidistribution of certain families of Artin L-functions,” \textit{Forum Math. Sigma}, vol. 7, no. e23, 62 pp., 2019.
\bibitem[SST2]{SST} A. Shankar, A. Södergren, and N. Templier, “Central values of zeta functions of non-Galois cubic fields,” Preprint, 2021, arXiv: 2107.10900.


\bibitem[TT1]{TT} T. Taniguchi and F. Thorne, “Secondary terms in counting functions for cubic fields,” \textit{Duke Math. J.}, vol. 162, no. 13, pp. 2451–2508, 2013.
\bibitem[TT2]{TTOrb} T. Taniguchi and F. Thorne, “Orbital exponential sums for prehomogeneous vector spaces,” Amer. J. Math., vol. 142, no. 1, pp. 177–213, 2020.
\bibitem[W]{WrightC3}D. J. Wright, “Distribution of Discriminants of Abelian Extensions,” \textit{Proc. Lond. Math. Soc.}, vol. 58, no. 1, pp. 17–50, 1989.
 \bibitem[Ya]{Yang}A. Yang, \textit{Distribution problems associated to zeta functions and invariant theory}. Ph.D. dissertation, Princeton University, 2009.
 \bibitem[Yo]{Young}M. P. Young, “Low-lying zeros of families of elliptic curves,” \textit{J. Amer. Math. Soc}., vol. 19, no. 1, pp. 205–250, 2006.
\bibitem[Z]{Zhao}Y. Zhao. \textit{On sieve methods for varieties over finite fields.} Ph.D. dissertation, University
of Wisconsin-Madison, 2013.

\end{thebibliography}
\end{document}